\documentclass[10pt,a4paper]{article}
\usepackage[utf8]{inputenc}
\usepackage[american]{babel}
\usepackage{amsmath,amsfonts,amssymb,amsthm,mathtools}
\usepackage{bm}
\usepackage{enumitem}
\usepackage{booktabs}
\usepackage[left=2.2cm,right=2.2cm,top=2.2cm,bottom=2.2cm]{geometry}
\usepackage[dvipsnames]{xcolor}
\usepackage[colorlinks=true,linkcolor=blue,citecolor=blue,urlcolor=blue]{hyperref}

\newtheorem{thm}{Theorem}[section]
\newtheorem{lem}[thm]{Lemma}
\newtheorem{prop}[thm]{Proposition}
\newtheorem{cor}[thm]{Corollary}
\theoremstyle{definition}
\newtheorem{defin}[thm]{Definition}
\theoremstyle{remark}
\newtheorem{remark}[thm]{Remark}
\numberwithin{equation}{section}

\newcommand{\R}{\mathbb{R}}
\newcommand{\N}{\mathbb{N}}
\newcommand{\Sph}{\mathbb{S}}
\newcommand{\eps}{\varepsilon}

\renewcommand{\(}{\left(}
\renewcommand{\)}{\right)}
\newcommand{\ird}[1]{\int_{\R^d}{#1}\,dx}
\newcommand{\irtwo}[1]{\int_{\R^2}{#1}\,dx}
\newcommand{\nrm}[2]{\left\|{#1}\right\|_{#2}}
\newcommand{\scalar}[2]{\left\langle{#1},{#2}\right\rangle}
\newcommand{\be}[1]{\begin{equation}\label{#1}}
\newcommand{\ee}{\end{equation}}
\newcommand{\cA}{\mathcal A}
\newcommand{\cL}{\mathcal L}

\newcommand{\Cs}{\mathcal C_\star}
\newcommand{\bmm}{\mathbf m}
\newcommand{\er}{\hat{\mathbf e}_r}
\newcommand{\nS}{\nabla_{\!\Sph}}
\newcommand{\ind}{\mathbf 1}

\title{Sharp spectral gap and large time asymptotics\\ for the Poisson--Nernst--Planck system with confinement}
\author{Xingyu Li}
\date{}

\begin{document}
\maketitle

\begin{abstract}
We study the Poisson--Nernst--Planck (Debye--H\"uckel) system for a single species of charged particles in $\R^d$, $d=2$ or $3$, with an external confinement potential. The particles repel each other through the Coulomb interaction. For a class of potentials with at least linear growth at infinity, solutions converge to the unique stationary state at an exponential rate in a weighted $\mathrm L^2$ space. The rate is given by the spectral gap of the linearized operator, which is symmetric for a scalar product adapted to the free energy. Our main result is that, for the harmonic potential $\phi(x)=\frac\mu2\,|x|^2$, this spectral gap is exactly $\mu$, for any mass and in both dimensions. The corresponding eigenspace is spanned by the translation modes. Unlike the attractive (Keller--Segel) case, the repulsive interaction destroys the monotonicity of the spectrum with respect to the angular momentum. We overcome this difficulty with a duality criterion and explicit \emph{chiral fluxes} built from vector spherical harmonics. These reduce the non-local spectral problem to one-dimensional weighted inequalities, which we prove by a ground state transform. The center of mass decays exactly like $e^{-\mu t}$. After removing this motion, the convergence rate improves to a second spectral gap $\Lambda_2>\mu$, for which we prove an explicit lower bound in terms of the central density of the stationary state. As consequences, we obtain rates of convergence in $\mathrm L^p$ norms, and improved intermediate asymptotics for the unconfined system in $\R^2$, after recentering at the center of mass.
\end{abstract}

\tableofcontents

\section{Introduction}\label{NeSec:Intro}

\subsection{The model}

The Nernst--Planck equations~\cite{nernst1889elektromotorische,planck1890ueber} describe charged particles which diffuse under the action of a drift caused by an electrostatic potential. With a self-consistent mean-field Poisson coupling, and for a single species of particles in $\R^2$, the system reads
\begin{equation}
\label{Eq:(5.1)}
\left\{
\begin{array}{l}
\frac{\partial u}{\partial t}= \Delta u + \nabla \cdot (u\,\nabla v)\,,\\[2pt]
v=G_{2}*u\,, \\[2pt]
u(0,x)=n_{0}(x)\ge 0\,,
\end{array} \right.\quad x\in \mathbb{R}^2\,,\;t>0\,,
\end{equation}
where $G_2(x)=-\frac1{2\,\pi}\,\log|x|$ is the Green function of $-\Delta$ in $\mathbb{R}^2$. This \emph{Poisson--Nernst--Planck system} is also called the \emph{Debye--H\"uckel system}~\cite{debye1923theorie}. Up to the sign of the mean-field term, it is the parabolic-elliptic Keller--Segel model. The Keller--Segel model is attractive, while~\eqref{Eq:(5.1)} is repulsive, and this difference plays a central role below.

In the whole space, particles repel each other and a \emph{runaway} phenomenon occurs: the solution locally vanishes while the mass escapes at infinity. This can be prevented by a container, or by an external \emph{confinement} potential $\phi$. We consider the Poisson--Nernst--Planck system with confinement in $\mathbb{R}^d$, $d=2$ or $d=3$,
\begin{equation}
\label{Eq:(5.2)}
\left\{
\begin{array}{l}
\frac{\partial n}{\partial t}= \Delta n + \nabla \cdot (n\,\nabla c) + \nabla \cdot(n\,\nabla\phi)\,,\\[2pt]
c=G_{d}*n\,, \\[2pt]
n(0,x)=n_{0}(x)\ge 0\,, \quad \ird{n_0}=M>0\,,
\end{array} \right.\quad x\in \mathbb{R}^d\,,\;t>0\,,
\end{equation}
where $G_d$ is the Green function of $-\Delta$ in $\mathbb{R}^d$, namely
\[
G_2(x)=-\frac1{2\,\pi}\,\log|x|\quad\mbox{for any}\quad x\in\mathbb R^2\quad\mbox{and}\quad G_3(x)=\frac 1{4\,\pi\,|x|}\quad\mbox{for any}\quad x\in\mathbb R^3\,,
\]
so that $-\Delta c=n$. The mass $\ird{n(t,x)}=M$ is conserved, and $M>0$ is fixed from now on.

The two problems are related. Let $d=2$ and $\phi(x)=\frac12\,|x|^2$, and consider the change of variables
\begin{equation}
\label{ch2}
u(t,x)=R^{-2}\,n(\tau,\xi)\,,\quad v(t,x)=c(\tau,\xi)\,,\quad
\xi=\frac xR\,,\quad\tau=\log R\,,\quad R=R(t):=\sqrt{1+2\,t}\,.
\end{equation}
Since $\frac{d\tau}{dt}=R^{-2}$, a direct computation shows that $(n,c)$ solves~\eqref{Eq:(5.2)} if and only if $(u,v)$ solves~\eqref{Eq:(5.1)}, where $v$ is defined up to an additive function of $t$. This equivalence requires the normalization $\mu=1$: with $\phi=\frac\mu2\,|\xi|^2$ and $R=\sqrt{1+2\,\mu\,t}$, the same change of variables gives $\mu\,\partial_\tau n=\Delta n+\nabla\cdot(n\,\nabla c)+\mu\,\nabla\cdot(\xi\,n)$. Hence rates of convergence for~\eqref{Eq:(5.2)} translate into \emph{intermediate asymptotics} for~\eqref{Eq:(5.1)}, that is, into rates of convergence towards self-similar solutions.

\emph{Assumptions on the potential.} We assume that $\phi\in\mathrm W_{\rm{loc}}^{2,\infty}(\R^d)$ with $D^2\phi\in\mathrm L^\infty(\R^d)$, and that
\be{Hyp-C1}\tag{C1}
\liminf_{|x|\to+\infty}\frac{\phi(x)}{\log|x|}>d\,.
\ee
We also use the Poincar\'e inequality for the measure $e^{-\phi}\,dx$:
\begin{equation}\label{Hyp-C2}\tag{C2}
\ird{|\nabla u|^2\,e^{-\phi}}\ge\Lambda_\phi\ird{|u|^2\,e^{-\phi}}\quad
\forall\,u\in\mathrm H^1(\R^d;e^{-\phi}\,dx)\quad\mbox{such that}\quad\ird{u\,e^{-\phi}}=0\,,
\end{equation}
for some $\Lambda_\phi>0$. By Persson's lemma, a sufficient condition is
\be{Hyp-C3}\tag{C3}
\sigma_\phi:=\lim_{r\to+\infty}\mathop{\mathrm{infess}}_{x\in B_r^c}\(\frac14\,|\nabla \phi|^2-\frac12\,\Delta\phi\)>0\quad\mbox{and}\quad\lim_{r\to+\infty}\mathop{\mathrm{infess}}_{x\in B_r^c}|\nabla \phi|>0\,,
\ee
see~\cite{addala} and references therein. In dimension $d=2$, we also require the stronger growth condition
\be{Hyp-C4}\tag{C4}
\liminf_{|x|\to+\infty}\frac{\phi(x)}{\log|x|}>4+\frac M{2\,\pi}\quad\mbox{if}\quad d=2\,.
\ee
Condition~\eqref{Hyp-C4} plays two roles. By Lemma~\ref{Lem:LowerBound}, the free energy defined below is bounded from below as soon as the above $\liminf$ is larger than $2+\frac M{2\pi}$; see also~\cite{DolbeaultLi} for the generalized logarithmic Hardy--Littlewood--Sobolev inequality. The additional growth required in~\eqref{Hyp-C4} guarantees that the stationary state has a finite second moment, which is used to control the electrostatic energy of the perturbations (Lemma~\ref{Lem:GradGL2}).
We say that \emph{$\phi$ satisfies Assumption \hypertarget{C}{$(\mathrm C)$}} if $\phi\in\mathrm W_{\rm{loc}}^{2,\infty}(\R^d)$, $D^2\phi\in\mathrm L^\infty(\R^d)$, and~\eqref{Hyp-C1},~\eqref{Hyp-C3} and~\eqref{Hyp-C4} hold. The harmonic potential $\phi(x)=\frac\mu2\,|x|^2$ satisfies Assumption~\hyperlink{C}{$(\mathrm C)$} for any $\mu>0$ and $M>0$.
\begin{remark}\label{Rem:Linear}
Condition~\eqref{Hyp-C3} is a condition of at least linear growth. Indeed, let $c_0>0$ and $r_0>0$ be such that $|\nabla\phi|\ge c_0$ on $\R^d\setminus B_{r_0}$. For $|x|>r_0$, let $s\mapsto x(s)$ be the solution of $x'=-\nabla\phi(x)/|\nabla\phi(x)|$ with $x(0)=x$, which is well defined as long as $|x(s)|>r_0$, because $\nabla\phi$ is locally Lipschitz. Then $\frac d{ds}\phi(x(s))=-|\nabla\phi(x(s))|\le-c_0$. Since $\phi$ is bounded from below by~\eqref{Hyp-C1} and $|x'|=1$, the curve reaches $\partial B_{r_0}$ for some finite $s_*\ge|x|-r_0$, so that
\be{LinearGrowth}
\phi(x)\ge\min_{\overline{B_{r_0}}}\phi+c_0\,\big(|x|-r_0\big)\quad\forall\,x\in\R^d\setminus B_{r_0}\,.
\ee
As a consequence, under Assumption~\hyperlink{C}{$(\mathrm C)$}, any function $f\ge0$ such that $f\,\phi\in\mathrm L^1(\R^d)$ has a finite first moment, $e^{-\phi}$ decays exponentially, and~\eqref{Hyp-C4} follows from~\eqref{Hyp-C1} and~\eqref{Hyp-C3}. We nevertheless refer to~\eqref{Hyp-C1} and~\eqref{Hyp-C4} in the proofs which only use a logarithmic growth, in order to make clear what is needed if~\eqref{Hyp-C3} is replaced by the Poincar\'e inequality~\eqref{Hyp-C2} as an assumption.
\end{remark}

\emph{Stationary state and free energy.} The stationary solutions of~\eqref{Eq:(5.2)} solve the \emph{Poisson--Boltzmann equation}~\cite{arnold2001convex,MR1777308}
\begin{equation}
\label{st}
-\Delta c_\infty=n_\infty=M\,\frac{e^{-c_\infty-\phi}}{\ird{e^{-c_\infty-\phi}}}\,.
\end{equation}
Under Assumption~\hyperlink{C}{$(\mathrm C)$}, the solution of~\eqref{st} is unique, and it is the minimizer of the \emph{free energy}
\begin{equation}
\label{freen}
\mathcal F[n]:=\ird{n\,\log n}+\ird{n\,\phi} + \frac{1}{2}\ird{n\,(-\Delta)^{-1}n}
\end{equation}
under the mass constraint $\ird n=M$ (see Section~\ref{NeSec:mi}). Along the flow of~\eqref{Eq:(5.2)}, the free energy is non-increasing and
\begin{equation}
\label{lypa}
\frac d{dt}\mathcal F[n(t,\cdot)]=-\mathcal I[n(t,\cdot)]\,,\qquad\mathcal I[n]:=\ird{n\,\big|\nabla{(\log n+c+\phi)}\big|^2}\,.
\end{equation}

\subsection{Main results}

Our first result is an exponential rate of convergence for general confinement potentials. The rate is the spectral gap $\Cs$ of the linearized operator, defined as the best constant in the coercivity inequality $Q_2\ge\Cs\,Q_1$ between the quadratic forms $Q_1$ and $Q_2$ obtained by expanding $\mathcal F$ and $\mathcal I$ at second order around $n_\infty$ (Section~\ref{Sec:Coercivity-NP}).
\begin{thm}\label{Thm4.1} Let $d=2$ or $3$ and consider a potential $\phi$ satisfying \hyperlink{C}{$(\mathrm C)$}. Assume that $n$ solves~\eqref{Eq:(5.2)} with initial datum $n(0,\cdot)= n_{0}\in \mathrm L^2_{+}(n_{\infty}^{-1}dx)$, $\ird{n_{0}}=M$, and $\mathcal F[n_0]<\infty$. Let $\Cs>0$ be the optimal constant in the coercivity inequality $Q_2\ge\Cs\,Q_1$ of Proposition~\ref{Prop:Coercivity}. Then there exists a positive constant $C$ such that
\[
\ird{\big|n(t,\cdot)-n_{\infty}\big|^2\,n_{\infty}^{-1}}\le C\,e^{-\,2\,\Cs\,t}\quad\forall\,t\ge0\,.
\]
\end{thm}
For a general potential, $\Cs$ is only known to be bounded from below by the Poincar\'e constant of $n_\infty$. Our main result is the exact value of $\Cs$ for a harmonic potential.
\begin{thm}\label{Thm:Sharp} Let $d=2$ or $3$, $\mu>0$, $M>0$ and $\phi(x)=\frac\mu2\,|x|^2$. Then $\Cs=\mu$. The equality case $Q_2[f]=\mu\,Q_1[f]$ is achieved if and only if $f$ is a linear combination of the translation modes $\partial_{x_i}n_\infty/n_\infty$, $i=1,\ldots,d$. As a consequence, the rate in Theorem~\ref{Thm4.1} is $e^{-2\,\mu\,t}$, and it is optimal.
\end{thm}
Three features of Theorem~\ref{Thm:Sharp} are worth emphasizing. First, the spectral gap does not depend on the mass $M$, although the stationary state $n_\infty$ does, and it flattens into a uniformly charged ball as $M\to+\infty$. In other words, the repulsive interaction never slows down the linearized dynamics below the rate $\mu$ given by the confinement alone. Second, the result holds in dimension $d=3$ as well as in dimension $d=2$. Third, the extremal functions are the translation modes, which correspond to an exact symmetry of the nonlinear problem. Indeed, the center of mass
\[
X(t):=\frac1M\int_{\R^d}x\,n(t,x)\,dx
\]
of a solution of~\eqref{Eq:(5.2)} with harmonic potential satisfies $X(t)=X_0\,e^{-\mu t}$, because the Coulomb self-interaction exerts no net force. After this explicit motion is removed, the rate improves.
\begin{thm}\label{Thm:Improved} Let $d=2$ or $3$ and $\phi(x)=\frac\mu2\,|x|^2$. Let $n_0\in\mathrm L^1_+(\R^d)$ be such that $\ird{n_0}=M$ and $\ird{|x|\,n_0}<\infty$, let $X_0:=M^{-1}\ird{x\,n_0}$, and assume that $\tilde n_0:=n_0(\cdot+X_0)\in\mathrm L^2(n_\infty^{-1}dx)$. Then~\eqref{Eq:(5.2)} has a solution $n$ with initial datum $n_0$, which is unique among the nonnegative functions in $C\big([0,\infty);\mathrm L^1\cap\mathrm L^2(\R^d)\big)$ satisfying Properties~(i) and~(ii) of Definition~\ref{Def:Solution}, and there is a constant $C>0$ such that
\[
\nrm{n(t,\cdot)-n_\infty\big(\cdot-X_0\,e^{-\mu t}\big)}{\mathrm L^1(\R^d)}\le C\,e^{-\Lambda_2\,t}\quad\forall\,t\ge0\,,
\]
where $\Lambda_2>\mu$ is the second spectral gap of Theorem~\ref{Thm:SecondGap}, which satisfies
\[
\Lambda_2\ge\mu+\frac{\delta_0}d\quad\mbox{if}\quad d=2\mbox{ or }d=3\,,\qquad\delta_0:=\mu-\frac{n_\infty(0)}d>0\,.
\]
\end{thm}
For uniformly convex potentials, we also prove that the free energy decays exponentially for $t\ge T$, with an explicit but non-optimal rate, by combining a logarithmic Sobolev inequality for a time-dependent Gibbs state with the Csisz\'ar--Kullback inequality (Proposition~\ref{Prop:FreeEnergyDecay}). For the harmonic potential, we also obtain the corresponding rates in $\mathrm L^p$ norms (Proposition~\ref{cormore}), a first order expansion $n(t)=n_\infty-e^{-\mu t}\,X_0\cdot\nabla n_\infty+O(e^{-\Lambda_2^*t})$ (Corollary~\ref{Cor:Expansion}), and, through~\eqref{ch2}, intermediate asymptotics for~\eqref{Eq:(5.1)}. After recentering at the (conserved) center of mass, these intermediate asymptotics improve from $(1+2\,t)^{-1/2}$ to $(1+2\,t)^{-\Lambda_2/2}$ in $\mathrm L^1(\R^2)$ (Theorem~\ref{Thm:Inter}). Numerical computations (Remark~\ref{Rem:Numerics}) indicate that the second spectral gap is achieved by the angular modes of degree two. They also indicate that $\Lambda_2\to\mu$ as $M\to+\infty$ if $d=2$, while the gap in the components of degree $\ell\ge2$ is bounded from below by $3\,\ell\,\mu/(2\,\ell+1)$ uniformly in $M$ if $d=3$.

\subsection{Strategy of proof}

The general part of the paper (Theorem~\ref{Thm4.1}) follows the strategy developed for the Keller--Segel model in~\cite{MR2226917,blanchet2010asymptotic,MR3196188,CDcras}. We linearize~\eqref{Eq:(5.2)} around $n_\infty$ in a weighted space whose scalar product is given by the Hessian of the free energy. In this space, the linearized operator is symmetric, and the nonlinear terms are controlled by the spectral gap.

The new difficulty is the computation of the spectral gap. In~\cite{MR3196188}, the sharp spectral gap of the linearized Keller--Segel operator is obtained by a decomposition in spherical harmonics. For the attractive sign, the non-local term makes the Rayleigh quotient of each angular component monotone with respect to the angular momentum $\ell$. It is then enough to study the radial component and the component $\ell=1$. For the repulsive sign, this monotonicity is lost. There is no a priori reason why the components $\ell\ge2$ should have a larger gap, and numerically their gap approaches $\mu$ as $M\to+\infty$ when $d=2$. A displacement convexity argument in the spirit of~\cite{MR2053570} does not apply either, because the Coulomb interaction energy is not displacement convex.

Our approach is based on a duality criterion (Lemma~\ref{Lem:Flux}): if a perturbation $h=f\,n_\infty$ can be written as $h=-\nabla\cdot\mathbf m$ with $\mu\ird{|\mathbf m|^2/n_\infty}\le Q_1[f]$, then $Q_2[f]\ge\mu\,Q_1[f]$. For each spherical harmonic of degree $\ell\ge1$, we construct an explicit \emph{chiral flux} $\mathbf m_\ell$. Its radial and tangential components are proportional, with a ratio determined by the decaying harmonic function of degree $\ell$. With this flux, the condition becomes a \emph{local}, one-dimensional weighted inequality. A ground state transform with respect to $n_\infty$ turns this inequality into an identity with an explicit nonnegative remainder (Proposition~\ref{Prop:GS}). The remainder vanishes for $\ell=1$ exactly on the translation modes, and it is positive for $\ell\ge2$. Its positivity only relies on the monotonicity of the radial profile $n_\infty$ (Lemma~\ref{Lem:Monotone}). The radial component is treated with the cumulated mass and a supersolution built from the dilation mode. In dimension $d=3$, the remainder contains a term of order $\mu/\ell$ which does not depend on $M$. This is the origin of the mass-independent bound $3\,\ell\,\mu/(2\,\ell+1)$.

\subsection{Related results}

The long time behaviour of drift-diffusion--Poisson systems in bounded domains is classical. Exponential convergence to the stationary state was proved by entropy methods in~\cite{biler1994debye,MR1777308,arnold2000large}, without sharp rates. In the whole space $\R^2$, the large time behaviour of~\eqref{Eq:(5.1)} for small data and the existence of self-similar solutions were studied in~\cite{biler1992existence,biler1995cauchy,herczak2008existence}. For the attractive Keller--Segel model in $\R^2$, functional inequalities, spectral gaps and sharp asymptotic rates in self-similar variables were obtained in~\cite{MR2226917,blanchet2010asymptotic,CDcras,MR3196188}. The present paper adapts this framework to the repulsive case and computes the sharp constant.

Gradient flow approaches in the Wasserstein distance have been developed for Poisson--Nernst--Planck systems and their variants, including nonlinear diffusions and several species~\cite{di2008large,dkinderlehrer2015wasserstein,MR3485423}. The displacement convexity method of~\cite{MR2053570} does not apply directly to the Coulomb interaction in dimension $d\ge2$, because the corresponding interaction energy is not displacement convex. For aggregation--diffusion equations with general interaction kernels, radial symmetry and uniqueness of stationary states and long time asymptotics were studied in~\cite{CHVY}.

The repulsive Coulomb and logarithmic interactions with quadratic confinement have also attracted attention in the context of mean-field limits of interacting particle systems. Uniform in time propagation of chaos and time-uniform functional inequalities have been obtained in~\cite{GLM,RosenzweigSerfaty,MRW}. These results rely on logarithmic Sobolev or Poincar\'e inequalities for the limiting nonlinear equation. Theorem~\ref{Thm:Sharp} gives the optimal constant at the level of the linearized equation. For the kinetic counterpart, the Vlasov--Poisson--Fokker--Planck system with confinement, $\mathrm L^2$-hypocoercivity of the linearized problem and exponential convergence were established in~\cite{addala,Toshpulatov}. Finally, Poisson--Nernst--Planck systems are widely used in scientific computing, see for instance~\cite{MR1450844,park1997qualitative,liu2014free}, where exponential rates of convergence are relevant for the design of long-time stable schemes. We refer to~\cite{MR2997301} for a discussion of the evolution problem from the point of view of physics. Similar ideas, based on a scalar product adapted to the free energy, have been used in~\cite{li2019flocking,Surin} for a McKean--Vlasov model of flocking and \cite{li2026patterns} for Smoluchowski equation, which also involves a non-local coupling. 

The present paper is a corrected and extended version of the preprint~\cite{LiPreprint}, which it supersedes. The functional framework is the same: free energy, scalar product adapted to its Hessian, coercivity of the quadratic forms. In~\cite{LiPreprint}, the rate $\Lambda=2\,\mu$ announced for the harmonic potential in dimension $d=2$ is the decay rate of the squared norm $\ird{|n-n_\infty|^2\,n_\infty^{-1}}$, which corresponds to the coercivity constant $\Cs=\mu$, in agreement with Theorem~\ref{Thm4.1} and Theorem~\ref{Thm:Sharp}. However, the identification of $\mu$ as the optimal constant in~\cite{LiPreprint} was limited to $d=2$, and the treatment of the non-radial components was incomplete: the fact that the infimum over non-radial perturbations is achieved by the first Fourier component was not proved, nor was the reduction of that component to the radial problem. Here we give a complete proof in dimensions $d=2$ and $d=3$, based on a different method (the duality criterion of Lemma~\ref{Lem:Flux} and the chiral fluxes of Section~\ref{Sec:Sharp}), we characterize the equality case, and we establish the improved rate of Theorem~\ref{Thm:Improved}.

To the best of our knowledge, the exact value of the spectral gap of the linearized Poisson--Nernst--Planck operator, and the identification of the translation modes as the extremal functions, were not previously known, even for the harmonic potential.

\subsection{Outline}

In Section~\ref{NeSec:mi}, we study the minimizer of the free energy and the stationary solution, establish uniform bounds on the solutions of~\eqref{Eq:(5.2)}, and prove convergence without rate. In Section~\ref{Sec:Coercivity-NP}, we introduce the quadratic forms $Q_1$ and $Q_2$ and prove the coercivity inequality for a general potential. Then we prove Theorem~\ref{Thm:Sharp} in Section~\ref{Sec:Sharp}, and estimate the second spectral gap in Section~\ref{Sec:SecondGap}. Section~\ref{Sec:LargeTime} is devoted to the proofs of Theorems~\ref{Thm4.1} and~\ref{Thm:Improved}, to rates in $\mathrm L^p$ norms, and to intermediate asymptotics. Appendix~\ref{App:WP} is devoted to the well-posedness theory in the weighted space $\mathrm L^2(n_\infty^{-1}dx)$.

\section{Minimizers of the free energy and convergence to the stationary solution}\label{NeSec:mi}
The main goal of this section is to prove that the minimizer of the free energy~$\mathcal F$ is the stationary solution $(n_{\infty}, c_{\infty})$ considered in the introduction and that it attracts any solution of~\eqref{Eq:(5.2)} as $t\to+\infty$.

\subsection{An estimate of the electric field}

We shall repeatedly use the following elementary estimate.
\begin{lem}\label{Lem:Field} Let $d=2$ or $3$. For any $\rho \in \mathrm L^{1}\cap \mathrm L^{d+1}(\mathbb{R}^d)$,
\be{EstimRho}
\|\nabla((-\Delta)^{-1} \rho)\| _{\mathrm L^{\infty}(\mathbb{R}^d)}\le \|\rho\| _{\mathrm L^{1}(\mathbb{R}^d)}+ d^\frac d{d+1}\,|\mathbb S^{d-1}|^{-\frac1{d+1}}\,\|\rho\| _{\mathrm L^{d+1}(\mathbb{R}^d)}\,.
\ee
\end{lem}
\begin{proof} The method is inspired by~\cite[Section~3]{MR3196188}. If $h=(-\Delta)^{-1}\rho$, then
\[
|\nabla h(x)|\le\frac{1}{|\mathbb S^{d-1}|}\int_{\R^d}{\frac{|\rho (y)|}{|x-y|^{d-1}}}\,dy\,.
\]
We split the integral into two parts corresponding to $|x-y| < 1$ and $|x-y| \ge1$. Since $\int_{|z|<1}|z|^{-(d-1)\frac{d+1}d}\,dz=d\,|\mathbb S^{d-1}|$, H\"older's inequality gives
\begin{align*}
&\frac{1}{|\mathbb S^{d-1}|}\int_{|x-y|<1}{\frac{|\rho (y)|}{|x-y|^{d-1}}}\,dy \le d^\frac d{d+1}\,|\mathbb S^{d-1}|^{-\frac1{d+1}}\,\|\rho\|_{\mathrm L^{d+1}(\mathbb{R}^d)}\,,\\
&\frac{1}{|\mathbb S^{d-1}|}\int_{|x-y|\ge1}{\frac{|\rho (y)|}{|x-y|^{d-1}}}\,dy \le \frac{1}{|\mathbb S^{d-1}|}\, \|\rho\|_{\mathrm L^1(\mathbb{R}^d)}\le\|\rho\|_{\mathrm L^1(\mathbb{R}^d)}\,.
\end{align*}
This proves~\eqref{EstimRho}.
\end{proof}

\subsection{Minimizers of the free energy and stationary solutions}
\label{Sec:mini}

In the following, we use the auxiliary function
\be{psi}
\psi(x):=\phi(x)\quad\mbox{if}\quad d=3\,,\qquad \psi(x):=\phi(x)-\frac M{2\,\pi}\,\log\big(1+|x|\big)\quad\mbox{if}\quad d=2\,.
\ee
By~\eqref{Hyp-C1} if $d=3$, and by~\eqref{Hyp-C4} if $d=2$, there exists $\theta\in(0,1)$ such that $e^{-\theta\,\psi}\in\mathrm L^1(\R^d)$. Moreover, $\psi$ is bounded from below and $\lim_{|x|\to+\infty}\psi(x)=+\infty$.
\begin{lem}\label{Lem:LowerBound} Let $d=2$ or $d=3$ and assume that the potential $\phi$ satisfies~\hyperlink{C}{$(\mathrm C)$}. On the set
\[
\mathcal X:=\Big\{f\in\mathrm L_+^1(\mathbb{R}^d)\,:\,\ird{f}=M\,,\; f\,\log f\in \mathrm L^1(\mathbb{R}^d)\,,\; f\,\phi\in \mathrm L^1(\mathbb{R}^d)\Big\}\,,
\]
the free energy $\mathcal F$ is bounded from below. More precisely, with $\theta$ as above, there exists a constant $C_\theta\in\R$ such that
\be{LowerF}
\mathcal F[f]\ge C_\theta+(1-\theta)\ird{f\,\psi}\quad\forall\,f\in\mathcal X\,.
\ee
\end{lem}
\begin{proof} If $d=3$, we have $\ird{f\,(-\Delta)^{-1}f}\ge 0$ because the Green function $G_3$ is positive. If $d=2$, we use the elementary inequality $|x-y|\le|x|+|y|\le(1+|x|)\,(1+|y|)$. It gives
\[
\frac12\ird{f\,(-\Delta)^{-1}f}=-\frac1{4\,\pi}\iint_{\R^2\times\R^2}f(x)\,f(y)\,\log|x-y|\,dx\,dy\ge-\frac M{2\,\pi}\int_{\R^2}f\,\log(1+|x|)\,dx\,.
\]
In both cases, we obtain $\mathcal F[f]\ge\ird{f\,\log f}+\ird{f\,\psi}$. Let $\rho_\theta:=M\,e^{-\theta\,\psi}/\ird{e^{-\theta\,\psi}}$. By Jensen's inequality, $\ird{f\,\log(f/\rho_\theta)}\ge0$, so that
\[
\ird{f\,\log f}+\theta\ird{f\,\psi}\ge M\,\log\(\frac M{\ird{e^{-\theta\,\psi}}}\)=:C_\theta\,,
\]
which proves~\eqref{LowerF}.
\end{proof}
If $d=2$, the proof above only uses $\liminf_{|x|\to+\infty}\phi(x)/\log|x|>2+\frac M{2\pi}$, which is what is needed for the integrability of $e^{-\theta\psi}$ for some $\theta\in(0,1)$. Lower bounds on $\mathcal F$ can also be deduced from the generalized logarithmic Hardy--Littlewood--Sobolev inequality of~\cite[Theorem~1.1 and Theorem~3.4]{DolbeaultLi}, which gives an optimal constant for the model potential $V(x)=2\log(1+|x|^2)+\log\pi$. The elementary proof above is enough for our purpose, and it applies to $d=2$ and $d=3$ simultaneously. The stronger growth condition~\eqref{Hyp-C4} is used only through the decay $n_\infty(x)=O(|x|^{-4-\eps})$ of Lemma~\ref{Lem:ninftyBounds}, which gives the second moment estimate~\eqref{FirstMoment}.
\begin{remark}\label{Rem:C4}
Let $d=2$ and $\gamma:=\limsup_{|x|\to+\infty}\phi(x)/\log|x|$. Let $f\ge0$ be a smooth function of mass $M$ supported in $\{1\le|x|\le2\}$, and $f_R(x):=R^{-2}f(x/R)$. Then $\ird{f_R\log f_R}=\ird{f\log f}-2\,M\,\log R$ and $\frac12\ird{f_R\,(-\Delta)^{-1}f_R}=\frac12\ird{f\,(-\Delta)^{-1}f}-\frac{M^2}{4\pi}\log R$. Moreover, for any $\eps>0$, $\ird{\phi\,f_R}\le(\gamma+\eps)\,M\,\log R+C_\eps$ for $R$ large. Hence
\[
\mathcal F[f_R]\le\(\gamma+\eps-2-\frac M{4\,\pi}\)M\,\log R+C_\eps'\,,
\]
and $\mathcal F$ is not bounded from below if $\gamma<2+\frac M{4\pi}$. On the other hand, Lemma~\ref{Lem:LowerBound} shows that $\mathcal F$ is bounded from below if $\liminf_{|x|\to+\infty}\phi(x)/\log|x|>2+\frac M{2\pi}$. We do not know the optimal threshold between $2+\frac M{4\pi}$ and $2+\frac M{2\pi}$. Under Assumption~\hyperlink{C}{$(\mathrm C)$}, these thresholds play no role because of Remark~\ref{Rem:Linear}. They are relevant only if~\eqref{Hyp-C3} is replaced by the Poincar\'e inequality~\eqref{Hyp-C2} as an assumption, in which case the second moment condition~\eqref{Hyp-C4} is the one which is used in Section~\ref{Sec:Coercivity-NP}.
\end{remark}

\begin{lem}\label{Lem:Positivity} Let $d=2$ or $3$ and $h\in\mathrm L^1(\R^d)$ be such that $\iint_{\R^d\times\R^d}|h(x)|\,|h(y)|\,|G_d(x-y)|\,dx\,dy<\infty$. If $d=2$, assume moreover that $\irtwo h=0$ and $\irtwo{|x|\,|h|}<\infty$. Then
\[
\iint_{\R^d\times\R^d}h(x)\,h(y)\,G_d(x-y)\,dx\,dy\ge0\,.
\]
\end{lem}
\begin{proof}
Let $\gamma_\eps(x):=(4\pi\eps)^{-d/2}\,e^{-|x|^2/(4\eps)}$ and $h_\eps:=\gamma_\eps*h$, for $\eps\in(0,1)$. Then $h_\eps\in\mathrm L^1\cap\mathrm L^2(\R^d)$ and, if $d=2$, $h_\eps$ has zero integral and a finite first moment. The Fourier computation of the proof of Lemma~\ref{Lem:GradGL2} below, which applies to any such function, shows that
\[
\ird{h_\eps\,(G_d*h_\eps)}=\frac1{(2\pi)^d}\int_{\R^d}\frac{|\hat h_\eps(\xi)|^2}{|\xi|^2}\,d\xi\ge0\,.
\]
On the other hand, $\ird{h_\eps\,(G_d*h_\eps)}=\iint h(x)\,h(y)\,K_\eps(x-y)\,dx\,dy$ with $K_\eps:=\gamma_\eps*\gamma_\eps*G_d=\gamma_{2\eps}*G_d$, and $K_\eps(z)\to G_d(z)$ as $\eps\to0_+$ for any $z\neq0$. Since $G_d$ is superharmonic and $\gamma_{2\eps}$ is a radial probability density, $K_\eps\le G_d$. If $d=3$, we also have $K_\eps>0$. If $d=2$, $\log|z-w|\le\log(1+|z|)+\log(1+|w|)$ gives $K_\eps(z)\ge-\frac1{2\pi}\big(\log(1+|z|)+C\big)$ with $C:=\sup_{\eps\in(0,1)}\irtwo{\gamma_{2\eps}(w)\,\log(1+|w|)}$. Hence $|K_\eps(x-y)|\le G_3(x-y)$ if $d=3$, and $|K_\eps(x-y)|\le|G_2(x-y)|+\frac1{2\pi}\,(|x|+|y|+C)$ if $d=2$. In both cases, the bound is integrable with respect to $|h(x)|\,|h(y)|\,dx\,dy$, and we conclude by dominated convergence.
\end{proof}
Lemma~\ref{Lem:Positivity} applies to $h=f_1-f_2$ if $f_1$, $f_2\in\mathcal X$ have finite interaction energies. Indeed, by Remark~\ref{Rem:Linear}, the elements of $\mathcal X$ have a finite first moment. If $d=3$, the Cauchy--Schwarz inequality for the positive kernel $G_3$, which is first applied to truncated functions and then extended by monotone convergence, shows that $\iint f_1(x)\,f_2(y)\,G_3(x-y)\,dx\,dy$ is finite. If $d=2$, the interaction energy of any $f\in\mathcal X$ is finite: $\log_+|x-y|\le|x|+|y|$, and Young's inequality $a\,b\le a\,\log a-a+e^b$ with $a=f(y)$ and $b=\log_-|x-y|$ gives
\[
\int_{|x-y|<1}f(y)\,\log_-|x-y|\,dy\le\irtwo{f\,|\log f|}+\int_{|z|<1}\frac{dz}{|z|}=\irtwo{f\,|\log f|}+2\,\pi\quad\forall\,x\in\R^2\,.
\]

\begin{lem}\label{Lem:FuniqueMin} Let $d=2$ or $d=3$ and assume that the potential $\phi$ satisfies~\hyperlink{C}{$(\mathrm C)$}.
There exists a unique minimizer $n_{\infty}$ of $\mathcal F$ in $\mathcal X$. Moreover, any solution $\bar n\in\mathcal X$ of~\eqref{st} coincides with $n_\infty$.
\end{lem}
\begin{proof} By~\eqref{LowerF}, a minimizing sequence $(f_k)_{k\in\N}$ is such that $\sup_k\ird{f_k\,|\log f_k|}<\infty$ (see below) and $\sup_k\ird{f_k\,\psi}<\infty$. Hence, up to the extraction of a subsequence, $f_k\rightharpoonup n_\infty$ weakly in $\mathrm L^1(\R^d)$, and $\ird{n_\infty}=M$ by tightness. Since $\ird{f_k}=M$, we can write, with $\rho_\theta=M\,e^{-\theta\psi}/\ird{e^{-\theta\psi}}$ as in the proof of Lemma~\ref{Lem:LowerBound},
\begin{multline*}
\ird{f_k\,\log f_k}+\ird{f_k\,\phi}=\ird{\rho_\theta\,h\(\frac{f_k}{\rho_\theta}\)}+(1-\theta)\ird{f_k\,\psi}\\
+\frac M{2\,\pi}\ird{f_k\,\log(1+|x|)}\,\ind_{d=2}+M\,\log\(\frac M{\ird{e^{-\theta\psi}}}\)\,,
\end{multline*}
where $h(s):=s\,\log s-s+1\ge0$. The first term is the integral of a nonnegative convex function of $f_k$, hence it is weakly lower semicontinuous in $\mathrm L^1(\R^d)$. The second and third terms are integrals of $f_k$ against functions which are bounded from below, hence they are also weakly lower semicontinuous. This takes care of the possibly negative part of $f_k\log f_k$ on the unbounded domain $\R^d$. 

Let us consider the interaction energy. Since weak convergence in $\mathrm L^1$ does not imply pointwise convergence, we argue on the product space. Let $F_k:=f_k\otimes f_k$. We first observe that $\sup_k\ird{f_k\,|\log f_k|}<\infty$. Indeed, $\ird{f_k\,\log f_k}\le\mathcal F[f_k]-\ird{f_k\,\psi}$ is bounded from above because $\psi$ is bounded from below, so that $\ird{\rho_\theta\,h(f_k/\rho_\theta)}=\ird{f_k\,\log f_k}+\theta\ird{f_k\,\psi}-C_\theta$ is bounded uniformly in $k$. Using $f_k\,\log f_k=f_k\,\log(f_k/\rho_\theta)+f_k\,\log\rho_\theta$, the elementary inequality $|s\,\log s|\le h(s)+s+1$ for any $s\ge0$, and $|\log\rho_\theta|\le\theta\,|\psi|+C$, we get
\[
\ird{f_k\,|\log f_k|}\le\ird{\rho_\theta\,h\(\frac{f_k}{\rho_\theta}\)}+2\,M+\theta\ird{f_k\,|\psi|}+C\,M\,,
\]
which is bounded uniformly in $k$, because $\psi$ is bounded from below and $\sup_k\ird{f_k\,\psi}<\infty$. Since $|\log F_k(x,y)|\le|\log f_k(x)|+|\log f_k(y)|$, we deduce that
\[
\iint_{\R^d\times\R^d}F_k\,|\log F_k|\,dx\,dy\le2\,M\ird{f_k\,|\log f_k|}\quad\mbox{and}\quad\iint_{\R^d\times\R^d}F_k\,\big(\psi(x)+\psi(y)\big)\,dx\,dy=2\,M\ird{f_k\,\psi}
\]
are bounded uniformly in $k$. By the de la Vall\'ee-Poussin criterion, $(F_k)_{k\in\N}$ is uniformly integrable, and it is tight because $\psi(x)+\psi(y)\to+\infty$ as $|x|+|y|\to+\infty$. By the Dunford--Pettis theorem, it is relatively compact for the weak topology of $\mathrm L^1(\R^d\times\R^d)$. Any limit point coincides with $n_\infty\otimes n_\infty$, because $\iint F_k\,a(x)\,b(y)\,dx\,dy\to\ird{n_\infty\,a}\ird{n_\infty\,b}$ for any $a$, $b\in\mathrm L^\infty(\R^d)$. Hence $F_k\rightharpoonup n_\infty\otimes n_\infty$ weakly in $\mathrm L^1(\R^d\times\R^d)$. Let $\mathsf K\ge0$ be a measurable kernel on $\R^d\times\R^d$. For any $j\in\N$, $\min\{\mathsf K,j\}$ is bounded, so that
\[
\liminf_{k\to+\infty}\iint F_k\,\mathsf K\,dx\,dy\ge\lim_{k\to+\infty}\iint F_k\,\min\{\mathsf K,j\}\,dx\,dy=\iint n_\infty\otimes n_\infty\,\min\{\mathsf K,j\}\,dx\,dy\,,
\]
and we let $j\to+\infty$ by monotone convergence. This proves the lower semicontinuity of the interaction energy if $d=3$, with $\mathsf K(x,y)=G_3(x-y)>0$. If $d=2$, we split $-\log|x-y|=\log_-|x-y|-\log_+|x-y|$. The part with $\mathsf K(x,y)=\log_-|x-y|\ge0$ is lower semicontinuous by the same argument. For the other part, we use $0\le\log_+|x-y|\le\log(1+|x|)+\log(1+|y|)$. For any $R>0$,
\[
\iint_{\{|x|>R\}\cup\{|y|>R\}}F_k\,\log_+|x-y|\,dx\,dy\le4\,M\int_{|x|>R}f_k\,\log(1+|x|)\,dx+2\,M\,\log(1+R)\int_{|x|>R}f_k\,dx\,,
\]
and the right-hand side converges to $0$ as $R\to+\infty$, uniformly in $k$, because $\log(1+|x|)=o(\psi(x))$ as $|x|\to+\infty$ and $\sup_k\ird{f_k\,\psi}<\infty$. On $B_R\times B_R$, the function $\log_+|x-y|$ is bounded, so that the corresponding integral converges by weak convergence of $F_k$. Hence $\iint F_k\,\log_+|x-y|\,dx\,dy$ converges to $\iint n_\infty\otimes n_\infty\,\log_+|x-y|\,dx\,dy$. Altogether, $\mathcal F[n_\infty]\le\liminf_{k\to+\infty}\mathcal F[f_k]$ and $n_\infty\in\mathcal X$ is a minimizer.

The minimizer is positive a.e. Otherwise, there is a bounded measurable set $A\subset\{n_\infty=0\}$ with $|A|>0$. For $\eps\in(0,1)$, let $n_\eps:=(1-\eps)\,n_\infty+\eps\,M\,|A|^{-1}\,\ind_A\in\mathcal X$. Since $n_\infty=0$ on $A$,
\[
\ird{n_\eps\,\log n_\eps}=(1-\eps)\ird{n_\infty\,\log n_\infty}+M\,(1-\eps)\,\log(1-\eps)+\eps\,M\,\log\(\frac{\eps\,M}{|A|}\)\,.
\]
The potential energy is an affine function of $\eps$ and the interaction energy is a polynomial of degree two in $\eps$, with finite coefficients because $\phi$ is bounded on $A$ and $c_\infty=G_d*n_\infty$ is locally integrable. Hence $\mathcal F[n_\eps]=\mathcal F[n_\infty]+\eps\,M\,\log\eps+O(\eps)<\mathcal F[n_\infty]$ for $\eps>0$ small enough, a contradiction.

For any $\delta\in(0,1)$, let $\eta$ be a bounded function with support in $\{\delta\le n_\infty\le1/\delta\}\cap B_{1/\delta}$ such that $\ird\eta=0$. Then $n_\infty+\eps\,\eta\in\mathcal X$ for $|\eps|$ small, and $\frac d{d\eps}\mathcal F[n_\infty+\eps\,\eta]_{|\eps=0}=\ird{\eta\,\big(\log n_\infty+\phi+c_\infty\big)}=0$. Since $\delta$ is arbitrary and $n_\infty>0$ a.e., the minimizer $n_{\infty}$ satisfies the Euler--Lagrange equation
\[
\log n_\infty+\phi+c_\infty=\lambda\,,\quad c_\infty=(-\Delta)^{-1}n_{\infty}\,,
\]
for some Lagrange multiplier $\lambda$ associated with the mass constraint, which means that $(n_\infty,c_\infty)$ solves~\eqref{st}.
Conversely, let $(\bar n,\bar c)$ be any solution of~\eqref{st} with $\bar n\in\mathcal X$ and $\bar c=G_d*\bar n$, so that $\log\bar n+\phi+\bar c$ is constant. By direct computation, with $c=(-\Delta)^{-1}n$, we observe that
\be{Fdiff}
\mathcal{F}[n]-\mathcal{F}[\bar n]=\ird{n\,\log\left(\frac{n}{\bar n}\right)}+ \frac{1}{2}\ird{(n-\bar n)\,(c-\bar c)}\,.
\ee
Since $\ird n=\ird{\bar n}=M$, we obtain from Jensen's inequality that
\[
\ird{n\log\left(\frac{n}{\bar n}\right)}\ge 0\,,
\]
with equality if and only if $n=\bar n$. The interaction energy of $\bar n$ is finite, because $\bar n$ is bounded if $d=3$. If $\mathcal F[n]=+\infty$, there is nothing to prove. Otherwise, the interaction energy of $n$ is finite and, by Lemma~\ref{Lem:Positivity} applied to $h=n-\bar n$, we also have
\[
\ird{(n-\bar n)\,(c-\bar c)}\ge 0\,,
\]
see also~\cite{MR2226917}. Hence $\mathcal F[n]-\mathcal F[\bar n]\ge 0$ for any $n\in\mathcal X$, with equality if and only if $n=\bar n$. This means that $\bar n$ is the unique minimizer of $\mathcal F$, that is, $\bar n=n_\infty$.
\end{proof}
As a consequence of the uniqueness result of Lemma~\ref{Lem:FuniqueMin}, $n_\infty$ is radially symmetric if $\phi$ is radially symmetric.

We now establish pointwise bounds on $n_\infty$ by a direct argument. By the Euler--Lagrange equation, $n_\infty=e^{\lambda}\,e^{-c_\infty-\phi}$. By~\eqref{Hyp-C1}, there is a constant $C_0$ such that $\log(2+|x|)\le C_0\,(1+\phi(x)-\inf\phi)$ for any $x\in\R^d$. Since $n_\infty\in\mathcal X$, the logarithmic moment
\be{logmoment}
A:=\int_{\R^d}n_\infty(y)\,\log\big(2+|y|\big)\,dy
\ee
is finite.
\begin{lem}\label{Lem:ninftyBounds} Let $d=2$ or $d=3$ and assume that the potential $\phi$ satisfies~\hyperlink{C}{$(\mathrm C)$}. Then:
\begin{enumerate}[label=(\roman*)]
\item if $d=3$, then $0\le c_\infty\le K$ for some constant $K$, and $e^{\lambda-K}\,e^{-\phi}\le n_\infty\le e^{\lambda}\,e^{-\phi}$ on $\R^3$;
\item if $d=2$, then $\sup_{|x|\ge1}\big|c_\infty(x)+\frac M{2\pi}\log|x|\big|<\infty$, and $n_\infty\le e^{\lambda+\frac A{2\pi}}\,(1+|x|)^{\frac M{2\pi}}\,e^{-\phi}$ on $\R^2$.
\end{enumerate}
As a consequence, $n_\infty\in\mathrm L^\infty(\R^d)$ and $n_\infty$ is bounded from above and from below by positive multiples of
\begin{equation}
\label{infcon}
|x|^{\frac{M}{2\,\pi}}e^{-\phi}\quad \mbox{if}\quad d=2\,,\qquad e^{-\phi}\quad \mbox{if}\quad d=3\,,
\ee
for $|x|\ge1$. Moreover, if $d=2$, then $n_\infty\le C\,|x|^{-4-\eps}$ for some $\eps>0$ and $|x|$ large, by~\eqref{Hyp-C4}.
\end{lem}
\begin{proof}
\emph{Case $d=3$.} We have $c_\infty=G_3*n_\infty\ge0$ because $G_3>0$, hence $n_\infty\le e^\lambda\,e^{-\phi}$. Since $\phi$ is continuous and $\phi(x)\to+\infty$ as $|x|\to+\infty$, $n_\infty$ is bounded. By splitting the integral according to $|x-y|\ge1$ and $|x-y|<1$, and using $\int_{|z|<1}|z|^{-1}\,dz=2\,\pi$, we get
\[
0\le c_\infty(x)=\frac1{4\,\pi}\int_{\R^3}\frac{n_\infty(y)}{|x-y|}\,dy\le\frac1{4\,\pi}\,\big(\nrm{n_\infty}{\mathrm L^1}+2\,\pi\,\nrm{n_\infty}{\mathrm L^\infty}\big)=:K\,.
\]

\emph{Case $d=2$, upper bound.} Since $|x-y|\le(1+|x|)\,(1+|y|)$, we have
\[
c_\infty(x)=-\frac1{2\,\pi}\int_{\R^2}\log|x-y|\,n_\infty(y)\,dy\ge-\frac M{2\,\pi}\,\log(1+|x|)-\frac A{2\,\pi}\quad\forall\,x\in\R^2\,,
\]
which gives the upper bound on $n_\infty$. By~\eqref{Hyp-C4}, $\phi(x)-\frac M{2\pi}\log(1+|x|)\to+\infty$ as $|x|\to+\infty$, so that $n_\infty\in\mathrm L^\infty(\R^2)$.

\emph{Case $d=2$, two-sided bound on $c_\infty$.} The lower bound $c_\infty(x)+\frac M{2\pi}\log|x|\ge-\frac A{2\pi}-\frac M{2\pi}\log2$ for $|x|\ge1$ follows from the previous step. For the upper bound, let $|x|\ge2$ and write
\[
c_\infty(x)+\frac M{2\,\pi}\log|x|=\frac1{2\,\pi}\int_{\R^2}\log\frac{|x|}{|x-y|}\,n_\infty(y)\,dy\,.
\]
If $|y|\le|x|/2$, then $|x-y|\ge|x|/2$ and the integrand is bounded by $\log2\;n_\infty(y)$. If $|y|>|x|/2$ and $|x-y|\ge1$, the integrand is bounded by $\log|x|\;n_\infty(y)\le\log(2\,|y|)\,n_\infty(y)$. If $|x-y|<1$, then $|y|>|x|-1\ge|x|/2$, and the integrand is bounded by $\big(\log(2\,|y|)+|\log|x-y||\big)\,n_\infty(y)$. Using $\int_{|z|<1}|\log|z||\,dz=\frac\pi2$, we obtain
\[
c_\infty(x)+\frac M{2\,\pi}\log|x|\le\frac1{2\,\pi}\(M\,\log2+2\,A+\frac\pi2\,\nrm{n_\infty}{\mathrm L^\infty}\)\quad\forall\,x\in\R^2\setminus B_2\,.
\]
On $\{1\le|x|\le2\}$, $c_\infty$ is continuous, hence bounded. This proves~(ii), and~\eqref{infcon} follows from $n_\infty=e^{\lambda}\,e^{-c_\infty-\phi}$.
\end{proof}
\begin{prop}\label{Prop:cinfty} Let $d=2$ or $d=3$ and assume that the potential $\phi$ satisfies~\hyperlink{C}{$(\mathrm C)$}.
Then $n_\infty\in\mathrm L^1\cap\mathrm L^\infty(\R^d)$ and $\nabla c_\infty\in\mathrm L^\infty(\R^d)$. Moreover, $c_{\infty}$ is bounded if $d=3$, and $\nrm{\nabla c_{\infty}}{\mathrm L^q(\R^2)}$ is finite for any $q\in(2,+\infty]$ if $d=2$.
\end{prop}
\begin{proof} The first statement and the boundedness of $c_\infty$ if $d=3$ follow from Lemma~\ref{Lem:ninftyBounds}.

In both cases, Lemma~\ref{Lem:Field} applied with $\rho=n_\infty$ shows that $\nabla c_\infty\in\mathrm L^\infty(\R^d)$. If $d=2$, the decay of $n_\infty$ implies that $|\nabla c_\infty(x)|\le C\,|x|^{-1}$ for $|x|$ large. Hence $\nabla c_\infty\in\mathrm L^q(\R^2)$ for any $q\in(2,+\infty]$.
\end{proof}

\subsection{Uniform bounds on the solution of~\texorpdfstring{\eqref{Eq:(5.2)}}{(1.2)}}

Throughout the paper, we use the following notion of solution.
\begin{defin}\label{Def:Solution}
Let $n_0\in\mathrm L^1_+\cap\mathrm L^2(\R^d)$. A \emph{solution} of~\eqref{Eq:(5.2)} is a nonnegative function $n\in C\big([0,\infty);\mathrm L^1\cap\mathrm L^2(\R^d)\big)$ with $n(0,\cdot)=n_0$ such that the following properties hold:
\begin{enumerate}[label=(\roman*)]
\item $n$ is a strong solution of~\eqref{Eq:(5.2)} on $(0,\infty)\times\R^d$, with $c=G_d*n$: $n$ and $\nabla n$ are continuous, $\partial_tn$ and $D^2n$ belong to $\mathrm L^p_{\rm loc}\big((0,\infty)\times\R^d\big)$ for any $p<\infty$, and the equation holds a.e.;
\item for any $0<\tau<t_0$, $n\in\mathrm L^\infty\big((\tau,t_0);\mathrm L^\infty(\R^d)\big)$ and $\nabla n\in\mathrm L^2\big((\tau,t_0);\mathrm L^2(\R^d)\big)$;
\item if $\mathcal F[n_0]<\infty$, the free energy inequality $\mathcal F[n(t,\cdot)]+\int_s^t\mathcal I[n(r,\cdot)]\,dr\le\mathcal F[n(s,\cdot)]$ holds for any $0\le s\le t$;
\item if $n_0\in\mathrm L^2(n_\infty^{-1}dx)$, then $f:=n/n_\infty-1$ belongs to $C\big([0,\infty);\mathrm L^2(n_\infty\,dx)\big)\cap\mathrm L^2_{\rm loc}\big([0,\infty);\mathrm H^1(n_\infty\,dx)\big)$, where $\mathrm H^1(n_\infty\,dx):=\{g\in\mathrm L^2(n_\infty\,dx)\,:\,\nabla g\in\mathrm L^2(n_\infty\,dx)\}$.
\end{enumerate}
\end{defin}
The mass is conserved (see Lemma~\ref{Lem:Moment}), and by Lemma~\ref{Lem:Field}, $\nabla c\in\mathrm L^\infty\big((\tau,t_0)\times\R^d\big)$ for any $0<\tau<t_0$. For initial data $n_0\in\mathrm L^2_+(n_\infty^{-1}dx)$, which is the setting of Theorems~\ref{Thm4.1} and~\ref{Thm:Improved}, the existence of a solution with Properties~(i)--(iv) is proved in Appendix~\ref{App:WP} (Proposition~\ref{Prop:A3}). For such data, $\mathcal F[n_0]$ is finite: by~\eqref{Fdiff} below, $\mathcal F[n_0]-\mathcal F[n_\infty]=\ird{n_\infty\,\Psi(f_0)}+\frac12\ird{|\nabla G_0|^2}$ with $f_0=n_0/n_\infty-1$, $G_0=G_d*(n_0-n_\infty)$ and $\Psi(z)=(1+z)\log(1+z)-z\le z^2$, and the right-hand side is finite by Lemma~\ref{Lem:GradGL2}. Here $n_0\in\mathcal X$: by Remark~\ref{Rem:Linear} and Lemma~\ref{Lem:ninftyBounds}, $n_\infty\,\phi^2\in\mathrm L^1(\R^d)$, so that $\ird{n_0\,\phi}\le\big(\ird{n_\infty\,\phi^2}\big)^{1/2}\big(\ird{n_0^2\,n_\infty^{-1}}\big)^{1/2}<\infty$, and $n_0\,\log n_0\in\mathrm L^1(\R^d)$ because $n_0\in\mathrm L^1\cap\mathrm L^2(\R^d)$ has a finite first moment. For more general initial data in $\mathrm L^1_+\cap\mathrm L^2(\R^d)$ with finite free energy, solutions with Properties~(i)--(iii) can be obtained from the classical theory for drift-diffusion--Poisson systems: mild solutions are constructed by a fixed point argument based on the Duhamel formula, and they are smooth for $t>0$ by parabolic regularity, see for instance~\cite{biler1994debye,MR1777308,arnold2000large} for related settings. We do not repeat this construction here: the results of Section~\ref{NeSec:mi} which are stated for such data hold for any solution with Properties~(i)--(iii), and all the statements of Section~\ref{NeSec:Intro} are covered by Appendix~\ref{App:WP}. Uniqueness holds in the class of the nonnegative functions which satisfy Properties~(i) and~(ii), if the initial datum has a finite second moment (Proposition~\ref{Prop:Uniq}). Property~(iii) is the integrated form of~\eqref{lypa}. In the sequel, the monotonicity of $t\mapsto\mathcal F[n(t,\cdot)]$ and Estimate~\eqref{conen} below are understood in the sense of~(iii). Property~(iv) expresses the well-posedness of~\eqref{Eq:(5.2)} in the weighted space $\mathrm L^2(n_\infty^{-1}dx)$, which is the natural functional setting of Theorems~\ref{Thm4.1} and~\ref{Thm:Improved}. It is used in the energy identity of Lemma~\ref{Lem:EnergyIdentity}. In the sequel, all solutions are understood in the sense of Definition~\ref{Def:Solution}; in particular, initial data are assumed to belong to $\mathrm L^1_+\cap\mathrm L^2(\R^d)$. This is the case under the assumptions of Theorems~\ref{Thm4.1} and~\ref{Thm:Improved}, since $\ird{n_0^2}\le\nrm{n_\infty}{\mathrm L^\infty}\ird{n_0^2\,n_\infty^{-1}}$. If $n_0\not\equiv0$, the strong maximum principle applied to the linear equation $\partial_tn=\Delta n+(\nabla c+\nabla\phi)\cdot\nabla n+(\Delta\phi-n)\,n$, which has locally bounded coefficients on $(0,\infty)\times\R^d$, shows that $n>0$ on $(0,\infty)\times\R^d$.

We now establish bounds on the solution $n(t,\cdot)$ of~\eqref{Eq:(5.2)} which are independent of $t$.
\begin{lem}\label{lem2.1} Let $d=2$ or $d=3$ and assume that the potential $\phi$ satisfies~\hyperlink{C}{$(\mathrm C)$}. Let $C>0$ be such that $\Delta\phi\le C$ a.e. Then any solution $n$ of~\eqref{Eq:(5.2)} in the sense of Definition~\ref{Def:Solution}, with mass $M$, satisfies
\begin{equation*}
\|n(t,\cdot)\|_{\mathrm L^p(\R^d)}\le(2\,C)^\frac{p-1}{p}\,M^\frac{1}{p}\quad\forall\, t\ge T:=\frac1C\,,\quad\forall\, p\in(1,+\infty)\,,
\end{equation*}
and $\|n(t,\cdot)\|_{\mathrm L^\infty(\R^d)}\le2\,C$ for any $t\ge T$.
\end{lem}
\begin{proof}
Fix $p\in(1,\infty)$ and $0<\tau<t_0$. The computation is localized with a cut-off function. Since $n>0$ has the regularity of Property~(i) of Definition~\ref{Def:Solution} on $(0,\infty)\times\R^d$, so does $n^{p-1}$, even if $p<2$, and all the integrations by parts below are justified. Let $\chi\in C_c^\infty(\R^d)$ be such that $0\le\chi\le1$, $\chi=1$ on $B_1$ and $\chi=0$ outside $B_2$, let $\chi_R:=\chi(\cdot/R)$ for $R\ge1$, and recall that $\partial_tn=\nabla\cdot J$ with $J=\nabla n+n\,\nabla c+n\,\nabla\phi$.

\emph{Step 1: localized $\mathrm L^p$ identity.} On $[\tau,t_0]$, we have
\[
\frac d{dt}\ird{n^p\,\chi_R}=p\ird{n^{p-1}\,\chi_R\,\nabla\cdot J}=-p\,(p-1)\ird{\chi_R\,n^{p-2}\,\nabla n\cdot J}-p\ird{n^{p-1}\,J\cdot\nabla\chi_R}\,.
\]
Using $p\,n^{p-1}\nabla n=\nabla(n^p)$, integrations by parts and $-\Delta c=n$, we get
\begin{align*}
-p\,(p-1)\ird{\chi_R\,n^{p-1}\,\nabla n\cdot\nabla c}&=(p-1)\ird{n^p\,\nabla\cdot(\chi_R\,\nabla c)}\\
&=(p-1)\ird{n^p\,\nabla\chi_R\cdot\nabla c}-(p-1)\ird{\chi_R\,n^{p+1}}\,,\\
-p\,(p-1)\ird{\chi_R\,n^{p-1}\,\nabla n\cdot\nabla\phi}&=(p-1)\ird{n^p\,\nabla\chi_R\cdot\nabla\phi}+(p-1)\ird{\chi_R\,n^p\,\Delta\phi}\,,\\
-p\ird{n^{p-1}\,\nabla n\cdot\nabla\chi_R}&=-\ird{\nabla(n^p)\cdot\nabla\chi_R}=\ird{n^p\,\Delta\chi_R}\,.
\end{align*}
Collecting terms, we obtain
\begin{multline}\label{LocLp}
\frac d{dt}\ird{n^p\,\chi_R}=-p\,(p-1)\ird{\chi_R\,n^{p-2}\,|\nabla n|^2}\\
-(p-1)\ird{\chi_R\,n^{p+1}}+(p-1)\ird{\chi_R\,n^p\,\Delta\phi}+\mathcal E_{R,p}\,,
\end{multline}
where
\[
\mathcal E_{R,p}:=\ird{n^p\,\Delta\chi_R}-\ird{n^p\,\nabla c\cdot\nabla\chi_R}-\ird{n^p\,\nabla\phi\cdot\nabla\chi_R}\,.
\]
Notice that, after the last integration by parts, $\mathcal E_{R,p}$ involves no derivative of $n$.

\emph{Step 2: estimate of the remainder.} With $|\Delta\chi_R|\le C_\chi\,R^{-2}$, $|\nabla\chi_R|\le C_\chi\,R^{-1}\,\ind_{R\le|x|\le2R}$ and $|\nabla\phi|\,|\nabla\chi_R|\le A\,\ind_{R\le|x|\le2R}$, where $A$ does not depend on $R\ge1$ because $|\nabla\phi(x)|\le|\nabla\phi(0)|+\nrm{D^2\phi}{\mathrm L^\infty}\,|x|$, we get
\[
|\mathcal E_{R,p}(s)|\le\frac{C_\chi}R\(1+\nrm{\nabla c(s)}{\mathrm L^\infty}\)\nrm{n(s)}{\mathrm L^p}^p+A\int_{|x|\ge R}n(s,x)^p\,dx\,.
\]
By Definition~\ref{Def:Solution}, $n\in C([\tau,t_0];\mathrm L^1(\R^d))\cap\mathrm L^\infty((\tau,t_0)\times\R^d)$. Let $N:=\nrm n{\mathrm L^\infty((\tau,t_0)\times\R^d)}$. Since $|a^p-b^p|\le p\,N^{p-1}\,|a-b|$ for any $a$, $b\in[0,N]$, we have $n^p\in C([\tau,t_0];\mathrm L^1(\R^d))$. Hence the family $\{n(s,\cdot)^p\}_{s\in[\tau,t_0]}$ is compact in $\mathrm L^1(\R^d)$, and $\sup_{s\in[\tau,t_0]}\int_{|x|\ge R}n(s,x)^p\,dx\to0$ as $R\to+\infty$. With the bounds on $\nrm{n(s)}{\mathrm L^p}$ and $\nrm{\nabla c(s)}{\mathrm L^\infty}$ on $[\tau,t_0]$, this shows that $\lim_{R\to+\infty}\int_\tau^{t_0}|\mathcal E_{R,p}(s)|\,ds=0$.

\emph{Step 3: integral inequality.} Let $z(t):=\ird{n(t,\cdot)^p}$, which is a continuous function on $[\tau,t_0]$ by Step~2. We drop the first term of the right-hand side of~\eqref{LocLp}, which is nonpositive, use $\Delta\phi\le C$, and integrate on $[s,t]\subset[\tau,t_0]$:
\begin{multline*}
\ird{n(t)^p\,\chi_R}-\ird{n(s)^p\,\chi_R}\\
\le(p-1)\int_s^t\(-\ird{\chi_R\,n(r)^{p+1}}+C\ird{\chi_R\,n(r)^p}\)dr+\int_\tau^{t_0}|\mathcal E_{R,p}(r)|\,dr\,.
\end{multline*}
Since $n\in\mathrm L^\infty((\tau,t_0)\times\R^d)$, the functions $n^p$ and $n^{p+1}$ belong to $\mathrm L^1((\tau,t_0)\times\R^d)$. We let $R\to+\infty$ using monotone convergence and Step~2, and obtain
\be{IntIneq}
z(t)-z(s)\le(p-1)\int_s^t\(-\ird{n(r)^{p+1}}+C\,z(r)\)dr\quad\forall\,\tau\le s\le t\le t_0\,.
\ee
By H\"older's inequality, $\ird{n^p}\le\left(\ird{n}\right)^{\frac{1}{p}}\left(\ird{n^{p+1}}\right)^{\frac{p-1}{p}}$, and the mass is $M$. Hence
\begin{equation*}
\ird{n^{p+1}}\ge M^{-\beta}\,z^{1+\beta}\,,\qquad\beta:=\frac1{p-1}\,.
\end{equation*}
With $a:=(p-1)\,M^{-\beta}$, $b:=(p-1)\,C$ and $F(z):=-a\,z^{1+\beta}+b\,z$, we deduce from~\eqref{IntIneq} that
\[
z(t)-z(s)\le\int_s^tF\big(z(r)\big)\,dr\quad\forall\,\tau\le s\le t\le t_0\,.
\]
Since $z$ and $F$ are continuous, this implies $D^+z(s)\le F(z(s))$ for any $s\in[\tau,t_0)$, where $D^+z(s):=\limsup_{h\to0_+}\frac{z(s+h)-z(s)}h$ denotes the upper right Dini derivative.

\emph{Step 4: comparison argument.} Let $Z_*:=(2\,b/a)^{1/\beta}=(2\,C)^{p-1}\,M$, so that $F(z)\le-\frac a2\,z^{1+\beta}$ for any $z\ge Z_*$. We claim that
\be{ZBound}
z(t)\le\max\Big\{Z_*\,,\;\big(\tfrac{a\,\beta}2\,(t-\tau)\big)^{-1/\beta}\Big\}\quad\forall\,t\in(\tau,t_0]\,.
\ee
Assume by contradiction that~\eqref{ZBound} fails at some $t_1\in(\tau,t_0]$. Then $z(t_1)>Z_*$. Let $s_0:=\inf\{s\in[\tau,t_1]\,:\,z>Z_*\mbox{ on }[s,t_1]\}$. On $(s_0,t_1]$, we have $z>Z_*$, so that $D^+z\le-\frac a2\,z^{1+\beta}$ and the continuous function $g:=z^{-\beta}$ satisfies $D^+g=-\beta\,z^{-\beta-1}D_+z\ge\frac{a\,\beta}2$, where $D_+z(s):=\liminf_{h\to0_+}\frac{z(s+h)-z(s)}h\le D^+z(s)$ is the lower right Dini derivative. Here we used that $\zeta\mapsto\zeta^{-\beta}$ is decreasing and $C^1$ on $(0,\infty)$. A continuous function whose upper right Dini derivative is bounded from below by a constant $\kappa$ on an interval satisfies $g(t)-g(s)\ge\kappa\,(t-s)$ on this interval. Hence
\[
z(t_1)^{-\beta}\ge z(s)^{-\beta}+\frac{a\,\beta}2\,(t_1-s)\quad\forall\,s\in(s_0,t_1]\,.
\]
If $s_0>\tau$, then $z(s_0)=Z_*$ by continuity, and letting $s\to s_0$ gives $z(t_1)^{-\beta}\ge Z_*^{-\beta}$, that is, $z(t_1)\le Z_*$, a contradiction. If $s_0=\tau$, then letting $s\to\tau$ gives $z(t_1)^{-\beta}\ge\frac{a\,\beta}2\,(t_1-\tau)$, which again contradicts the assumption on $t_1$. This proves~\eqref{ZBound}.

\emph{Step 5: conclusion.} The bound~\eqref{ZBound} does not depend on $z(\tau)$, and $\tau\in(0,t)$ and $t_0>t$ are arbitrary. Letting $\tau\to0_+$, we obtain $z(t)\le\max\big\{Z_*,(\frac{a\,\beta}2\,t)^{-1/\beta}\big\}$ for any $t>0$. Since $a\,\beta=M^{-\beta}$, we have $(\frac{a\,\beta}2\,t)^{-1/\beta}=M\,(2/t)^{p-1}\le(2\,C)^{p-1}\,M=Z_*$ if and only if $t\ge1/C$. Hence
\begin{equation*}
\|n(t,\cdot)\|_{\mathrm L^p(\R^d)}\le(2\,C)^\frac{p-1}{p}\,M^\frac{1}{p}\quad\forall\,t\ge\frac1C\,.
\end{equation*}
Since $n(t,\cdot)\in\mathrm L^1\cap\mathrm L^\infty(\R^d)$, we have $\|n(t,\cdot)\|_{\mathrm L^\infty}=\lim_{p\to+\infty}\|n(t,\cdot)\|_{\mathrm L^p}\le2\,C$ for any $t\ge1/C$. See~\cite{blanchet2010asymptotic} for a similar estimate in the case of the Keller--Segel model.
\end{proof}

\begin{cor}\label{Cor:gradc} Let $d=2$ or $d=3$ and assume that the potential $\phi$ satisfies~\hyperlink{C}{$(\mathrm C)$}. For any solution $n$ of~\eqref{Eq:(5.2)} with initial datum $n_0\in\mathrm L^1_+\cap\mathrm L^2(\R^d)$ such that $\mathcal F[n_0]<+\infty$, we have
\be{gradc2}
\|\nabla c(t,\cdot)-\nabla c_\infty\|_{\mathrm L^2(\R^d)}^2\le 2\,\big(\mathcal F[n_0]-\mathcal F[n_\infty]\big)\quad\forall\,t\ge0\,,
\ee
and, for any $q\in(2,+\infty]$, there exists a constant $\mathcal C_q>0$ such that, with $T$ as in Lemma~\ref{lem2.1},
\begin{equation*}
\|\nabla c(t,\cdot)\|_{\mathrm L^q(\mathbb{R}^d)}\le\mathcal C_q\quad\forall\, t\ge T\,.
\end{equation*}
\end{cor}
\begin{proof} Estimate~\eqref{gradc2} follows from~\eqref{Fdiff} with $\bar n=n_\infty$, from the monotonicity of $t\mapsto\mathcal F[n(t,\cdot)]$, and from the identity $\ird{(n-n_\infty)\,(c-c_\infty)}=\ird{|\nabla c-\nabla c_\infty|^2}$. This identity is given by the Fourier computation of the proof of Lemma~\ref{Lem:GradGL2}, because $n(t,\cdot)-n_\infty\in\mathrm L^1\cap\mathrm L^2(\R^d)$ has zero integral and, by~\eqref{LowerF} and Remark~\ref{Rem:Linear}, a finite first moment. By Lemmas~\ref{Lem:Field} and~\ref{lem2.1}, $\|\nabla c(t,\cdot)\|_{\mathrm L^\infty(\R^d)}\le M+C\,\|n(t,\cdot)\|_{\mathrm L^{d+1}(\R^d)}$ is bounded for $t\ge T$. The same holds for $\nabla c_\infty$ by Proposition~\ref{Prop:cinfty}. For $q\in(2,\infty)$, we interpolate $\nabla c-\nabla c_\infty$ between $\mathrm L^2$ and $\mathrm L^\infty$, and use $\nabla c_\infty\in\mathrm L^q(\R^d)$, which holds by Proposition~\ref{Prop:cinfty} if $d=2$ and because $|\nabla c_\infty(x)|=O(|x|^{-2})$ if $d=3$.
\end{proof}

\subsection{Convergence to stationary solutions}\label{Sec:CVNernstPlanck}

The next step is to establish the convergence \emph{without rate} of the solution of~\eqref{Eq:(5.2)} to the stationary solution.
\begin{prop}
\label{Thm:(2.1)}
Suppose that $d=2$ or $3$. Let $n$ be the solution of~\eqref{Eq:(5.2)} with $n_0\in\mathrm L^1_+\cap\mathrm L^2(\R^d)$, $\ird{n_0}=M$, $\mathcal F[n_0]<+\infty$, and assume that the potential $\phi$ satisfies~\hyperlink{C}{$(\mathrm C)$}. Then for any $p\in [1,\infty)$ and any $q\in (2,\infty]$, we have
\[
\lim_{t\to \infty}\|n(t,\cdot)-n_\infty\| _{\mathrm L^p(\mathbb{R}^d)}=0\quad\mbox{and}\quad
\lim_{t\to \infty}\|\nabla c(t,\cdot)-\nabla c_\infty\| _{\mathrm L^q(\mathbb{R}^d)} =0\,.
\]
If $d=3$, the second limit also holds with $q=2$.
\end{prop}
\begin{proof}
\emph{Step 1: a priori estimates.} Since $\mathcal F[n(t,\cdot)]-\mathcal F[n_\infty]\ge0$ is non-increasing, we know by~\eqref{lypa} that
\begin{equation}
\label{conen}
\int_0^\infty\mathcal I[n(s,\cdot)]\,ds\le\mathcal F[n_0]-\mathcal F[n_\infty]<\infty\,.
\end{equation}
By Lemma~\ref{Lem:LowerBound}, $\sup_{t\ge0}\ird{n(t,x)\,\psi(x)}<\infty$. Since $\psi\to+\infty$ at infinity, the family $\big(n(t,\cdot)\big)_{t\ge0}$ is tight. By Lemma~\ref{lem2.1} and Corollary~\ref{Cor:gradc}, $n$ and $\nabla c$ are bounded in $\mathrm L^\infty\big((T,\infty)\times\R^d\big)$, and $\nabla\phi$ is locally bounded. Hence $n$ solves, for $t>T$, a linear parabolic equation in divergence form,
\[
\partial_tn=\nabla\cdot\big(\nabla n+n\,b\big)\,,\qquad b:=\nabla c+\nabla\phi\in\mathrm L^\infty_{\rm loc}\,,
\]
with $n\,b$ bounded on $(T,\infty)\times B_R$ uniformly in time, for any $R>0$. By the De Giorgi--Nash--Moser theory (see~\cite[Chapter~III]{LSU}), $n$ is H\"older continuous on $[t,t+1]\times B_R$ for any $t\ge T+1$, with bounds which depend on $R$ but not on $t$.

\emph{Step 2: compactness.} Let $(t_k)_{k\in\N}$ be a sequence such that $t_k\to+\infty$, and let $n_k(s,x):=n(t_k+s,x)$ and $c_k:=(-\Delta)^{-1}n_k$ for $s\in[0,1]$. By Step~1 and the Arzel\`a--Ascoli theorem, up to the extraction of a subsequence, $(n_k)_{k\in\N}$ converges locally uniformly on $[0,1]\times\R^d$ to some $\bar n$. By tightness and the uniform $\mathrm L^\infty$ bound, $n_k\to\bar n$ in $C\big([0,1];\mathrm L^1\cap\mathrm L^{d+1}(\R^d)\big)$, and $\ird{\bar n(s,\cdot)}=M$ for any $s\in[0,1]$. By Lemma~\ref{Lem:Field}, $\nabla c_k\to\nabla\bar c:=\nabla(-\Delta)^{-1}\bar n$ uniformly on $[0,1]\times\R^d$.

\emph{Step 3: identification of the limit.} Let $J:=\nabla n+n\,\nabla(c+\phi)$, so that $\mathcal I[n]=\ird{|J|^2/n}$ and $\partial_tn=\nabla\cdot J$. By the Cauchy--Schwarz inequality and~\eqref{conen},
\[
\int_0^1\nrm{J(t_k+s,\cdot)}{\mathrm L^1(\R^d)}\,ds\le\(M\int_{t_k}^{t_k+1}\mathcal I[n(s,\cdot)]\,ds\)^{1/2}\to0\quad\mbox{as}\quad k\to+\infty\,.
\]
Hence $\partial_s\bar n=0$ in $\mathcal D'\big((0,1)\times\R^d\big)$, and passing to the limit in $J_k=\nabla n_k+n_k\,\nabla(c_k+\phi)$,
\[
\nabla\bar n+\bar n\,\nabla(\bar c+\phi)=0\quad\mbox{in}\quad\mathcal D'\big((0,1)\times\R^d\big)\,.
\]
Since $\bar c+\phi\in\mathrm W^{1,\infty}_{\rm loc}$, this means that $\nabla\big(\bar n\,e^{\bar c+\phi}\big)=0$, so that $\bar n=Z\,e^{-\bar c-\phi}$ for some constant $Z>0$. Hence $\bar n$ solves~\eqref{st} with mass $M$. Moreover, $\bar n$ is bounded and $\ird{\bar n\,\psi}\le\liminf_{k\to+\infty}\ird{n_k\,\psi}<\infty$ by Fatou's lemma, because $\psi$ is bounded from below and $n_k\to\bar n$ pointwise. Hence $\bar n\in\mathcal X$, and $\bar n=n_\infty$ by Lemma~\ref{Lem:FuniqueMin}. The limit does not depend on the subsequence. Hence $\lim_{t\to+\infty}\nrm{n(t,\cdot)-n_\infty}{\mathrm L^1\cap\mathrm L^{d+1}(\R^d)}=0$.

\emph{Step 4: conclusion.} The convergence in $\mathrm L^p(\R^d)$ for any $p\in[1,\infty)$ follows by H\"older interpolation with the uniform bound of Lemma~\ref{lem2.1}. By Lemma~\ref{Lem:Field},
\be{gradcinfty}
\nrm{\nabla c(t,\cdot)-\nabla c_\infty}{\mathrm L^\infty(\R^d)}\le\nrm{n(t,\cdot)-n_\infty}{\mathrm L^1(\R^d)}+C\,\nrm{n(t,\cdot)-n_\infty}{\mathrm L^{d+1}(\R^d)}\to0\,.
\ee
For $q\in(2,\infty)$, the result follows by interpolation with the bound~\eqref{gradc2}. If $d=3$, the Hardy--Littlewood--Sobolev inequality gives $\nrm{\nabla c-\nabla c_\infty}{\mathrm L^2(\R^3)}\le C\,\nrm{n-n_\infty}{\mathrm L^{6/5}(\R^3)}\to0$.
\end{proof}

\begin{remark} The above proof does not use the convergence of $\mathcal F[n(t,\cdot)]$ to $\mathcal F[n_\infty]$. By lower semicontinuity, only $\liminf_{t\to+\infty}\mathcal F[n(t,\cdot)]\ge\mathcal F[n_\infty]$ is obvious. Under the assumptions of Theorem~\ref{Thm4.1}, the convergence of the free energy is a consequence of the exponential decay of $Q_1$ established in Section~\ref{Sec:LargeTime}. For uniformly convex potentials, it also follows from Proposition~\ref{Prop:FreeEnergyDecay} below, with an explicit exponential rate.
\end{remark}

\subsection{Exponential decay of the free energy for uniformly convex potentials}\label{Sec:FreeEnergyDecay}

In this section, we give a direct nonlinear argument for the decay of the free energy. It relies on the \emph{time-dependent Gibbs state}
\be{LocalGibbs}
G_n:=M\,\frac{e^{-c-\phi}}{\ird{e^{-c-\phi}}}\,,\qquad c=G_d*n\,,
\ee
and on the observation that the Fisher information is exactly the relative Fisher information of $n$ with respect to $G_n$:
\[
\mathcal I[n]=\ird{n\,\left|\nabla\log\(\frac n{G_n}\)\right|^2}\,.
\]
This strategy was used in~\cite{li2019flocking} for a McKean--Vlasov model of flocking. Notice that $G_{n_\infty}=n_\infty$.
\begin{lem}\label{Lem:FvsGibbs}
Let $d=2$ or $3$ and assume that $\phi$ satisfies~\hyperlink{C}{$(\mathrm C)$}. For any $n\in\mathcal X\cap\mathrm L^\infty(\R^d)$, with $\mathcal X$ as in Lemma~\ref{Lem:LowerBound}, the function $G_n$ is well defined and
\be{FvsGibbs}
\mathcal F[n]-\mathcal F[n_\infty]\le\ird{n\,\log\(\frac n{G_n}\)}\,.
\ee
\end{lem}
\begin{proof}
If $d=3$, then $0\le c\le\frac1{4\pi}(\nrm n{\mathrm L^1}+2\pi\,\nrm n{\mathrm L^\infty})$, as in the proof of Lemma~\ref{Lem:ninftyBounds}. If $d=2$, then $c(x)\ge-\frac M{2\pi}\log(1+|x|)-\frac1{2\pi}\ird{n\,\log(1+|y|)}$ for the same reason. In both cases, $e^{-c-\phi}\in\mathrm L^1(\R^d)$ by~\eqref{Hyp-C1} and~\eqref{Hyp-C4}. Let $Z:=\ird{e^{-c-\phi}}$.

Let $m\in\mathcal X$. By Lemma~\ref{Lem:Positivity}, the interaction energy is a nonnegative quadratic form on the differences of elements of $\mathcal X$ with finite interaction energies. If the interaction energy of $m$ is infinite, the inequality below is trivial. Otherwise,
\[
\frac12\ird{m\,(-\Delta)^{-1}m}-\ird{m\,c}+\frac12\ird{n\,c}=\frac12\ird{(m-n)\,(-\Delta)^{-1}(m-n)}\ge0\,,
\]
where we used $\ird{m\,c}=\ird{n\,(-\Delta)^{-1}m}$. Hence $\mathcal F[m]\ge\ird{m\,\log m}+\ird{m\,(\phi+c)}-\frac12\ird{n\,c}$. By Jensen's inequality, $\ird{m\,\log(m/G_n)}\ge0$, that is,
\[
\ird{m\,\log m}+\ird{m\,(\phi+c)}\ge M\,\log M-M\,\log Z\,.
\]
With $m=n_\infty$, we obtain $\mathcal F[n_\infty]\ge M\,\log M-M\,\log Z-\frac12\ird{n\,c}$. On the other hand, $\mathcal F[n]=\ird{n\,\log n}+\ird{n\,\phi}+\frac12\ird{n\,c}$, so that
\[
\mathcal F[n]-\mathcal F[n_\infty]\le\ird{n\,\log n}+\ird{n\,(\phi+c)}+M\,\log Z-M\,\log M=\ird{n\,\log\(\frac n{G_n}\)}\,.\qedhere
\]
\end{proof}

\begin{lem}\label{Lem:LSIperturb}
Let $V\in\mathrm W^{2,\infty}_{\rm loc}(\R^d)$ be such that $D^2V\ge\kappa\,\mathrm{Id}$ a.e. for some $\kappa>0$, and let $W$ be a Lipschitz function with $\nrm{\nabla W}{\mathrm L^\infty}\le L$. Let $\nu:=e^{-V-W}/\ird{e^{-V-W}}$. Then there exists a constant $K=K(\kappa,L,d)>0$ such that
\be{LSI}
\ird{g\,\left|\nabla\log\(\frac g{M\,\nu}\)\right|^2}\ge2\,K\ird{g\,\log\(\frac g{M\,\nu}\)}
\ee
for any nonnegative $g$ with $\ird g=M$. One can take $K=\frac\kappa2\,\exp\big(-\frac{4\,A_d\,L^2}\kappa\big)$, where $A_d$ depends only on $d$.
\end{lem}
\begin{proof}
By homogeneity, we may assume that $M=1$. Let $\varrho\in C_c^\infty(B_1)$ be a nonnegative mollifier with $\ird\varrho=1$, let $\varrho_\eps:=\eps^{-d}\varrho(\cdot/\eps)$, and let $A_d:=\nrm{\nabla\varrho}{\mathrm L^1}$. We split $W=W_1+W_2$ with $W_1:=W*\varrho_\eps$ and $W_2:=W-W*\varrho_\eps$. Then $|W_2|\le L\,\eps$, and $D^2W_1=\nabla W*\nabla\varrho_\eps$ satisfies $|D^2W_1|\le A_d\,L/\eps$. With $\eps:=2\,A_d\,L/\kappa$, we get $D^2(V+W_1)\ge\frac\kappa2\,\mathrm{Id}$. By the Bakry--\'Emery criterion, $e^{-V-W_1}$ satisfies~\eqref{LSI} with constant $\kappa/2$. The perturbation $W_2$ is bounded with $\mathrm{osc}\,W_2\le2\,L\,\eps$. By the Holley--Stroock lemma, $\nu$ satisfies~\eqref{LSI} with constant $\frac\kappa2\,e^{-2L\eps}=\frac\kappa2\,e^{-4A_dL^2/\kappa}$.
\end{proof}

\begin{prop}\label{Prop:FreeEnergyDecay}
Let $d=2$ or $3$ and assume that $\phi$ satisfies~\hyperlink{C}{$(\mathrm C)$} and $D^2\phi\ge\kappa\,\mathrm{Id}$ on $\R^d$ for some $\kappa>0$. Let $n$ be a solution of~\eqref{Eq:(5.2)} in the sense of Definition~\ref{Def:Solution}, with $\mathcal F[n_0]<\infty$. Let $T$ and $\mathcal C:=\mathcal C_\infty$ be as in Corollary~\ref{Cor:gradc}, and $K:=K(\kappa,\mathcal C,d)$ as in Lemma~\ref{Lem:LSIperturb}. Then, for any $t\ge T$,
\be{FreeEnergyDecay}
\mathcal F[n(t,\cdot)]-\mathcal F[n_\infty]\le\big(\mathcal F[n(T,\cdot)]-\mathcal F[n_\infty]\big)\,e^{-2K(t-T)}\,.
\ee
As a consequence, for any $t\ge T$,
\[
\nrm{n(t,\cdot)-n_\infty}{\mathrm L^1(\R^d)}^2\le2\,M\,\big(\mathcal F[n(t,\cdot)]-\mathcal F[n_\infty]\big)\,,\qquad\nrm{\nabla c(t,\cdot)-\nabla c_\infty}{\mathrm L^2(\R^d)}^2\le2\,\big(\mathcal F[n(t,\cdot)]-\mathcal F[n_\infty]\big)\,.
\]
\end{prop}
\begin{proof}
For $t\ge T$, we have $\nrm{\nabla c(t,\cdot)}{\mathrm L^\infty}\le\mathcal C$ by Corollary~\ref{Cor:gradc}, and $n(t,\cdot)\in\mathrm L^\infty(\R^d)$ by Lemma~\ref{lem2.1}. We apply Lemma~\ref{Lem:LSIperturb} with $V=\phi$, $W=c(t,\cdot)$ and $g=n(t,\cdot)$, so that $M\,\nu=G_{n(t,\cdot)}$. By Lemma~\ref{Lem:FvsGibbs},
\[
\mathcal I[n(t,\cdot)]=\ird{n\,\left|\nabla\log\(\frac n{G_n}\)\right|^2}\ge2\,K\ird{n\,\log\(\frac n{G_n}\)}\ge2\,K\,\big(\mathcal F[n(t,\cdot)]-\mathcal F[n_\infty]\big)\,.
\]
Let $y(t):=\mathcal F[n(t,\cdot)]-\mathcal F[n_\infty]\ge0$. By Property~(iii) of Definition~\ref{Def:Solution}, $y(t)-y(s)\le-\int_s^t\mathcal I[n(r,\cdot)]\,dr\le-2\,K\int_s^ty(r)\,dr$ for any $T\le s\le t$. The function $y$ is nonincreasing, so that $\int_s^ty(r)\,dr\ge(t-s)\,y(t)$ and $y(t)\le y(s)/\big(1+2\,K\,(t-s)\big)$. We apply this inequality on the $N$ intervals of a uniform subdivision of $[T,t]$, and obtain $y(t)\le y(T)\,\big(1+2\,K\,(t-T)/N\big)^{-N}$. Letting $N\to+\infty$ proves~\eqref{FreeEnergyDecay}.

By~\eqref{Fdiff} with $\bar n=n_\infty$, we have $\ird{n\,\log(n/n_\infty)}\le y(t)$ and $\frac12\nrm{\nabla c-\nabla c_\infty}{\mathrm L^2}^2\le y(t)$. The Csisz\'ar--Kullback inequality for nonnegative functions with the same mass $M$,
\[
\ird{n\,\log\(\frac n{n_\infty}\)}\ge\frac1{2\,M}\,\nrm{n-n_\infty}{\mathrm L^1(\R^d)}^2\,,
\]
completes the proof.
\end{proof}

\begin{remark}
The rate $2\,K$ of Proposition~\ref{Prop:FreeEnergyDecay} is explicit and holds for all $t\ge T$, but it is not optimal. It depends on the bound $\mathcal C$ on $\nabla c$, which is uniform but not small. For the harmonic potential $\phi=\frac\mu2\,|x|^2$, we have $\kappa=\mu$, and the optimal asymptotic rate of $Q_1$ is $2\,\mu$ by Theorem~\ref{Thm:Sharp}. Since $\mathcal F[n]-\mathcal F[n_\infty]=\frac12\,Q_1[f]+o(Q_1[f])$ as $n\to n_\infty$, Proposition~\ref{Prop:FreeEnergyDecay} and Theorem~\ref{Thm:Sharp} are complementary: the first one gives an explicit global estimate, and the second one gives the sharp asymptotic rate. In dimension $d=2$, $c$ is unbounded, because $c(x)\sim-\frac M{2\pi}\log|x|$, so the Holley--Stroock lemma alone does not apply to $W=c$. This is why we use the Lipschitz perturbation argument of Lemma~\ref{Lem:LSIperturb}. For a general potential satisfying~\hyperlink{C}{$(\mathrm C)$}, only a Poincar\'e inequality is available. Proposition~\ref{Prop:FreeEnergyDecay} then requires an additional logarithmic Sobolev inequality for $e^{-\phi}$, for instance if $\phi$ is uniformly convex outside a ball.
\end{remark}

\subsection{Uniform convergence in \texorpdfstring{$\mathrm L^{\infty}$}{L-infinity} norm in the harmonic potential case}\label{Sec:Linfty}

The convergence of $n(t,\cdot)$ to $n_\infty$ in $\mathrm L^\infty(\R^d)$ was left open in Section~\ref{Sec:CVNernstPlanck}. As in the case of the Keller--Segel model, see~\cite{blanchet2010asymptotic}, better results can be achieved in the case of the harmonic potential. Let $\phi(x)=\frac\mu2\,|x|^2$ and
\[
\sigma(t):=\frac{1-e^{-2\,\mu\,t}}\mu\,,\qquad K(t,x,y):=\big(2\,\pi\,\sigma(t)\big)^{-\frac d2}\,\exp\(-\frac{|x-e^{-\mu t}\,y|^2}{2\,\sigma(t)}\)\,.
\]
The function $K$ is the Green function of the Fokker--Planck equation $\partial_tn=\Delta n+\mu\,\nabla\cdot(x\,n)$.
\begin{lem}\label{Lem:Duhamel} Let $d=2$ or $3$, $\phi(x)=\frac\mu2\,|x|^2$, and fix $p\in(d,\infty)$. For any solution $n$ of~\eqref{Eq:(5.2)}, there is a constant $C>0$ such that, for any $t\ge T$,
\begin{multline*}
\nrm{n(t+1,\cdot)-n_\infty}{\mathrm L^\infty(\R^d)}\le C\,\nrm{n(t,\cdot)-n_\infty}{\mathrm L^1(\R^d)}\\
+C\sup_{s\in[t,t+1]}\Big(\nrm{n(s,\cdot)-n_\infty}{\mathrm L^p(\R^d)}+\nrm{\nabla c(s,\cdot)-\nabla c_\infty}{\mathrm L^p(\R^d)}\Big)\,.
\end{multline*}
\end{lem}
\begin{proof}
By the Duhamel formula (see~\cite{MR3196188}) and the semi-group property, we have
\be{du1}
n(t+1,x)=\int_{\R^d}K(1,x,y)\,n(t,y)\,dy-\int_t^{t+1}\int_{\R^d}\nabla_yK(t+1-s,x,y)\cdot n(s,y)\,\nabla c(s,y)\,dy\,ds\,.
\ee
The stationary solution $n_\infty$ is a fixed point of the evolution, that is,
\be{du2}
n_{\infty}(x)=\int_{\R^d}K(1,x,y)\,n_{\infty}(y)\,dy-\int_t^{t+1}\int_{\R^d}\nabla_yK(t+1-s,x,y)\cdot n_{\infty}(y)\,\nabla c_{\infty}(y)\,dy\,ds\,.
\ee
By taking the difference of~\eqref{du1} and~\eqref{du2}, and writing
\[
n\,\nabla c-n_\infty\,\nabla c_\infty=(n-n_\infty)\,\nabla c_\infty+n\,(\nabla c-\nabla c_\infty)\,,
\]
we obtain with $r=p/(p-1)$ that
\begin{multline*}
\nrm{n(t+1,\cdot)-n_{\infty}}{\mathrm L^{\infty}}\le \sup_{x\in\R^d}\nrm{K(1,x,\cdot)}{\mathrm L^{\infty}}\,\nrm{n(t,\cdot)-n_{\infty}}{\mathrm L^1}\\
+\int_0^{1}\sup_{x\in\R^d}\nrm{\nabla_y K(s,x,\cdot)}{\mathrm L^r}\,ds\;\sup_{s\in[t,t+1]}\Big(\nrm{\nabla c_\infty}{\mathrm L^\infty}\nrm{n(s,\cdot)-n_\infty}{\mathrm L^p}+\nrm{n(s,\cdot)}{\mathrm L^\infty}\nrm{\nabla c(s,\cdot)-\nabla c_\infty}{\mathrm L^p}\Big)\,.
\end{multline*}
Here $\nrm{K(1,x,\cdot)}{\mathrm L^\infty}\le(2\,\pi\,\sigma(1))^{-d/2}$. Since $\nabla_yK(s,x,y)=e^{-\mu s}\,\sigma(s)^{-1}\,(x-e^{-\mu s}y)\,K(s,x,y)$, the change of variables $z=x-e^{-\mu s}y$ gives
\[
\nrm{\nabla_y K(s,x,\cdot)}{\mathrm L^r}=\kappa(r)\,e^{\mu\,s\,(\frac dr-1)}\,\sigma(s)^{-\frac12-\frac d2+\frac d{2r}}\,,
\]
where $\kappa(r)$ depends only on $r$ and $d$. As $s\to0_+$, $\sigma(s)\sim2\,s$, so this quantity is integrable on $(0,1)$ if and only if $\frac{d+1}2-\frac d{2r}<1$, that is, $r<\frac d{d-1}$, \emph{i.e.}, $p>d$. We conclude using Lemma~\ref{lem2.1} and Proposition~\ref{Prop:cinfty}.
\end{proof}
\begin{prop}
\label{propinf}
Let $d=2$ or $3$ and $\phi=\frac{\mu}{2}\,|x|^2$ for some $\mu>0$. Then any solution $n$ of~\eqref{Eq:(5.2)} with $\mathcal F[n_0]<\infty$ is such that
\[
\lim_{t\to+\infty}\|n(t,\cdot)-n_{\infty}\|_{\mathrm L^{\infty}(\mathbb{R}^d)}=0\,.
\]
\end{prop}
\begin{proof}
This follows from Lemma~\ref{Lem:Duhamel} and Proposition~\ref{Thm:(2.1)} applied with $p\in(d,\infty)$ and $q=p$.
\end{proof}

\section{Coercivity of the quadratic forms}\label{Sec:Coercivity-NP}

In this section, we study the quadratic forms associated with the free energy $\mathcal{F}$ and the Fisher information $\mathcal{I}$ when we Taylor expand these functionals around the stationary solution $(n_{\infty}, c_{\infty})$ defined by~\eqref{st}.

\subsection{The quadratic forms}

Let
\[
\mathcal A:=\Big\{f\in\mathrm L^2(\R^d,\,n_\infty\,dx)\,:\, \ird{f\,n_\infty}=0\Big\}\,.
\]
For $f\in\cA$, we define
\be{Def:hGu}
h:=f\,n_\infty\,,\qquad G:=(-\Delta)^{-1}h=G_d*h\,,\qquad u:=f+G\,.
\ee
Notice that $h\in\mathrm L^1\cap\mathrm L^2(\R^d)$, because $\ird{|h|}\le\sqrt M\,\nrm{f\,\sqrt{n_\infty}}{\mathrm L^2}$ and $\ird{h^2}\le\nrm{n_\infty}{\mathrm L^\infty}\ird{f^2\,n_\infty}$. If $d=2$, we also have
\be{FirstMoment}
\irtwo{|h|\,|x|}\le\nrm{f\,\sqrt{n_\infty}}{\mathrm L^2}\(\irtwo{|x|^2\,n_\infty}\)^{1/2}\,,
\ee
where $\irtwo{|x|^2\,n_\infty}<\infty$ because $n_\infty(x)=O(|x|^{-4-\eps})$ by~\eqref{infcon} and~\eqref{Hyp-C4}.
\begin{lem}\label{Lem:GradGL2}
Let $d=2$ or $3$ and assume that $\phi$ satisfies~\hyperlink{C}{$(\mathrm C)$}. There exists a constant $C_G>0$ such that, for any $f\in\cA$, $G=G_d*(f\,n_\infty)$ satisfies $\nabla G\in\mathrm L^2(\R^d)$,
\be{GradGL2}
\nrm{\nabla G}{\mathrm L^2(\R^d)}\le C_G\,\nrm{f\,\sqrt{n_\infty}}{\mathrm L^2(\R^d)}\,,
\ee
and $\ird{h\,G}=\ird{|\nabla G|^2}$.
\end{lem}
\begin{proof}
Let $\hat h(\xi):=\ird{e^{-i\,x\cdot\xi}\,h(x)}$. Since $h\in\mathrm L^1\cap\mathrm L^2(\R^d)$, $\hat h$ is continuous and bounded by $\nrm h{\mathrm L^1}$. We have $\widehat{\nabla G}(\xi)=i\,\xi\,\hat h(\xi)/|\xi|^2$ in the sense of distributions on $\R^d\setminus\{0\}$, and by Plancherel's formula,
\[
\ird{|\nabla G|^2}=\frac1{(2\pi)^d}\int_{\R^d}\frac{|\hat h(\xi)|^2}{|\xi|^2}\,d\xi
\le\frac1{(2\pi)^d}\int_{|\xi|<1}\frac{|\hat h(\xi)|^2}{|\xi|^2}\,d\xi+\ird{h^2}\,,
\]
provided the right-hand side is finite. If $d=3$, $|\xi|^{-2}$ is integrable on $B_1$, and the first term is bounded by $\frac{4\pi}{(2\pi)^3}\,\nrm h{\mathrm L^1}^2$. If $d=2$, we use $\irtwo h=0$ to write $\hat h(\xi)=\irtwo{(e^{-i\,x\cdot\xi}-1)\,h(x)}$, so that $|\hat h(\xi)|\le|\xi|\irtwo{|x|\,|h|}$, and the first term is bounded by $\frac1{4\pi}\big(\irtwo{|x|\,|h|}\big)^2$. With $\nrm h{\mathrm L^1}\le\sqrt M\,\nrm{f\sqrt{n_\infty}}{\mathrm L^2}$, $\nrm h{\mathrm L^2}^2\le\nrm{n_\infty}{\mathrm L^\infty}\nrm{f\sqrt{n_\infty}}{\mathrm L^2}^2$ and~\eqref{FirstMoment}, this proves~\eqref{GradGL2}. The same Fourier representation gives $\ird{h\,G}=(2\pi)^{-d}\int_{\R^d}|\hat h|^2\,|\xi|^{-2}\,d\xi=\ird{|\nabla G|^2}$, where, if $d=2$, the additive constant in $G$ does not contribute because $\irtwo h=0$.
\end{proof}
In particular, $Q_1[f]\le\big(1+C_G^2\big)\ird{f^2\,n_\infty}$ for any $f\in\cA$, so that $Q_1$ is continuous on $\cA$ for the norm of $\mathrm L^2(\R^d,n_\infty\,dx)$. In dimension $d=2$, the potential $G$ is defined up to an additive constant, which plays no role in what follows. Since $\log n_\infty+\phi+c_\infty$ is constant and $\ird{f\,n_\infty}=0$, the first order terms vanish in the expansions of $\mathcal F[n_\infty(1+\eps f)]$ and $\mathcal I[n_\infty(1+\eps f)]$. We define
\be{Q1}
Q_1[f]:=\lim_{\varepsilon\to0}\frac{2}{\varepsilon^{2}}\,\Big(\mathcal F[n_\infty\,(1+\varepsilon\,f)]-\mathcal F[n_\infty]\Big)=\ird{f^2\,n_{\infty}}+\ird{|\nabla G|^2}=\ird{h\,u}\,,
\ee
\be{Q2}
Q_2[f]:=\lim_{\varepsilon\to0}\frac{1}{\varepsilon^{2}}\,\mathcal I[n_\infty\,(1+\varepsilon\,f)]=\ird{|\nabla u|^2\,n_{\infty}}\,.
\ee
The normalizations are chosen so that $\frac d{dt}Q_1[f]=-2\,Q_2[f]$ for the linearized evolution, see Section~\ref{Sec:LargeTime}. The limits in~\eqref{Q1} and~\eqref{Q2} are formal if $f$ is not bounded from below. In the sequel, $Q_1[f]$ and $Q_2[f]$ are defined for any $f\in\cA$ by the right-hand sides of~\eqref{Q1} and~\eqref{Q2}, and $Q_2[f]$ takes values in $[0,+\infty]$.

\subsection{A spectral gap inequality}

\begin{lem}\label{Lem:PoincareNinfty} Let $d=2$ or $d=3$ and assume that the potential $\phi$ satisfies~\hyperlink{C}{$(\mathrm C)$}. Then there exists a positive constant $\mathcal C_P$ such that
\begin{equation}
\label{Poincare}
\ird{|\nabla w|^2\,n_\infty}\ge\mathcal C_P\ird{w^2\,n_\infty}\quad
\forall\,w\in\mathrm H^1(\R^d,n_\infty\,dx)\quad \mbox{such that}\quad\ird{w\,n_\infty}=0\,.
\end{equation}
\end{lem}
\begin{proof} We have $n_\infty=e^{-V}$ with $V:=\phi+c_\infty-\lambda$. We check that $V$ satisfies~\eqref{Hyp-C3}, and then apply the same Persson-type criterion as for $e^{-\phi}$, see~\cite[Section~3.2]{addala}. By Proposition~\ref{Prop:cinfty} and~\eqref{infcon}, $\nabla c_\infty$ is bounded and $|\nabla c_\infty(x)|\to0$ as $|x|\to+\infty$. Moreover, $-\Delta c_\infty=n_\infty\ge0$. For any $\eta\in(0,1)$, we have $|\nabla V|^2\ge(1-\eta)\,|\nabla\phi|^2-\eta^{-1}\,|\nabla c_\infty|^2$, and $\Delta\phi\le d\,\nrm{D^2\phi}{\mathrm L^\infty}=:C_\phi$. Hence
\begin{multline*}
\frac14\,|\nabla V|^2-\frac12\,\Delta V\\
\ge(1-\eta)\(\frac14\,|\nabla\phi|^2-\frac12\,\Delta\phi\)-\frac\eta2\,C_\phi-\frac1{4\,\eta}\,|\nabla c_\infty|^2+\frac12\,n_\infty\,.
\end{multline*}
This gives
\[
\liminf_{|x|\to\infty}\(\frac14\,|\nabla V|^2-\frac12\,\Delta V\)\ge(1-\eta)\,\sigma_\phi-\frac\eta2\,C_\phi>0
\]
for $\eta$ small enough. Similarly, $\liminf_{|x|\to\infty}|\nabla V|\ge\liminf_{|x|\to\infty}|\nabla\phi|>0$.
\end{proof}

Notice that a bounded perturbation argument does not apply in dimension $d=2$, because $c_\infty(x)\sim-\frac M{2\pi}\log|x|$ is unbounded. This is why we check~\eqref{Hyp-C3} directly for $V$.
\begin{prop}\label{Prop:Coercivity} Let $d=2$ or $d=3$ and assume that the potential $\phi$ satisfies~\hyperlink{C}{$(\mathrm C)$}. Then for any $f\in\cA$ such that $Q_2[f]<\infty$, we have
\[
Q_2[f]\ge\mathcal C_P\,Q_1[f]\,.
\]
\end{prop}
In the sequel, $\Cs\ge\mathcal C_P$ denotes the optimal constant in the inequality $Q_2\ge\Cs\,Q_1$ on $\cA$.
\begin{proof}
We apply~\eqref{Poincare} to $w=u-\frac1M\ird{G\,n_\infty}$, so that $\ird{w\,n_\infty}=0$ because $\ird{f\,n_{\infty}}=0$. We obtain
\begin{align*}
Q_2[f]&=\ird{|\nabla u|^2\,n_{\infty}}\ge\mathcal C_P\ird{u^2\,n_{\infty}}-\frac{\mathcal C_P}{M}\left(\ird{G\,n_{\infty}}\right)^2\\
&=\mathcal C_P\,Q_1[f]+\mathcal C_P\ird{h\,G}+\mathcal C_P\left[\ird{G^2\,n_{\infty}}-\frac{1}{M}\left(\ird{G\,n_{\infty}}\right)^2\right]\,,
\end{align*}
where we used $\ird{u^2\,n_\infty}=\ird{f^2\,n_\infty}+2\ird{h\,G}+\ird{G^2\,n_\infty}$ and $Q_1[f]=\ird{f^2\,n_\infty}+\ird{h\,G}$. The last bracket is nonnegative by the Cauchy--Schwarz inequality, since $\ird{n_\infty}=M$. Moreover, $\ird{h\,G}=\ird{|\nabla G|^2}\ge0$. Altogether, we conclude that
\[
Q_2[f]\ge \mathcal C_P\,Q_1[f]+\mathcal C_P\ird{|\nabla G|^2} \ge \mathcal C_P\,Q_1[f]\,.\qedhere
\]
\end{proof}

\subsection{Optimal spectral gap for the harmonic potential}\label{Sec:Sharp}

In this section, we assume that $d=2$ or $3$ and $\phi(x)=\frac\mu2\,|x|^2$ for some $\mu>0$, and prove Theorem~\ref{Thm:Sharp}. We write $n=n_\infty$ for brevity. This function is radially symmetric, and we use the notation $r=|x|$ and
\be{notation}
m(r):=\int_0^r n(\rho)\,\rho^{d-1}\,d\rho\,,\qquad
c_\infty'(r)=-\frac{m(r)}{r^{d-1}}\,,\qquad
\nu(r):=-\frac{n'(r)}{n(r)}=\mu\,r-\frac{m(r)}{r^{d-1}}\,.
\ee
The second identity is the radial Poisson equation $-(r^{d-1}c_\infty')'=r^{d-1}\,n$, and the third one is the equation $n'=-n\,(\phi'+c_\infty')$. Since $m'=r^{d-1}\,n$, we have
\be{nuprime}
\nu'(r)=\mu-n(r)+(d-1)\,\frac{m(r)}{r^d}\,.
\ee

\subsubsection{Monotonicity of the stationary solution}

\begin{lem}\label{Lem:Monotone}
We have $n(0)<d\,\mu$, $n$ is strictly decreasing on $(0,\infty)$, and $m(r)<\mu\,r^d$ for any $r>0$. As a consequence,
\be{mbounds}
\frac{n(r)}d\le\frac{m(r)}{r^d}\le\frac{n(0)}d=\mu-\delta_0<\mu\quad\forall\,r>0\,,\qquad\delta_0:=\mu-\frac{n(0)}d>0\,.
\ee
\end{lem}
\begin{proof}
Let $y(r):=\mu\,r^d-m(r)$. Then $y'=r^{d-1}\,(d\,\mu-n)$ and, by~\eqref{notation}, $n'=-n\,y/r^{d-1}$. Assume that $n(0)=d\,\mu$ and let $N:=n-d\,\mu$, so that $N'=-n\,y\,r^{1-d}$, $y'=-r^{d-1}\,N$ and $N(0)=y(0)=0$. With $\sigma(r):=\sup_{[0,r]}|N|$, we get $|y(\rho)|\le\sigma(r)\,\rho^d/d$ for any $\rho\le r$, hence
\[
|N(r)|\le\nrm n{\mathrm L^\infty}\int_0^r|y(\rho)|\,\rho^{1-d}\,d\rho\le\frac{\nrm n{\mathrm L^\infty}\,r^2}{2\,d}\,\sigma(r)\,.
\]
Since the right-hand side is nondecreasing in $r$, this gives $\sigma(r)\le\frac12\,\sigma(r)$ if $\nrm n{\mathrm L^\infty}\,r^2\le d$, that is, $N\equiv0$ and $y\equiv0$ on some interval $[0,r_0]$ with $r_0>0$. For $r\ge r_0$, the system is not singular and the Cauchy--Lipschitz theorem gives $n\equiv d\,\mu$ on $(0,\infty)$, which has infinite mass. If $n(0)>d\,\mu$, then $y<0$ near $0$, hence $n'>0$. Then $n$ increases and stays above $d\,\mu$, so $y'<0$ and $y$ stays negative. Thus $n$ is increasing on $(0,\infty)$, which again contradicts integrability. Hence $n(0)<d\,\mu$, and $y(r)\sim\big(\mu-n(0)/d\big)\,r^d>0$ near $0$. If $r_1$ were the first positive zero of $y$, then $n'<0$ on $(0,r_1)$, hence $n<n(0)<d\,\mu$ and $y'>0$ on $(0,r_1]$, so $y(r_1)>0$, a contradiction. Hence $y>0$ and $n'<0$ on $(0,\infty)$. The bounds~\eqref{mbounds} follow from $n(r)\le n(\rho)\le n(0)$ for any $\rho\le r$, and $\delta_0>0$ because $n(0)<d\,\mu$.
\end{proof}
Here $d\,\mu=\Delta\phi$ is the density of the uniform ``background charge'' which neutralizes the confinement. In the sequel, we mostly use~\eqref{mbounds} in the form $m(r)/r^d\le\mu-\delta_0$.

\subsubsection{A duality criterion}

\begin{lem}\label{Lem:Flux}
Let $f\in\cA$ with $h=f\,n$ and $u$ as in~\eqref{Def:hGu}, and let $\kappa>0$. Assume that there exists a compactly supported vector field $\bmm\in\mathrm L^2(\R^d;\R^d)$ such that
\[
-\nabla\cdot\bmm=h\quad\mbox{in}\quad\mathcal D'(\R^d)\qquad\mbox{and}\qquad \kappa\ird{\frac{|\bmm|^2}{n}}\le Q_1[f]\,.
\]
Then $Q_2[f]\ge\kappa\,Q_1[f]$.
\end{lem}
\begin{proof}
If $Q_2[f]=+\infty$, there is nothing to prove. Otherwise, $\nabla u\in\mathrm L^2(\R^d,n\,dx)$, hence $u\in\mathrm H^1_{\rm loc}(\R^d)$ because $n$ is locally bounded away from $0$. Since $\bmm$ has compact support, we get from~\eqref{Q1} and the Cauchy--Schwarz inequality that
\[
Q_1[f]=\ird{h\,u}=\ird{\bmm\cdot\nabla u}\le\(\ird{\frac{|\bmm|^2}{n}}\)^{1/2}\(\ird{n\,|\nabla u|^2}\)^{1/2}\le\(\frac{Q_1[f]}{\kappa}\)^{1/2}Q_2[f]^{1/2}\,.
\]
The conclusion follows by squaring this inequality.
\end{proof}
The criterion is sharp. The optimal flux is $\bmm=n\,\nabla w$ with $-\nabla\cdot(n\,\nabla w)=h$, and the optimal constant in $Q_2\ge\kappa\,Q_1$ is $\inf_h\scalar h{S\,h}/\scalar h{A^{-1}h}$ with $S=n^{-1}+(-\Delta)^{-1}$ and $A=-\nabla\cdot(n\,\nabla\,\cdot)$. Lemma~\ref{Lem:Flux} uses neither the symmetry of the linearized operator nor any spectral theory.

\subsubsection{Translation modes}

\begin{lem}\label{Lem:Upper}
For any $i=1,\ldots,d$, the function $f_i:=-\partial_{x_i}n/n$ belongs to $\cA$ and $Q_2[f_i]=\mu\,Q_1[f_i]$. As a consequence, $\Cs\le\mu$.
\end{lem}
\begin{proof}
We have $h_i=-\partial_{x_i}n$, $\ird{h_i}=0$, and $G_i=(-\Delta)^{-1}(-\partial_{x_i}n)=-\partial_{x_i}c_\infty$. Since $\nabla\log n=-\mu\,x-\nabla c_\infty$, we find that $f_i=\mu\,x_i+\partial_{x_i}c_\infty$ and $u_i=f_i+G_i=\mu\,x_i$. Hence $Q_2[f_i]=\mu^2\ird n=\mu^2\,M$ and $Q_1[f_i]=\ird{h_i\,u_i}=-\mu\ird{x_i\,\partial_{x_i}n}=\mu\,M$.
\end{proof}
A direct computation also shows that $-\cL f_i=\mu\,f_i$, where $\cL$ is the linearized operator of Section~\ref{Sec:LargeTime}.

\subsubsection{Angular decomposition and chiral fluxes}

Let $(Y_{\ell j})$ be a real orthonormal basis of $\mathrm L^2(\Sph^{d-1})$ made of spherical harmonics, with $-\Delta_{\Sph}Y_{\ell j}=L_\ell\,Y_{\ell j}$ and
\[
L_\ell:=\ell\,(\ell+d-2)\,,\qquad \int_{\Sph^{d-1}}|\nS Y_{\ell j}|^2\,d\sigma=L_\ell\,.
\]
For $d=2$, these are $(2\pi)^{-1/2}$, $\pi^{-1/2}\cos(\ell\,\theta)$ and $\pi^{-1/2}\sin(\ell\,\theta)$, with $L_\ell=\ell^2$. For a radial function $g$, a radial function $F$, and a spherical harmonic $Y$ of degree $\ell$, we have
\begin{gather}
\nabla(g\,Y)=g'\,Y\,\er+\frac{g}{r}\,\nS Y\,,\qquad
\Delta(g\,Y)=\(g''+\frac{d-1}r\,g'-\frac{L_\ell}{r^2}\,g\)Y\,,\label{grad}\\
\nabla\cdot\big(F\,\nS Y\big)=\frac{F}{r}\,\Delta_{\Sph}Y=-\frac{L_\ell\,F}{r}\,Y\,.\label{divtan}
\end{gather}
Here $\er:=x/|x|$, and~\eqref{divtan} follows from~\eqref{grad} with $g=r\,F$, by subtracting $\nabla\cdot(g'\,Y\,\er)=r^{1-d}(r^{d-1}g')'\,Y$. Moreover, $\er\cdot\nS Y=0$.

Let $f\in\cA$ be such that $h=f\,n\in C_c^\infty(\R^d)$, with $\mathrm{supp}\,h\subset B_R$. We decompose
\[
h=\sum_{\ell,j}h_{\ell j}(r)\,Y_{\ell j}\,,\qquad G=\sum_{\ell,j}g_{\ell j}(r)\,Y_{\ell j}\,,\qquad
h_{\ell j}=-g_{\ell j}''-\frac{d-1}r\,g_{\ell j}'+\frac{L_\ell}{r^2}\,g_{\ell j}\,.
\]
For $r>R$, $g_{\ell j}(r)=a_{\ell j}\,r^{-(\ell+d-2)}$ if $\ell\ge1$, and $r^{d-1}g_{00}'(r)=0$ because $\ird h=0$. As $r\to0$, we have $g_{\ell j}(r)=O(r^\ell)$. Since $n$ is radial, the quadratic form $Q_1$ splits into orthogonal components:
\be{Q1split}
Q_1[f]=\sum_{\ell,j}Q_1^{(\ell j)}\,,\qquad Q_1^{(\ell j)}:=\int_0^\infty\(\frac{h_{\ell j}^2}{n}+g_{\ell j}'^2+\frac{L_\ell}{r^2}\,g_{\ell j}^2\)r^{d-1}\,dr\,.
\ee

We look for a flux whose radial and tangential components are proportional. Fix $\ell\ge1$ and one component $G_\ell=g(r)\,Y(\hat x)$, with $h_\ell=-\Delta G_\ell$. By~\eqref{grad}--\eqref{divtan}, a vector field $A(r)\,Y\,\er+\frac{C(r)}r\,\nS Y$ satisfies $-\nabla\cdot\bmm=h_\ell$ if and only if
\[
r^{1-d}\big(r^{d-1}(A-g')\big)'=\frac{L_\ell}{r^2}\,(C-g)\,.
\]
With the ansatz $A=g'+\alpha\,g/r$, this gives $\frac Cr=\frac\alpha{L_\ell}\,g'+\big(1+\frac{\alpha\,(d-2)}{L_\ell}\big)\frac gr$. Then $C/r$ is proportional to $A$ if and only if $\alpha^2-(d-2)\,\alpha-L_\ell=0$, that is, $\alpha\in\{\ell+d-2,\,-\ell\}$. We choose the root $\alpha=\ell+d-2$, which is associated with the decaying harmonic function $r^{-(\ell+d-2)}$. Then $\alpha/L_\ell=1/\ell$.
\begin{defin}[Chiral flux]
For $\ell\ge1$ and $G_\ell=g(r)\,Y$, let
\be{chiral}
\omega:=g'+\frac{\ell+d-2}{r}\,g=r^{-(\ell+d-2)}\big(r^{\ell+d-2}\,g\big)'\,,\qquad
\bmm_\ell:=\omega(r)\(Y\,\er+\frac1\ell\,\nS Y\)\,.
\ee
\end{defin}
\begin{lem}\label{Lem:Chiral}
For any $\ell\ge1$, the following properties hold.
\begin{enumerate}[label=(\alph*)]
\item $\Delta G_\ell=\big(\omega'-\frac{\ell-1}{r}\,\omega\big)\,Y=r^{\ell-1}\big(r^{1-\ell}\,\omega\big)'\,Y$ and $-\nabla\cdot\bmm_\ell=h_\ell$.
\item $\omega\equiv0$ for $r>R$, so that $\bmm_\ell$ is supported in $\overline{B_R}$.
\item $\int_{\Sph^{d-1}}|\bmm_\ell|^2\,d\sigma=\kappa_\ell\,\omega^2$ with $\kappa_\ell:=1+\dfrac{L_\ell}{\ell^2}=\dfrac{2\,\ell+d-2}{\ell}$.
\item $\displaystyle\int_0^\infty\(g'^2+\frac{L_\ell}{r^2}\,g^2\)r^{d-1}\,dr=\int_0^\infty\omega^2\,r^{d-1}\,dr$.
\end{enumerate}
\end{lem}
\begin{proof}
Let $\alpha=\ell+d-2$. (a) A direct computation gives
\[
\omega'-\frac{\ell-1}{r}\,\omega=g''+(\alpha-\ell+1)\,\frac{g'}{r}-\alpha\,\ell\,\frac{g}{r^2}=g''+\frac{d-1}{r}\,g'-\frac{L_\ell}{r^2}\,g\,,
\]
which is the radial part of $\Delta G_\ell$ by~\eqref{grad}. By~\eqref{divtan},
\[
\nabla\cdot\bmm_\ell=\(r^{1-d}(r^{d-1}\omega)'-\frac{L_\ell}{\ell\,r}\,\omega\)Y=\(\omega'+\frac{d-1-\alpha}{r}\,\omega\)Y=\(\omega'-\frac{\ell-1}{r}\,\omega\)Y=\Delta G_\ell\,.
\]
This identity holds pointwise on $\R^d\setminus\{0\}$. Since $\bmm_\ell$ is bounded in a neighbourhood of the origin, it also holds in $\mathcal D'(\R^d)$. (b) For $r>R$, $r^{\alpha}\,g$ is constant. (c) This follows from $\er\perp\nS Y$, $\int_{\Sph^{d-1}}Y^2\,d\sigma=1$ and $\int_{\Sph^{d-1}}|\nS Y|^2\,d\sigma=L_\ell$. (d) By expanding the square, we get
\[
\int_0^\infty\omega^2\,r^{d-1}\,dr=\int_0^\infty g'^2\,r^{d-1}\,dr+\alpha\int_0^\infty(g^2)'\,r^{d-2}\,dr+\alpha^2\int_0^\infty g^2\,r^{d-3}\,dr\,.
\]
An integration by parts gives $\alpha\int_0^\infty(g^2)'\,r^{d-2}\,dr=-\alpha\,(d-2)\int_0^\infty g^2\,r^{d-3}\,dr$. The boundary terms $g^2\,r^{d-2}$ vanish at $r=0$ because $g=O(r^\ell)$ with $\ell\ge1$, and at $r=\infty$ because $g=O(r^{-\alpha})$. We conclude using $\alpha^2-\alpha\,(d-2)=L_\ell$.
\end{proof}
For $\ell=1$ and $Y\propto x_i/r$, we have $Y\,\er+\nS Y\propto e_i$, so that $\bmm_1=\omega(r)\,e_i$. For the translation mode of Lemma~\ref{Lem:Upper}, $g=m/r^{d-1}$ and $\omega=r^{1-d}\,m'=n$, so that $\bmm_1=n\,e_i$, which is the optimal flux of a rigid translation. For $d=2$, $\bmm_\ell=\nabla G_\ell-\nabla^\perp(\mathcal HG_\ell)$, where $\mathcal H$ is the Hilbert transform in the angular variable. This is why we call $\bmm_\ell$ a \emph{chiral} flux.

\subsubsection{Non-radial components: a one-dimensional inequality}

By Lemma~\ref{Lem:Chiral}, the condition $\kappa\ird{|\bmm_\ell|^2/n}\le Q_1^{(\ell j)}$ is equivalent to
\[
\int_0^\infty\(\frac{\big(r^{\ell-1}(r^{1-\ell}\omega)'\big)^2}{n}+\omega^2\)r^{d-1}\,dr\ge\kappa\,\kappa_\ell\int_0^\infty\frac{\omega^2}{n}\,r^{d-1}\,dr\,.
\]
We write $\omega=r^{\ell-1}\,\eta$ and $w_\ell(r):=r^{2\ell+d-3}$. The condition becomes
\be{Jl}
\int_0^\infty w_\ell\(\frac{\eta'^2}{n}+\eta^2\)dr\ge\kappa\,\kappa_\ell\int_0^\infty w_\ell\,\frac{\eta^2}{n}\,dr\,,
\ee
where $\eta$ is smooth on $[0,\infty)$ with compact support. Indeed, $g\sim b\,r^\ell$ near $r=0$, so that $\omega\sim(2\,\ell+d-2)\,b\,r^{\ell-1}$ and $\eta$ is bounded.
\begin{prop}[Ground state identity]\label{Prop:GS}
Let $\ell\ge1$, let $\eta$ be as above, and let $\chi:=\eta/n$. Then
\be{GSidentity}
\int_0^\infty w_\ell\(\frac{\eta'^2}{n}+\eta^2\)dr-\mu\,\kappa_\ell\int_0^\infty w_\ell\,\frac{\eta^2}{n}\,dr
=\int_0^\infty w_\ell\,n\,\chi'^2\,dr+(\ell-1)\int_0^\infty w_\ell\,n\,\beta_\ell\,\chi^2\,dr\,,
\ee
where
\[
\beta_\ell(r):=\frac{2\,\ell+d-2}{\ell}\,\mu-\frac{2\,m(r)}{r^d}\,.
\]
\end{prop}
\begin{proof}
Let $w=w_\ell$. Since $\eta'=n'\chi+n\,\chi'$, we have
\[
\int_0^\infty w\,\frac{\eta'^2}{n}\,dr=\int_0^\infty w\,n\,\chi'^2\,dr+\int_0^\infty w\,\frac{n'^2}{n}\,\chi^2\,dr+\int_0^\infty w\,n'\,(\chi^2)'\,dr\,.
\]
The boundary term $\big[w\,n'\chi^2\big]_0^\infty=-\big[w\,\nu\,\eta^2/n\big]_0^\infty$ vanishes, because $w\,\nu=O(r^{2\ell+d-2})$ as $r\to0$, $\eta$ is bounded, and $\eta$ has compact support. After an integration by parts, using $n'=-n\,\nu$, we get
\[
\int_0^\infty w\,\frac{n'^2}{n}\,\chi^2\,dr-\int_0^\infty(w\,n')'\,\chi^2\,dr=\int_0^\infty\big(w'\,\nu+w\,\nu'\big)\,n\,\chi^2\,dr\,.
\]
By~\eqref{notation},~\eqref{nuprime} and $w'=(2\,\ell+d-3)\,w/r$, we find
\begin{align*}
w'\,\nu+w\,\nu'&=w\left[(2\,\ell+d-3)\(\mu-\frac{m}{r^d}\)+\mu-n+(d-1)\,\frac{m}{r^d}\right]\\
&=w\left[(2\,\ell+d-2)\,\mu-n-(2\,\ell-2)\,\frac{m}{r^d}\right]\,.
\end{align*}
Adding $\int_0^\infty w\,\eta^2\,dr=\int_0^\infty w\,n^2\,\chi^2\,dr$ cancels the term $-n$. Then we subtract
\[
\mu\,\kappa_\ell\int_0^\infty w\,\frac{\eta^2}n\,dr=\mu\,\frac{2\,\ell+d-2}{\ell}\int_0^\infty w\,n\,\chi^2\,dr\,,
\]
and the coefficient of $w\,n\,\chi^2$ becomes
\[
(2\,\ell+d-2)\,\mu\(1-\frac1\ell\)-(2\,\ell-2)\,\frac{m}{r^d}=(\ell-1)\,\beta_\ell\,.\qedhere
\]
\end{proof}
\begin{cor}\label{Cor:NonRadial}
For any $\ell\ge1$, Inequality~\eqref{Jl} holds with $\kappa=\mu$. Moreover:
\begin{enumerate}[label=(\roman*)]
\item if $\ell=1$, the difference of the two sides of~\eqref{Jl} with $\kappa=\mu$ is $\int_0^\infty w_1\,n\,\chi'^2\,dr$ with $\chi=\eta/n$. It vanishes only if $\eta$ is proportional to $n$, which is not compactly supported: equality is never achieved by an admissible $\eta$, but it is approached by truncations of $n$, which correspond to the translation modes;
\item if $\ell\ge2$, Inequality~\eqref{Jl} holds with $\kappa=\mu_\ell$, where
\[
\mu_\ell\ge\mu+(\ell-1)\,\delta_0\quad\mbox{if}\quad d=2\,,\qquad\mu_\ell\ge\frac{3\,\ell\,\mu}{2\,\ell+1}+\frac{2\,\ell\,(\ell-1)}{2\,\ell+1}\,\delta_0\quad\mbox{if}\quad d=3\,.
\]
\end{enumerate}
\end{cor}
\begin{proof}
By Lemma~\ref{Lem:Monotone}, we have $2\,m/r^d\le2\,n(0)/d=2\,\mu-2\,\delta_0$, so that
\[
\beta_\ell\ge\frac{2\,\ell+d-2}\ell\,\mu-2\,\mu+2\,\delta_0=\frac{(d-2)\,\mu}\ell+2\,\delta_0>0\,.
\]
If $d=2$, then $\kappa_\ell=2$ and $\beta_\ell\ge2\,\delta_0$. If $d=3$, then $\kappa_\ell=\frac{2\ell+1}\ell$ and $\beta_\ell\ge\frac\mu\ell+2\,\delta_0$. Both terms of the right-hand side of~\eqref{GSidentity} are then nonnegative. If $\ell\ge2$, the second one is bounded from below by $(\ell-1)\,\inf\beta_\ell\int_0^\infty w_\ell\,\eta^2/n\,dr$. This gives~\eqref{Jl} with $\kappa=\mu+(\ell-1)\,\inf\beta_\ell/\kappa_\ell$, which is the claimed value of $\mu_\ell$ in both cases: if $d=3$, $\mu+\frac{(\ell-1)\,\ell}{2\ell+1}\big(\frac\mu\ell+2\,\delta_0\big)=\frac{3\,\ell\,\mu}{2\ell+1}+\frac{2\,\ell\,(\ell-1)}{2\ell+1}\,\delta_0$. If $\ell=1$, the identity reduces to $\int_0^\infty w_1\,n\,\chi'^2\,dr\ge0$, which vanishes if and only if $\chi$ is constant. When $\omega=n$, we have $r^{d-1}\,g=m$, \emph{i.e.}, $G_1=-\partial_{x_i}c_\infty$ for the corresponding direction, which is the translation mode of Lemma~\ref{Lem:Upper}.
\end{proof}

\subsubsection{The radial component}

For the radial component $G_0=g(r)\,Y_0$, we take the flux $\bmm_0:=\nabla G_0$, so that $-\nabla\cdot\bmm_0=h_0$. Let
\[
\Theta(r):=\int_0^rh_{00}(\rho)\,\rho^{d-1}\,d\rho=-r^{d-1}\,g'(r)\,.
\]
Since $\ird h=0$, the function $\Theta$ is supported in $[0,R]$, and $h_{00}=r^{1-d}\,\Theta'$. Then
\[
Q_1^{(00)}=\int_0^\infty\(\frac{\Theta'^2}{n\,r^{d-1}}+\frac{\Theta^2}{r^{d-1}}\)dr=:N_0[\Theta]\,,\qquad
\ird{\frac{|\bmm_0|^2}{n}}=\int_0^\infty\frac{\Theta^2}{n\,r^{d-1}}\,dr=:D_0[\Theta]\,.
\]
\begin{prop}\label{Prop:Radial}
For any $\Theta$ as above, $N_0[\Theta]\ge2\,\mu\,D_0[\Theta]$.
\end{prop}
\begin{proof}
Let $p:=(n\,r^{d-1})^{-1}$, $q:=r^{1-d}$ and $\Theta_0:=r^d\,n>0$. The latter corresponds to the dilation mode $\nabla\cdot(x\,n)$. Since $\Theta_0'=r^{d-1}\,n\,(d-r\,\nu)$, we deduce from~\eqref{notation} that
\[
p\,\Theta_0'=d-\mu\,r^2+\frac{m}{r^{d-2}}\,,\qquad
-(p\,\Theta_0')'=2\,\mu\,r-n\,r+(d-2)\,\frac{m}{r^{d-1}}\,,\qquad q\,\Theta_0=r\,n\,,
\]
hence
\be{supersol}
-(p\,\Theta_0')'+q\,\Theta_0=\(2\,\mu+(d-2)\,\frac{m}{r^d}\)p\,\Theta_0\,.
\ee
We write $\Theta=\Theta_0\,\chi$. Near $r=0$, $\Theta\sim h_{00}(0)\,r^d/d$, so that $\chi$ is bounded. The standard ground state computation gives
\[
\int_0^\infty\big(p\,\Theta'^2+q\,\Theta^2\big)\,dr=\int_0^\infty p\,\Theta_0^2\,\chi'^2\,dr+\int_0^\infty\big(-(p\,\Theta_0')'+q\,\Theta_0\big)\,\Theta_0\,\chi^2\,dr+\big[p\,\Theta_0\,\Theta_0'\,\chi^2\big]_0^\infty\,.
\]
The boundary term vanishes, because $p\,\Theta_0\,\Theta_0'=r^d\,n\,(d-r\,\nu)\to0$ as $r\to0$ and $\chi$ has compact support. By~\eqref{supersol},
\[
N_0[\Theta]-2\,\mu\,D_0[\Theta]=\int_0^\infty p\,\Theta_0^2\,\chi'^2\,dr+(d-2)\int_0^\infty\frac{m}{r^d}\,p\,\Theta_0^2\,\chi^2\,dr\ge0\,.\qedhere
\]
\end{proof}
If $d=2$, then~\eqref{supersol} is an exact eigenvalue equation. The function $\Theta_0=r^2\,n$ is not compactly supported, but it belongs to the closure of the admissible functions for the norm $(N_0+D_0)^{1/2}$, because of the Gaussian decay of $n$: if $\zeta_R$ is a radial cut-off function as in Step~2 of the proof of Theorem~\ref{Thm:Sharp}, then $N_0[\zeta_R\Theta_0]\to N_0[\Theta_0]$ and $D_0[\zeta_R\Theta_0]\to D_0[\Theta_0]$ as $R\to+\infty$. Hence $\inf N_0/D_0=2\,\mu$ on radial functions, that is, the radial spectral gap is exactly $2\,\mu$. This is consistent with the computation of~\cite{MR3196188} for the Keller--Segel model, and it reflects the closed virial identity of the two-dimensional Coulomb kernel. If $d=3$, the extra term $\int_0^\infty\frac m{r^3}\,p\,\Theta_0^2\,\chi^2\,dr$ is positive for any $\Theta\neq0$, so that $N_0[\Theta]>2\,\mu\,D_0[\Theta]$. This does not prove by itself that the infimum of $N_0/D_0$ is larger than $2\,\mu$, and we do not use such a property. The numerical computations of Remark~\ref{Rem:Numerics} indicate that the radial spectral gap is strictly larger than $2\,\mu$ in this case.

\subsubsection{Proof of Theorem~\ref{Thm:Sharp}}

\begin{proof}[Proof of Theorem~\ref{Thm:Sharp}]
\emph{Step 1: smooth perturbations.} Let $f\in\cA_0:=\big\{f\in\cA\,:\,f\,n\in C_c^\infty(\R^d)\big\}$, and define $\bmm:=\bmm_0+\sum_{\ell\ge1,j}\bmm_{\ell j}$, using $\bmm_0$ for $\ell=0$ and~\eqref{chiral} for $\ell\ge1$. All the $\bmm_{\ell j}$ are supported in $\overline{B_R}$. By Lemma~\ref{Lem:Chiral}~(c)--(d) and since $\kappa_\ell\le d$,
\[
\sum_{\ell,j}\ird{|\bmm_{\ell j}|^2}\le d\ird{|\nabla G|^2}<\infty\,,
\]
so that the series converges in $\mathrm L^2(\R^d)$ and $-\nabla\cdot\bmm=h$. The families $(Y_{\ell j})$ and $(\nS Y_{\ell j})$ are orthogonal on $\Sph^{d-1}$ and $n$ is radial, so that $\ird{|\bmm|^2/n}=\sum_{\ell,j}\ird{|\bmm_{\ell j}|^2/n}$. Proposition~\ref{Prop:Radial} and Corollary~\ref{Cor:NonRadial} give $\mu\ird{|\bmm_{\ell j}|^2/n}\le Q_1^{(\ell j)}$ for any $(\ell,j)$, hence $\mu\ird{|\bmm|^2/n}\le Q_1[f]$ by~\eqref{Q1split}. Lemma~\ref{Lem:Flux} then gives $Q_2[f]\ge\mu\,Q_1[f]$.

\emph{Step 2: density.} Let $f\in\cA$ with $Q_2[f]<\infty$. We denote by $\mathrm H^1_n$ the space of the functions $g\in\mathrm L^2(\R^d,n\,dx)$ such that $\nabla g\in\mathrm L^2(\R^d,n\,dx)$, with the norm $\nrm g{\mathrm H^1_n}^2:=\ird{(g^2+|\nabla g|^2)\,n}$. The proof is divided into four parts.

\smallskip\noindent\emph{(a) $f\in\mathrm H^1_n$.} By Lemma~\ref{Lem:GradGL2}, $\ird{n\,|\nabla G|^2}\le\nrm n{\mathrm L^\infty}\,C_G^2\ird{f^2\,n}$. Since $\nabla f=\nabla u-\nabla G$, we get
\[
\ird{n\,|\nabla f|^2}\le2\,Q_2[f]+2\,\nrm n{\mathrm L^\infty}\,C_G^2\ird{f^2\,n}<\infty\,.
\]

\smallskip\noindent\emph{(b) Continuity of $Q_1$ and $Q_2$.} For $f_1$, $f_2\in\cA\cap\mathrm H^1_n$, the maps $f\mapsto G$ and $f\mapsto u$ are linear. By Lemma~\ref{Lem:GradGL2},
\[
\big|Q_1[f_1]^{1/2}-Q_1[f_2]^{1/2}\big|\le Q_1[f_1-f_2]^{1/2}\le\big(1+C_G^2\big)^{1/2}\,\nrm{f_1-f_2}{\mathrm H^1_n}\,,
\]
\[
\big|Q_2[f_1]^{1/2}-Q_2[f_2]^{1/2}\big|\le Q_2[f_1-f_2]^{1/2}\le\nrm{\sqrt n\,\nabla(f_1-f_2)}{\mathrm L^2}+\nrm n{\mathrm L^\infty}^{1/2}\,C_G\,\nrm{f_1-f_2}{\mathrm H^1_n}\,,
\]
by the triangle inequality for the seminorms $Q_1^{1/2}$ and $Q_2^{1/2}$. Hence $Q_1$ and $Q_2$ are continuous on $\cA\cap\mathrm H^1_n$ for the norm of $\mathrm H^1_n$.

\smallskip\noindent\emph{(c) Approximation by smooth compactly supported functions.} Let $\zeta\in C_c^\infty(\R^d)$ be such that $0\le\zeta\le1$, $\zeta=1$ on $B_1$ and $\zeta=0$ outside $B_2$, and let $\zeta_R:=\zeta(\cdot/R)$. Then $\zeta_R\,f\to f$ in $\mathrm H^1_n$ as $R\to+\infty$. Indeed, $|\zeta_R\,f-f|\le|f|\,\ind_{|x|>R}$ and
\[
|\nabla(\zeta_R\,f)-\nabla f|\le|\nabla f|\,\ind_{|x|>R}+\frac{\nrm{\nabla\zeta}{\mathrm L^\infty}}R\,|f|\,,
\]
so the convergence follows by dominated convergence. For a fixed $R$, the function $F:=\zeta_R\,f$ is supported in $\overline{B_{2R}}$. Since $n$ is continuous and positive, there are constants $0<a_R\le b_R$ such that $a_R\le n\le b_R$ on $B_{2R+1}$. Hence $F\in\mathrm H^1(\R^d)$, with $\nrm F{\mathrm H^1}^2\le a_R^{-1}\nrm F{\mathrm H^1_n}^2$. Let $\varrho_\eps$ be a standard mollifier with support in $B_\eps$, $\eps\in(0,1)$. Then $\varrho_\eps*F\in C_c^\infty(B_{2R+1})$ converges to $F$ in $\mathrm H^1(\R^d)$ as $\eps\to0$, hence in $\mathrm H^1_n$, because $\nrm g{\mathrm H^1_n}^2\le b_R\,\nrm g{\mathrm H^1}^2$ for any $g$ supported in $B_{2R+1}$. By a diagonal argument, we obtain a sequence $(g_k)_{k\in\N}$ in $C_c^\infty(\R^d)$ such that $g_k\to f$ in $\mathrm H^1_n$.

\smallskip\noindent\emph{(d) Restoring the constraint.} Let $\psi_0\in C_c^\infty(\R^d)$ be such that $\ird{\psi_0}=1$, and let
\[
f_k:=g_k-c_k\,\frac{\psi_0}n\,,\qquad c_k:=\ird{g_k\,n}\,.
\]
Since $n$ is smooth and positive, $\psi_0/n\in C_c^\infty(\R^d)\subset\mathrm H^1_n$. Then $h_k=f_k\,n=g_k\,n-c_k\,\psi_0\in C_c^\infty(\R^d)$ and $\ird{h_k}=0$, that is, $f_k\in\cA_0$. Since $\ird{f\,n}=0$, the Cauchy--Schwarz inequality gives $|c_k|=\big|\ird{(g_k-f)\,n}\big|\le\sqrt M\,\nrm{g_k-f}{\mathrm L^2(n\,dx)}\to0$. Hence $f_k\to f$ in $\mathrm H^1_n$, and by~(b), $Q_1[f_k]\to Q_1[f]$ and $Q_2[f_k]\to Q_2[f]$. By Step~1, $Q_2[f_k]\ge\mu\,Q_1[f_k]$ for any $k$, and the inequality passes to the limit.

\emph{Step 3: optimality and equality case.} Optimality follows from Lemma~\ref{Lem:Upper}. For the equality case, we use the bilinear forms associated with $Q_1$ and $Q_2$, namely $\scalar{f}{g}=\ird{h_f\,u_g}$ and $B_2(f,g)=\ird{n\,\nabla u_f\cdot\nabla u_g}$. With the notations of Lemma~\ref{Lem:Upper}, we have $\nabla u_i=\mu\,e_i$, so that, for any $g\in\cA$,
\[
B_2(f_i,g)=\mu\ird{n\,\partial_{x_i}u_g}=-\mu\ird{\partial_{x_i}n\,u_g}=\mu\ird{h_i\,u_g}=\mu\,\scalar{f_i}{g}\,,\qquad\scalar{f_i}{g}=\mu\ird{x_i\,h_g}\,.
\]
Let $f\in\cA$ and decompose $f=f_\parallel+f_\perp$, where $f_\parallel\in\mathrm{span}\{f_1,\ldots,f_d\}$ and $\scalar{f_i}{f_\perp}=0$ for any $i$, \emph{i.e.}, $\ird{x\,f_\perp\,n}=0$. Then $Q_1[f]=Q_1[f_\parallel]+Q_1[f_\perp]$, and $B_2(f_\parallel,f_\perp)=0$. By Theorem~\ref{Thm:SecondGap}, which does not use the equality case, we get
\[
Q_2[f]=\mu\,Q_1[f_\parallel]+Q_2[f_\perp]\ge\mu\,Q_1[f]+(\Lambda_2-\mu)\,Q_1[f_\perp]\,.
\]
Since $\Lambda_2>\mu$, equality holds if and only if $f_\perp=0$.

\emph{Step 4: consequence for Theorem~\ref{Thm4.1}.} Since $\Cs=\mu$, Theorem~\ref{Thm4.1} gives the rate $e^{-2\mu t}$. This rate cannot be improved in general, because the center of mass satisfies $X(t)=X_0\,e^{-\mu t}$ (Lemma~\ref{Lem:CoM}). Hence $\ird{x\,(n(t,x)-n_\infty)}=M\,X_0\,e^{-\mu t}$, and $|M\,X_0|\,e^{-\mu t}\le\big(\ird{|x|^2\,n_\infty}\big)^{1/2}\big(\ird{|n-n_\infty|^2\,n_\infty^{-1}}\big)^{1/2}$.
\end{proof}

\subsection{The second spectral gap}\label{Sec:SecondGap}

\begin{thm}\label{Thm:SecondGap}
Let $d=2$ or $3$ and $\phi(x)=\frac\mu2\,|x|^2$. Let $\Lambda_2$ be the optimal constant in the inequality $Q_2[f]\ge\Lambda_2\,Q_1[f]$ for all $f\in\cA$ such that $\ird{x\,f\,n_\infty}=0$. Then
\[
\Lambda_2\ge\mu+\frac{\delta_0}d\,,\qquad\mbox{that is,}\qquad\Lambda_2\ge\mu+\frac{\delta_0}2\quad\mbox{if}\quad d=2\,,\qquad
\Lambda_2\ge\mu+\frac{\delta_0}3\quad\mbox{if}\quad d=3\,.
\]
\end{thm}
\begin{proof}
Since $x_i=r\,(x_i/r)$ and $x_i/r$ is a spherical harmonic of degree one, the condition $\ird{x\,h}=0$ is equivalent to $\int_0^\infty h_{1j}\,r^d\,dr=0$ for any $j$. For $\ell=1$, Lemma~\ref{Lem:Chiral}~(a) gives $h_{1j}=-\omega'$. Hence
\[
0=\int_0^\infty h_{1j}\,r^d\,dr=d\int_0^\infty\omega\,r^{d-1}\,dr=d\int_0^\infty\chi\,n\,r^{d-1}\,dr\,.
\]
For $\ell=1$, we have $w_1=r^{d-1}$, so that $\chi$ has zero average with respect to the measure $n\,r^{d-1}\,dr=e^{-W}dr$ on $(0,\infty)$, where $W:=\phi+c_\infty-(d-1)\log r$ up to an additive constant. By~\eqref{notation},~\eqref{nuprime} and~\eqref{mbounds},
\[
W''=\mu-n+(d-1)\,\frac{m}{r^d}+\frac{d-1}{r^2}\ge\mu-\frac nd\ge\mu-\frac{n(0)}d=\delta_0>0\,.
\]
The measure $e^{-W}dr$ is uniformly log-concave on the convex set $(0,\infty)$. By the Brascamp--Lieb inequality~\cite{BrascampLieb}, $\int_0^\infty w_1\,n\,\chi'^2\,dr\ge\delta_0\int_0^\infty w_1\,n\,\chi^2\,dr$. Then~\eqref{GSidentity} with $\ell=1$ shows that~\eqref{Jl} holds with $\kappa=\mu+\delta_0/\kappa_1=\mu+\delta_0/d$. We combine this with Proposition~\ref{Prop:Radial} (constant $2\,\mu$) and Corollary~\ref{Cor:NonRadial} (constant $\mu_\ell$ for $\ell\ge2$). As in Steps~1 and~2 of the proof of Theorem~\ref{Thm:Sharp}, this gives $Q_2\ge\Lambda_2\,Q_1$ with
\[
\Lambda_2\ge\min\left\{2\,\mu,\,\mu+\frac{\delta_0}d,\,\inf_{\ell\ge2}\mu_\ell\right\}\,.
\]
In Step~2, the zero first moment has to be preserved along the approximation. With the notations of Step~2~(d), we replace $f_k$ by
\[
\tilde f_k:=f_k+\sum_{i=1}^da_{k,i}\,\frac{\partial_{x_i}\psi_0}n\,,\qquad a_{k,i}:=\ird{x_i\,f_k\,n}\,.
\]
Since $\ird{\partial_{x_i}\psi_0}=0$ and $\ird{x_j\,\partial_{x_i}\psi_0}=-\delta_{ij}$, the function $\tilde f_k\,n$ belongs to $C_c^\infty(\R^d)$, and it has zero mass and zero first moment. If $\ird{x\,f\,n}=0$, then $|a_{k,i}|\le\big(\ird{|x|^2\,n}\big)^{1/2}\nrm{f_k-f}{\mathrm L^2(n\,dx)}\to0$. Hence $\tilde f_k\to f$ in $\mathrm H^1_n$, because $\partial_{x_i}\psi_0/n\in C_c^\infty(\R^d)$. If $d=2$, then $\mu_\ell\ge\mu+\delta_0$ and $\delta_0\le\mu$, so that $\Lambda_2\ge\mu+\delta_0/2$. If $d=3$, then $\mu_\ell$ is bounded from below by an increasing function of $\ell$, so that $\inf_{\ell\ge2}\mu_\ell\ge\frac65\,\mu+\frac45\,\delta_0>\mu+\frac{\delta_0}3$, and $\mu+\delta_0/3<2\,\mu$, which proves the claim.
\end{proof}

\begin{remark}\label{Rem:NotSharp}
The lower bound of Theorem~\ref{Thm:SecondGap} is not sharp. Since $\nu(r)\le\mu\,r$, we have $n(r)\ge n(0)\,e^{-\mu r^2/2}$, hence $M\ge n(0)\,(2\pi/\mu)^{d/2}$ and $\delta_0\to\mu$ as $M\to0_+$. In this limit, $n_\infty/M$ converges to a Gaussian function and the linearized operator converges, at least formally, to the Ornstein--Uhlenbeck operator, whose eigenvalues are the integer multiples of $\mu$: the second spectral gap is $2\,\mu$, while Theorem~\ref{Thm:SecondGap} only gives $\mu\,(1+1/d)$. The loss comes from the component $\ell=1$, for two reasons: the Brascamp--Lieb inequality is applied with the crude bound $W''\ge\delta_0$, and the chiral flux, which is the optimal flux for the translation modes, is not optimal for the other perturbations of degree one. In contrast, the bounds $\mu_\ell\ge\mu+(\ell-1)\,\delta_0$ if $d=2$ and $\mu_\ell\ge\frac{3\ell\mu}{2\ell+1}+\frac{2\ell(\ell-1)}{2\ell+1}\,\delta_0$ if $d=3$ of Corollary~\ref{Cor:NonRadial}, and the bound $2\,\mu$ of Proposition~\ref{Prop:Radial}, converge as $M\to0_+$ to the Ornstein--Uhlenbeck values $\ell\,\mu$ and $2\,\mu$. Also notice that $\Lambda_2$ is not given by an explicit formula: the lower bound is expressed in terms of the central density $n_\infty(0)$, which depends on $M$ and $\mu$ through the Poisson--Boltzmann equation~\eqref{st}.
\end{remark}

\begin{remark}[Numerical illustration]\label{Rem:Numerics}
Let $\mu=1$. In each angular component $\ell$, we write $f=a(r)\,Y_\ell$ and $G_f=b(r)\,Y_\ell$, where $Y_\ell$ is a spherical harmonic of degree $\ell$ and $b(r)=\int_0^\infty K_\ell(r,s)\,n_\infty(s)\,a(s)\,s^{d-1}\,ds$ is computed with the exact radial Green kernels: $K_\ell(r,s)=-\log r_>$ if $\ell=0$ and $K_\ell(r,s)=(r_</r_>)^\ell/(2\,\ell)$ if $\ell\ge1$ when $d=2$, and $K_\ell(r,s)=r_<^\ell/\big((2\,\ell+1)\,r_>^{\ell+1}\big)$ when $d=3$, with $r_<=\min\{r,s\}$ and $r_>=\max\{r,s\}$. Then, up to a common normalization factor,
\[
Q_1[f]=\int_0^\infty\big(a^2+a\,b\big)\,n_\infty\,r^{d-1}\,dr\,,\qquad Q_2[f]=\int_0^\infty\Big(|(a+b)'|^2+\frac{\ell\,(\ell+d-2)}{r^2}\,(a+b)^2\Big)\,n_\infty\,r^{d-1}\,dr\,.
\]
The stationary state is obtained by solving $n'=-n\,(\mu\,r-m\,r^{1-d})$, $m'=r^{d-1}\,n$ with $n(0)=n_\infty(0)$ and $m(0)=0$, with an explicit Runge--Kutta method of order~$8$ (relative tolerance $10^{-12}$), and $M=|\mathbb S^{d-1}|\,\lim_{r\to+\infty}m(r)$. The generalized eigenvalue problem $Q_2=\lambda\,Q_1$ is discretized on $(0,R)$ with $R=20$, on a uniform cell-centred grid of $N$ points: the integrals are approximated by the midpoint rule, and $(a+b)'$ by centred differences on the cell interfaces, which corresponds to a natural boundary condition at $r=R$. The constraint $\ird{f\,n_\infty}=0$ is imposed in the component $\ell=0$. The resulting symmetric generalized eigenvalue problem is solved with \texttt{scipy.linalg.eigh}. We use $N=1600$ and $N=3200$ and a Richardson extrapolation for a second order method. In all cases below, the difference between the extrapolated value and the value obtained with $N=3200$ is less than $3\cdot10^{-5}$, and changing $R=20$ into $R=26$ changes the eigenvalues by less than $10^{-5}$. As a test, the exact values $\lambda=2\,\mu$ for $\ell=0$ if $d=2$ (Proposition~\ref{Prop:Radial}) and $\lambda=\mu$ for $\ell=1$ (Theorem~\ref{Thm:Sharp}) are recovered with an error less than $10^{-5}$. The following table displays the lowest eigenvalue $\lambda_{\ell,1}$ in the components $\ell=0,\dots,3$, the second eigenvalue $\lambda_{1,2}$ in the component $\ell=1$, and the lower bound $\mu+\delta_0/d$ of Theorem~\ref{Thm:SecondGap}, as functions of the central density $n_\infty(0)$. The values of $M$ are rounded.
\begin{center}\small
\begin{tabular}{@{}ccccccccc@{}}
\toprule
$d$ & $n_\infty(0)$ & $M$ & $\lambda_{0,1}$ & $\lambda_{1,1}$ & $\lambda_{1,2}$ & $\lambda_{2,1}$ & $\lambda_{3,1}$ & $\mu+\delta_0/d$\\
\midrule
$2$ & $0.05$ & $0.3197$ & $2.000$ & $1.000$ & $2.994$ & $1.987$ & $2.981$ & $1.488$\\
$2$ & $1$ & $9.930$ & $2.000$ & $1.000$ & $2.838$ & $1.711$ & $2.531$ & $1.250$\\
$2$ & $1.9$ & $63.02$ & $2.000$ & $1.000$ & $2.460$ & $1.267$ & $1.655$ & $1.025$\\
$2$ & $1.99$ & $150.7$ & $2.000$ & $1.000$ & $2.268$ & $1.130$ & $1.336$ & $1.003$\\
\midrule
$3$ & $0.1$ & $1.623$ & $2.012$ & $1.000$ & $3.004$ & $1.991$ & $2.986$ & $1.322$\\
$3$ & $1.5$ & $42.32$ & $2.212$ & $1.000$ & $3.061$ & $1.823$ & $2.707$ & $1.167$\\
$3$ & $2.9$ & $490.6$ & $2.651$ & $1.000$ & $3.129$ & $1.450$ & $1.925$ & $1.011$\\
$3$ & $2.999$ & $3129$ & $2.871$ & $1.000$ & $3.085$ & $1.286$ & $1.518$ & $1.000$\\
\bottomrule
\end{tabular}
\end{center}
We also computed $\lambda_{4,1}$, which is larger than $\lambda_{3,1}$ in all cases. If $d=3$, the uniform lower bound $3\,\ell\,\mu/(2\,\ell+1)$ of Corollary~\ref{Cor:NonRadial} for $\ell=2$ is $\frac65\,\mu$. The code is available as supplementary material.
In all computed cases, the second spectral gap $\Lambda_2=\min\{\lambda_{0,1},\lambda_{1,2},\inf_{\ell\ge2}\lambda_{\ell,1}\}$ is achieved in the component $\ell=2$, as far as the components $\ell\le4$ are concerned. If $d=2$, the numerical values suggest that the second gap tends to $\mu$ as $M\to+\infty$, in which case a mass-dependent improvement such as $\mu+\delta_0/2$ cannot be avoided. If $d=3$, the gap in the component $\ell\ge2$ is bounded from below by $3\,\ell\,\mu/(2\,\ell+1)$ uniformly in $M$. A formal computation suggests that this bound is asymptotically sharp as $M\to+\infty$: the profile $n_\infty$ then approaches the uniformly charged ball of density $3\,\mu$, and the restoring rate of a surface wave of degree $\ell$ is $3\,\ell\,\mu/(2\,\ell+1)$. Identity~\eqref{GSidentity} has also been checked with a computer algebra system, as a pointwise identity for arbitrary functions $n$, $m$, $\chi$ related by $n'=-n\,(\mu\,r-m\,r^{1-d})$ and $m'=r^{d-1}\,n$, for $\ell=1,2,3$ and $d=2,3$.
\end{remark}

\section{Linearized equation and the large time behaviour}\label{Sec:LargeTime}

This section is devoted to the proofs of Theorems~\ref{Thm4.1} and~\ref{Thm:Improved}. It also collects some additional results.

\subsection{The scalar product and the linearized operator}

We adapt the strategy of~\cite{MR3196188}. The bilinear form
\begin{equation}
\label{scalar}
\scalar{f_1}{f_2}:=\ird{f_1\,f_2\,n_\infty}+\ird{n_\infty\,f_1\,\big(G_d*(f_2\,n_\infty)\big)}
\end{equation}
is a scalar product on $\cA$ because $Q_1[f]=\scalar{f}{f}$. With the notations of~\eqref{Def:hGu}, we have $\scalar{f_1}{f_2}=\ird{h_1\,u_2}=\ird{h_2\,u_1}$. Let us come back to the Poisson--Nernst--Planck system with confinement~\eqref{Eq:(5.2)}. For any $ x\in \mathbb{R}^d$ and $t \ge 0$, let
\[
n(t,x)=n_\infty(x)\,\big(1+f(t,x)\big)\,,\qquad c(t,x)=c_\infty(x)+G(t,x)\,,\qquad -\Delta G=f\,n_\infty\,.
\]
Using $\nabla n_\infty=-n_\infty\,\nabla(c_\infty+\phi)$, we find that
\[
\nabla n+n\,\nabla(c+\phi)=n_\infty\,\nabla f+n_\infty\,\nabla G+n_\infty\,f\,\nabla G\,.
\]
Hence $(f,G)$ solves
\begin{equation}
\label{Eq:(4.1)}
\frac{\partial f}{\partial t}-\mathcal{L}f=\frac{1}{n_\infty}\,\nabla\cdot\big(f\,n_\infty\,\nabla G\big)\,,\qquad -\Delta G=f\,n_\infty\,,\qquad x\in \mathbb{R}^d\,,\; t>0\,,
\end{equation}
where the linear operator $\mathcal L$ is defined by
\[
\mathcal Lf:= \frac{1}{n_\infty}\,\nabla\cdot\big(n_\infty\,\nabla u\big)\,,\qquad u=f+G\,.
\]
Notice that $\ird{n_\infty\,\cL f}=0$, so that $\cL$ maps $\cA$ into $\cA$ (on its domain).
\begin{lem}\label{Lem:SelfAdjoint}
The linearized operator $\mathcal L$ is symmetric with respect to the scalar product defined in~\eqref{scalar}: $\scalar{f_1}{\mathcal Lf_2}=\scalar{\mathcal Lf_1}{f_2}$ for any compactly supported functions $f_1$, $f_2\in\cA\cap\mathrm W^{2,\infty}(\R^d)$. Moreover, for any such function $f$,
\[
-\scalar{f}{\mathcal Lf}=Q_2[f]\,.
\]
\end{lem}
\begin{proof}
Let $u_i=f_i+G_i$ with $G_i=(-\Delta)^{-1}(f_i\,n_{\infty})$, $i=1,2$. Since $n_\infty\,\cL f_1=\nabla\cdot(n_\infty\,\nabla u_1)$, we have
\[
\scalar{\mathcal Lf_1}{f_2}=\ird{(n_\infty\,\cL f_1)\,u_2}=-\ird{n_{\infty}\,\nabla u_1\cdot\nabla u_2}\,,
\]
which is symmetric in $f_1$ and $f_2$. The integration by parts is justified because $n_\infty\,\nabla u_1=n_\infty\,\nabla G_1$ outside of the support of $f_1$, with $n_\infty(x)=O(|x|^{-d-\eps})$ by Lemma~\ref{Lem:ninftyBounds},~\eqref{Hyp-C1} and~\eqref{Hyp-C4}, while $|\nabla G_1(x)|\,|u_2(x)|=O(|x|^{1-d}\log|x|)$ as $|x|\to+\infty$. Similarly, $\cL f_1=-\nabla(\phi+c_\infty)\cdot\nabla G_1$ outside of the support of $f_1$, so that $\cL f_1\in\mathrm L^2(\R^d,n_\infty\,dx)$ because $n_\infty\,|\nabla(\phi+c_\infty)|^2$ is bounded (Lemma~\ref{Lem:A1}) and $\nabla G_1\in\mathrm L^2(\R^d)$, and $\ird{n_\infty\,\cL f_1}=0$ by the same boundary estimate: $\cL f_1\in\cA$. With $f_1=f_2=f$, we recover $Q_2[f]$.
\end{proof}
We shall not need a description of the domain of the self-adjoint (Friedrichs) extension of $\cL$ associated with the closed quadratic form $Q_2$: all the proofs below rely only on the quadratic forms $Q_1$ and $Q_2$ and on the energy identity of Lemma~\ref{Lem:EnergyIdentity}.

\subsection{Proof of Theorem~\ref{Thm4.1}}

The key estimate is the following one.
\begin{lem}\label{Lem:ShortTime}
Let $d=2$ or $3$ and assume that $\phi$ satisfies~\hyperlink{C}{$(\mathrm C)$}. Let $C_\phi>0$ be such that $\Delta\phi\le C_\phi$ a.e. Let $n$ be a solution of~\eqref{Eq:(5.2)} in the sense of Definition~\ref{Def:Solution}, with initial datum $n_0\in\mathrm L^1_+\cap\mathrm L^2(\R^d)$, and let $G=c-c_\infty$. Then, for any $t_0>0$,
\begin{align}
&\nrm{n(t,\cdot)}{\mathrm L^2}^2\le e^{C_\phi\,t}\,\nrm{n_0}{\mathrm L^2}^2\quad\forall\,t\in[0,t_0]\,,\qquad
\int_0^{t_0}\nrm{\nabla n(t,\cdot)}{\mathrm L^2}^2\,dt\le\frac12\,e^{C_\phi\,t_0}\,\nrm{n_0}{\mathrm L^2}^2\,,\label{ShortTime1}\\
&\int_0^{t_0}\nrm{\nabla G(t,\cdot)}{\mathrm L^\infty}^2\,dt\le K\,,\label{ShortTime2}
\end{align}
where $K$ depends only on $d$, $M$, $C_\phi$, $t_0$, $\nrm{n_0}{\mathrm L^2}$ and $\nrm{n_\infty}{\mathrm L^1\cap\mathrm L^\infty}$.
\end{lem}
\begin{proof}
We first prove~\eqref{ShortTime1}. The computation is done with a cut-off function on time intervals $[\tau,t]$ with $\tau>0$, where the solution is a strong solution, and then we let the cut-off radius tend to infinity and $\tau\to0_+$.

\emph{Step 1: localized $\mathrm L^2$ identity.} Let $\chi\in C_c^\infty(\R^d)$ be such that $0\le\chi\le1$, $\chi=1$ on $B_1$ and $\chi=0$ outside $B_2$, and let $\chi_R:=\chi(\cdot/R)$ for $R\ge1$. We write~\eqref{Eq:(5.2)} as $\partial_tn=\nabla\cdot J$ with $J:=\nabla n+n\,\nabla c+n\,\nabla\phi$. On $[\tau,t_0]\times B_{2R}$, $n$ has the regularity of Property~(i) of Definition~\ref{Def:Solution}, so that
\[
\frac d{dt}\ird{n^2\,\chi_R}=2\ird{n\,\chi_R\,\nabla\cdot J}=-2\ird{J\cdot\(\chi_R\,\nabla n+n\,\nabla\chi_R\)}\,.
\]
By integration by parts and $-\Delta c=n$, we have
\begin{align*}
-2\ird{n\,\chi_R\,\nabla n\cdot\nabla c}&=\ird{n^2\,\nabla\cdot(\chi_R\,\nabla c)}=\ird{n^2\,\nabla\chi_R\cdot\nabla c}-\ird{n^3\,\chi_R}\,,\\
-2\ird{n\,\chi_R\,\nabla n\cdot\nabla\phi}&=\ird{n^2\,\nabla\chi_R\cdot\nabla\phi}+\ird{n^2\,\chi_R\,\Delta\phi}\,.
\end{align*}
Collecting terms, we obtain
\be{LocL2}
\frac d{dt}\ird{n^2\,\chi_R}=-2\ird{\chi_R\,|\nabla n|^2}-\ird{n^3\,\chi_R}+\ird{n^2\,\chi_R\,\Delta\phi}+\mathcal E_R(t)\,,
\ee
where
\[
\mathcal E_R:=-2\ird{n\,\nabla n\cdot\nabla\chi_R}-\ird{n^2\,\nabla c\cdot\nabla\chi_R}-\ird{n^2\,\nabla\phi\cdot\nabla\chi_R}\,.
\]

\emph{Step 2: estimate of the remainder.} We have $|\nabla\chi_R|\le\nrm{\nabla\chi}{\mathrm L^\infty}\,R^{-1}\,\ind_{R\le|x|\le2R}$. Since $D^2\phi\in\mathrm L^\infty$, we have $|\nabla\phi(x)|\le|\nabla\phi(0)|+\nrm{D^2\phi}{\mathrm L^\infty}\,|x|$. Hence $|\nabla\phi|\,|\nabla\chi_R|\le A\,\ind_{R\le|x|\le2R}$ for some $A$ which does not depend on $R\ge1$. Therefore, for any $s\in[\tau,t_0]$,
\[
|\mathcal E_R(s)|\le\frac{C_\chi}R\(2\,\nrm{n(s)}{\mathrm L^2}\nrm{\nabla n(s)}{\mathrm L^2}+\nrm{\nabla c(s)}{\mathrm L^\infty}\nrm{n(s)}{\mathrm L^2}^2\)+A\int_{R\le|x|\le2R}n(s,x)^2\,dx\,.
\]
By Definition~\ref{Def:Solution}, the quantities $\nrm{n(s)}{\mathrm L^2}$ and $\nrm{\nabla c(s)}{\mathrm L^\infty}$ are bounded on $[\tau,t_0]$, and $s\mapsto\nrm{\nabla n(s)}{\mathrm L^2}$ belongs to $\mathrm L^2(\tau,t_0)\subset\mathrm L^1(\tau,t_0)$. Since $n\in C([\tau,t_0];\mathrm L^2(\R^d))$, the family $\{n(s,\cdot)^2\}_{s\in[\tau,t_0]}$ is compact in $\mathrm L^1(\R^d)$, hence uniformly integrable at infinity, so that $\sup_{s\in[\tau,t_0]}\int_{|x|\ge R}n(s,x)^2\,dx\to0$ as $R\to+\infty$. As a consequence,
\be{RemainderLimit}
\lim_{R\to+\infty}\int_\tau^{t_0}|\mathcal E_R(s)|\,ds=0\,.
\ee

\emph{Step 3: integration in time.} Since $\Delta\phi\le C_\phi$, $\chi_R\ge0$ and $n\ge0$, we deduce from~\eqref{LocL2} that
\[
\frac d{dt}\(e^{-C_\phi t}\ird{n^2\,\chi_R}\)\le-2\,e^{-C_\phi t}\ird{\chi_R\,|\nabla n|^2}+e^{-C_\phi t}\,\mathcal E_R(t)\,.
\]
After integration on $[\tau,t]$, for any $t\in[\tau,t_0]$,
\[
e^{-C_\phi t}\ird{n(t)^2\,\chi_R}+2\int_\tau^te^{-C_\phi s}\ird{\chi_R\,|\nabla n(s)|^2}\,ds\le e^{-C_\phi\tau}\ird{n(\tau)^2\,\chi_R}+\int_\tau^{t_0}|\mathcal E_R(s)|\,ds\,.
\]
We let $R\to+\infty$. The left-hand side converges by monotone convergence, because $\chi_R\uparrow1$. The last term vanishes by~\eqref{RemainderLimit}. Hence
\be{L2tau}
e^{-C_\phi t}\nrm{n(t)}{\mathrm L^2}^2+2\int_\tau^te^{-C_\phi s}\nrm{\nabla n(s)}{\mathrm L^2}^2\,ds\le e^{-C_\phi\tau}\nrm{n(\tau)}{\mathrm L^2}^2\quad\forall\,t\in[\tau,t_0]\,.
\ee

\emph{Step 4: limit $\tau\to0_+$.} Since $n\in C([0,\infty);\mathrm L^2(\R^d))$, we have $\nrm{n(\tau)}{\mathrm L^2}\to\nrm{n_0}{\mathrm L^2}$ as $\tau\to0_+$. The integral in~\eqref{L2tau} is nondecreasing as $\tau\downarrow0$, and it converges by monotone convergence. We obtain
\[
e^{-C_\phi t}\nrm{n(t)}{\mathrm L^2}^2+2\int_0^te^{-C_\phi s}\nrm{\nabla n(s)}{\mathrm L^2}^2\,ds\le\nrm{n_0}{\mathrm L^2}^2\quad\forall\,t\in[0,t_0]\,.
\]
The first estimate of~\eqref{ShortTime1} follows. The second one follows from $e^{-C_\phi s}\ge e^{-C_\phi t_0}$ for $s\in[0,t_0]$. In particular, $\nabla n(s)\in\mathrm L^2(\R^d)$ for a.e. $s\in(0,t_0)$, and the Gagliardo--Nirenberg inequality below can be applied for a.e. $s$.

\emph{Gagliardo--Nirenberg inequality.} Let
\[
\vartheta:=\frac{d\,(d-1)}{2\,(d+1)}\,,\qquad\mbox{that is,}\qquad\vartheta=\frac13\quad\mbox{if}\quad d=2\,,\qquad\vartheta=\frac34\quad\mbox{if}\quad d=3\,.
\]
The exponent $\vartheta$ is determined by the scaling condition $\frac1{d+1}=\frac12-\frac\vartheta d$. Since $\vartheta\in(0,1)$, the Gagliardo--Nirenberg inequality
\[
\nrm g{\mathrm L^{d+1}(\R^d)}\le C_{\rm GN}\,\nrm g{\mathrm L^2(\R^d)}^{1-\vartheta}\,\nrm{\nabla g}{\mathrm L^2(\R^d)}^\vartheta\quad\forall\,g\in\mathrm H^1(\R^d)
\]
holds for $d=2$ and $d=3$. With $g=n(t,\cdot)$ and~\eqref{ShortTime1}, we get, by H\"older's inequality in time with exponents $\frac1{1-\vartheta}$ and $\frac1\vartheta$,
\[
\int_0^{t_0}\nrm{n(t,\cdot)}{\mathrm L^{d+1}}^2\,dt\le C_{\rm GN}^2\,\sup_{t\in[0,t_0]}\nrm{n(t,\cdot)}{\mathrm L^2}^{2(1-\vartheta)}\;t_0^{1-\vartheta}\(\int_0^{t_0}\nrm{\nabla n(t,\cdot)}{\mathrm L^2}^2\,dt\)^\vartheta\le C_{\rm GN}^2\,t_0^{1-\vartheta}\,e^{C_\phi\,t_0}\,\nrm{n_0}{\mathrm L^2}^2\,.
\]

\emph{Conclusion.} By Lemma~\ref{Lem:Field} applied to $\rho=n-n_\infty$, we have
\[
\nrm{\nabla G}{\mathrm L^\infty}\le\nrm{n-n_\infty}{\mathrm L^1}+C_d\,\nrm{n-n_\infty}{\mathrm L^{d+1}}\le2\,M+C_d\,\nrm{n_\infty}{\mathrm L^{d+1}}+C_d\,\nrm{n}{\mathrm L^{d+1}}\,,
\]
with $C_d=d^{\frac d{d+1}}|\Sph^{d-1}|^{-\frac1{d+1}}$. Since $\nrm{n_\infty}{\mathrm L^{d+1}}\le\nrm{n_\infty}{\mathrm L^1}^{\frac1{d+1}}\nrm{n_\infty}{\mathrm L^\infty}^{\frac d{d+1}}<\infty$ by Proposition~\ref{Prop:cinfty}, we get
\[
\int_0^{t_0}\nrm{\nabla G(t,\cdot)}{\mathrm L^\infty}^2\,dt\le2\,t_0\,\big(2\,M+C_d\,\nrm{n_\infty}{\mathrm L^{d+1}}\big)^2+2\,C_d^2\int_0^{t_0}\nrm{n(t,\cdot)}{\mathrm L^{d+1}}^2\,dt\,,
\]
which proves~\eqref{ShortTime2}.
\end{proof}
Notice that $n_0\in\mathrm L^2(n_\infty^{-1}dx)$ implies $n_0\in\mathrm L^2(\R^d)$, since $\ird{n_0^2}\le\nrm{n_\infty}{\mathrm L^\infty}\ird{n_0^2\,n_\infty^{-1}}$. Hence Lemma~\ref{Lem:ShortTime} applies under the assumptions of Theorem~\ref{Thm4.1}.

\begin{lem}\label{Lem:EnergyIdentity}
Let $d=2$ or $3$ and assume that $\phi$ satisfies~\hyperlink{C}{$(\mathrm C)$}. Let $n$ be a solution of~\eqref{Eq:(5.2)} in the sense of Definition~\ref{Def:Solution}, with $n_0\in\mathrm L^2(n_\infty^{-1}dx)$, and let $f=n/n_\infty-1$, $G=c-c_\infty$ and $u=f+G$. Then $t\mapsto Q_1[f(t,\cdot)]$ is absolutely continuous on $[0,t_0]$ for any $t_0>0$, and
\be{EnergyIdentity}
Q_1[f(t,\cdot)]-Q_1[f(s,\cdot)]=-2\int_s^t\big(Q_2[f(r,\cdot)]+\lambda(r)\big)\,dr\,,\qquad\lambda:=\ird{n_\infty\,f\,\nabla u\cdot\nabla G}\,,
\ee
for any $0\le s\le t$.
\end{lem}
\begin{proof}
Fix $t_0>0$, and let $\mathcal H:=\mathrm L^2(n_\infty\,dx)$ and $\mathcal V:=\mathrm H^1(n_\infty\,dx)$, with the norms $\nrm g{\mathcal H}^2=\ird{g^2\,n_\infty}$ and $\nrm g{\mathcal V}^2=\ird{(g^2+|\nabla g|^2)\,n_\infty}$. Then $\mathcal V\subset\mathcal H\subset\mathcal V'$ is a Gelfand triple, since $C_c^\infty(\R^d)$ is dense in $\mathcal V$ and in $\mathcal H$ (see Step~2~(c) of the proof of Theorem~\ref{Thm:Sharp}, which only uses the continuity and positivity of $n_\infty$).

\emph{Step 1: the flux.} With $J:=\nabla n+n\,\nabla(c+\phi)$, we have $\partial_tn=\nabla\cdot J$, and, as in Section~\ref{Sec:LargeTime},
\[
J=n_\infty\,\nabla f+n_\infty\,(1+f)\,\nabla G=n_\infty\,\nabla u+n_\infty\,f\,\nabla G\,.
\]
By Lemma~\ref{Lem:GradGL2}, $\nrm{\nabla G(t)}{\mathrm L^2}\le C_G\,\nrm{f(t)}{\mathcal H}$, and by Lemma~\ref{Lem:ShortTime}, $\int_0^{t_0}\nrm{\nabla G}{\mathrm L^\infty}^2\,dt<\infty$. Hence
\[
\int_0^{t_0}\!\!\ird{\frac{|J|^2}{n_\infty}}\,dt\le3\int_0^{t_0}\!\(\nrm{\sqrt{n_\infty}\,\nabla f}{\mathrm L^2}^2+\nrm{n_\infty}{\mathrm L^\infty}C_G^2\,\nrm f{\mathcal H}^2+\nrm{\nabla G}{\mathrm L^\infty}^2\,\nrm f{\mathcal H}^2\)dt<\infty
\]
by Property~(iv) of Definition~\ref{Def:Solution}. In particular, $J/\sqrt{n_\infty}\in\mathrm L^2\big((0,t_0)\times\R^d\big)$ and $J\in\mathrm L^2\big((0,t_0)\times\R^d\big)$, because $|J|^2\le\nrm{n_\infty}{\mathrm L^\infty}\,|J|^2/n_\infty$.

\emph{Step 2: the weighted $\mathrm L^2$ part.} For any $g\in C_c^\infty(\R^d)$, the equation $n_\infty\,\partial_tf=\partial_tn=\nabla\cdot J$ gives $\frac d{dt}\ird{f\,g\,n_\infty}=-\ird{J\cdot\nabla g}$ in $\mathcal D'(0,t_0)$. Since $|\ird{J\cdot\nabla g}|\le\nrm{J/\sqrt{n_\infty}}{\mathrm L^2}\,\nrm g{\mathcal V}$, this extends to any $g\in\mathcal V$ by density. Hence $\partial_tf\in\mathrm L^2(0,t_0;\mathcal V')$ with $\langle\partial_tf,g\rangle_{\mathcal V',\mathcal V}=-\ird{J\cdot\nabla g}$, and $f\in\mathrm L^2(0,t_0;\mathcal V)$. By the Lions--Magenes lemma (see~\cite[Chapter~III, Lemma~1.2]{Temam}), $t\mapsto\nrm{f(t)}{\mathcal H}^2$ is absolutely continuous and
\[
\frac d{dt}\ird{f^2\,n_\infty}=2\,\langle\partial_tf,f\rangle_{\mathcal V',\mathcal V}=-2\ird{J\cdot\nabla f}\quad\mbox{for a.e. }t\in(0,t_0)\,.
\]

\emph{Step 3: the electrostatic part.} Let $h:=n-n_\infty=f\,n_\infty$ and let $\hat\cdot$ denote the Fourier transform in $x$. The equation $\partial_th=\nabla\cdot J$ with $J\in\mathrm L^2\big((0,t_0)\times\R^d\big)$ gives $\partial_t\hat h=i\,\xi\cdot\hat J$, so that $\nabla G=\nabla(-\Delta)^{-1}h$ satisfies
\[
\partial_t\nabla G=\nabla(-\Delta)^{-1}\nabla\cdot J=-\mathcal PJ\,,\qquad\widehat{\mathcal PJ}(\xi):=\frac{\xi\,(\xi\cdot\hat J(\xi))}{|\xi|^2}\,.
\]
The operator $\mathcal P$ is the $\mathrm L^2$-orthogonal projection onto gradient fields, so that $\nrm{\mathcal PJ}{\mathrm L^2}\le\nrm J{\mathrm L^2}$. Hence $\nabla G\in\mathrm H^1\big(0,t_0;\mathrm L^2(\R^d)\big)$, and $t\mapsto\nrm{\nabla G(t)}{\mathrm L^2}^2$ is absolutely continuous, with
\[
\frac d{dt}\ird{|\nabla G|^2}=-2\ird{\nabla G\cdot\mathcal PJ}=-2\ird{\nabla G\cdot J}\quad\mbox{for a.e. }t\in(0,t_0)\,,
\]
because $\mathcal P$ is self-adjoint and $\mathcal P\nabla G=\nabla G$.

\emph{Step 4: conclusion.} By~\eqref{Q1}, $Q_1[f]=\nrm f{\mathcal H}^2+\nrm{\nabla G}{\mathrm L^2}^2$. Adding Steps~2 and~3, we get $\frac d{dt}Q_1[f]=-2\ird{J\cdot\nabla u}$ for a.e. $t$. Moreover, $J\cdot\nabla u=n_\infty\,|\nabla u|^2+n_\infty\,f\,\nabla G\cdot\nabla u$, and both terms are integrable on $(0,t_0)\times\R^d$ by Step~1. This proves~\eqref{EnergyIdentity}.
\end{proof}

\begin{lem}\label{Lem:Q1decay}
Let $d=2$ or $3$, let $\phi$ satisfy~\hyperlink{C}{$(\mathrm C)$}, and let $n$ be a solution of~\eqref{Eq:(5.2)} as in Theorem~\ref{Thm4.1}. Let $\Lambda>0$ be such that $Q_2[f(t,\cdot)]\ge\Lambda\,Q_1[f(t,\cdot)]$ for any $t\ge0$, where $f=n/n_\infty-1$. Then there exists $C>0$ such that $Q_1[f(t,\cdot)]\le C\,e^{-2\,\Lambda\,t}$ for any $t\ge0$.
\end{lem}
\begin{proof}
By Lemma~\ref{Lem:EnergyIdentity}, $t\mapsto Q_1[f(t,\cdot)]$ is absolutely continuous and, for a.e. $t>0$,
\[
\frac{d}{dt}Q_{1}[f]=-2\,Q_{2}[f]-2\,\lambda(t)\quad\mbox{with}\quad\lambda(t):=\ird{n_\infty\,f\,\nabla u\cdot\nabla G}\,.
\]
This is consistent with the formal computation based on~\eqref{Eq:(4.1)} and Lemma~\ref{Lem:SelfAdjoint}. According to the Cauchy--Schwarz inequality, we have
\be{lambda}
|\lambda(t)|\le\sqrt{Q_{2}[f]}\(\ird{f^2\,n_\infty}\)^{1/2}\|\nabla G\|_{\mathrm L^{\infty}(\mathbb{R}^d)}
\le \sqrt{Q_{2}[f]\,Q_{1}[f]}\;\|\nabla G\|_{\mathrm L^{\infty}(\mathbb{R}^d)}\,.
\ee

\emph{Step 1: bounded time intervals.} By~\eqref{lambda},
\[
\frac{d}{dt}Q_1[f]\le-2\,Q_2[f]+2\sqrt{Q_2[f]\,Q_1[f]}\,\|\nabla G\|_{\mathrm L^\infty}\le\frac12\,\|\nabla G\|_{\mathrm L^\infty}^2\,Q_1[f]\,.
\]
By Gr\"onwall's lemma,
\[
Q_1[f(t,\cdot)]\le Q_1[f(0,\cdot)]\,\exp\(\frac12\int_0^t\|\nabla G(s,\cdot)\|_{\mathrm L^\infty}^2\,ds\)\,.
\]
Here $Q_1[f(0,\cdot)]\le(1+C_G^2)\ird{f_0^2\,n_\infty}<\infty$ by Lemma~\ref{Lem:GradGL2}, since
\[
\ird{f_0^2\,n_\infty}=\ird{|n_0-n_\infty|^2\,n_\infty^{-1}}\le2\ird{n_0^2\,n_\infty^{-1}}+2\,M\,.
\]
By Lemma~\ref{Lem:ShortTime}, $t\mapsto\|\nabla G(t,\cdot)\|_{\mathrm L^\infty}^2$ is integrable on $(0,t_0)$ for any $t_0>0$. Hence $Q_1[f(t,\cdot)]$ is bounded on any bounded time interval.

\emph{Step 2: large times.} By Lemma~\ref{Lem:Field} and Proposition~\ref{Thm:(2.1)}, we have $\lim_{t\to+\infty}\|\nabla G(t,\cdot)\|_{\mathrm L^{\infty}}=0$ by~\eqref{gradcinfty}. Using $Q_1\le Q_2/\Lambda$ in~\eqref{lambda}, we obtain
\[
\frac{d}{dt}Q_{1}[f]\le -2\left(1-\frac{\|\nabla G\|_{\mathrm L^{\infty}}}{\sqrt{\Lambda}}\right)Q_2[f]\le -2\,\Lambda\left(1-\frac{\|\nabla G\|_{\mathrm L^{\infty}}}{\sqrt{\Lambda}}\right)Q_1[f]
\]
for $t$ large enough. This proves that $\limsup_{t\to\infty}e^{2\,(\Lambda-\varepsilon)\,t}\,Q_1[f(t,\cdot)]<\infty$ for any $\varepsilon\in(0,\Lambda)$. To take $\varepsilon=0$, we give a more accurate estimate of $\|\nabla G\|_{\mathrm L^{\infty}}$. By Lemma~\ref{Lem:Field} applied to $\rho=f\,n_\infty$,
\[
\|\nabla G\|_{\mathrm L^{\infty}}\le C\(\|f\,n_\infty\|_{\mathrm L^1} +\|f\,n_\infty\|_{\mathrm L^{d+1}}\)\,,
\]
where $\|f\,n_\infty\|_{\mathrm L^1}\le\sqrt{M}\,\|f\sqrt{n_\infty}\|_{\mathrm L^2}$ and
\[
\|f\,n_\infty\|_{\mathrm L^{d+1}}\le \|f\sqrt{n_\infty}\|^\frac{2}{d+1}_{\mathrm L^2}\,\|f\,n_\infty\|^\frac{d-1}{d+1}_{\mathrm L^{\infty}}\,\|n_\infty\|^\frac{1}{d+1}_{\mathrm L^{\infty}}\,.
\]
Here $\|f\,n_\infty\|_{\mathrm L^{\infty}}=\|n-n_\infty\|_{\mathrm L^{\infty}}$ is bounded for $t\ge T$ by Lemma~\ref{lem2.1} and Proposition~\ref{Prop:cinfty}. Since $\|f\sqrt{n_\infty}\|^2_{\mathrm L^2}\le Q_{1}[f]$ and $Q_1[f(t,\cdot)]\to0$, we deduce that
\begin{equation}
\label{secondestimate1}
\|\nabla G(t,\cdot)\|_{\mathrm L^{\infty}} \le C\,Q_1[f(t,\cdot)]^{\varsigma_d}\,,\qquad \varsigma_d:=\frac1{d+1}\,,
\end{equation}
for $t$ large enough. As a consequence, $2\,|\lambda(t)|\le2\,C\,Q_2[f]\,Q_1[f]^{\varsigma_d}/\sqrt\Lambda$, and
\[
\frac{d}{dt}Q_1[f]\le-2\,Q_2[f]\(1-\frac C{\sqrt\Lambda}\,Q_1[f]^{\varsigma_d}\)\le-2\,\Lambda\,Q_1[f]+\frac{2\,C\,\Lambda}{\sqrt\Lambda}\,Q_1[f]^{1+\varsigma_d}\,.
\]
With $y(t):=e^{2\Lambda t}\,Q_1[f(t,\cdot)]$, this gives $y'\le C'\,Q_1[f]^{\varsigma_d}\,y\le C''\,e^{-2\varsigma_d(\Lambda-\eps)t}\,y$, so that $y$ is bounded.
\end{proof}

\begin{proof}[Proof of Theorem~\ref{Thm4.1}]
By Proposition~\ref{Prop:Coercivity}, Lemma~\ref{Lem:Q1decay} applies with $\Lambda=\Cs$, and $\ird{|n-n_\infty|^2\,n_\infty^{-1}}=\ird{f^2\,n_\infty}\le Q_1[f]$.
\end{proof}

\subsection{Rates of convergence in \texorpdfstring{$\mathrm L^p$}{Lp} norms}

In this section, we assume that $\phi(x)=\frac\mu2\,|x|^2$, and we write $\delta(t,x):=n(t,x)-n_\infty(x)$ and $E(t,x):=\nabla c(t,x)-\nabla c_\infty(x)=\nabla G(t,x)$.
\begin{prop}\label{cormore}
Let $d=2$ or $3$ and $\phi(x)=\frac\mu2\,|x|^2$. Let $n$ be a solution of~\eqref{Eq:(5.2)} such that $Q_1[f(t,\cdot)]\le C\,e^{-2\gamma t}$ for any $t\ge0$, for some $\gamma>0$. Then, as $t\to+\infty$,
\begin{enumerate}[label=(\roman*)]
\item $\nrm{\delta(t,\cdot)}{\mathrm L^p}=O\(e^{-\gamma t}\)$ for any $p\in[1,2]$, and $\nrm{E(t,\cdot)}{\mathrm L^2}=O\(e^{-\gamma t}\)$,
\item $\nrm{\delta(t,\cdot)}{\mathrm L^p}=O\(e^{-\lambda t}\)$ and $\nrm{E(t,\cdot)}{\mathrm L^q}=O\(e^{-\lambda t}\)$ for any $p\in(2,\infty]$, $q\in(2,\infty]$ and $\lambda<\gamma$.
\end{enumerate}
By Theorems~\ref{Thm4.1} and~\ref{Thm:Sharp}, these estimates hold with $\gamma=\mu$ under the assumptions of Theorem~\ref{Thm4.1}.
\end{prop}
\begin{proof}
\emph{Basic estimates.} By the Cauchy--Schwarz inequality,
\[
\nrm{\delta}{\mathrm L^1}\le\(M\ird{\frac{|n-n_\infty|^2}{n_\infty}}\)^{1/2}\le\sqrt{M\,Q_1[f]}\,,\qquad
\nrm{\delta}{\mathrm L^2}\le\nrm{n_\infty}{\mathrm L^\infty}^{1/2}\,Q_1[f]^{1/2}\,,\qquad\nrm{E}{\mathrm L^2}\le Q_1[f]^{1/2}\,.
\]
This proves~(i), using H\"older's inequality for $p\in(1,2)$. Moreover, $\nrm{\delta(t,\cdot)}{\mathrm L^\infty}$ is bounded for $t\ge T$ by Lemma~\ref{lem2.1}.

\emph{Bootstrap.} Assume that $\nrm{\delta(t,\cdot)}{\mathrm L^\infty}=O(e^{-a t})$ for some $a\in[0,\gamma]$. Fix $p_0\in(d,\infty)$, for instance $p_0=2\,d$. By H\"older's inequality, $\nrm{\delta}{\mathrm L^p}\le\nrm{\delta}{\mathrm L^2}^{2/p}\,\nrm{\delta}{\mathrm L^\infty}^{1-2/p}$ for any $p\ge2$. Hence $\nrm{\delta}{\mathrm L^{p_0}}=O(e^{-a_1t})$ with $a_1:=\frac{2\gamma+(p_0-2)\,a}{p_0}$. By Lemma~\ref{Lem:Field},
\[
\nrm E{\mathrm L^\infty}\le\nrm\delta{\mathrm L^1}+C\,\nrm\delta{\mathrm L^{d+1}}=O\(e^{-\gamma t}\)+O\(e^{-\frac{2\gamma+(d-1)a}{d+1}t}\)=O\(e^{-a t}\)\,,
\]
because $a\le\gamma$. By interpolation between $\mathrm L^2$ and $\mathrm L^\infty$, we get $\nrm E{\mathrm L^{p_0}}=O(e^{-a_1t})$. Lemma~\ref{Lem:Duhamel} applied with $p=p_0$ then shows that
\[
\nrm{\delta(t,\cdot)}{\mathrm L^\infty}=O\(e^{-a't}\)\,,\qquad a':=\min\{\gamma,a_1\}=\frac{2\,\gamma}{p_0}+\(1-\frac2{p_0}\)a\,.
\]
The map $a\mapsto a'$ is a contraction with fixed point $\gamma$. Starting from $a_0=0$, we obtain a sequence $a_k=\gamma\,\big(1-(1-2/p_0)^k\big)$ which converges to $\gamma$. At each step, the time origin is shifted by $1$ and the constant is changed. Hence $\nrm{\delta(t,\cdot)}{\mathrm L^\infty}=O(e^{-\lambda t})$ for any $\lambda<\gamma$. The other estimates of~(ii) follow by interpolation as above, since $\frac{2\gamma+(p-2)\lambda}p\ge\lambda$.
\end{proof}

\begin{remark}
At each step of the bootstrap, the estimates of $\delta$ \emph{and} of $E$ have to be improved simultaneously. An induction on the estimate of $\delta$ alone does not close, because a fixed rate for $\nabla c-\nabla c_\infty$ in $\mathrm L^q$ limits the rate obtained from Lemma~\ref{Lem:Duhamel}. Lemma~\ref{Lem:Field} shows that $E$ inherits the rate of $\delta$, which closes the argument. Also notice that the rate $\gamma$ itself is not reached in $\mathrm L^\infty$ by this method, because of the loss in the interpolation.
\end{remark}

\subsection{Center of mass and improved rates}\label{Sec:Improved}

\begin{lem}\label{Lem:CoM}
Let $d=2$ or $3$ and $\phi(x)=\frac\mu2\,|x|^2$.
\begin{enumerate}[label=(\roman*)]
\item Let $n$ be a solution of~\eqref{Eq:(5.2)} with $n_0\in\mathrm L^2(n_\infty^{-1}dx)$, and let $X(t):=M^{-1}\ird{x\,n(t,x)}$. Then $X(t)=X_0\,e^{-\mu t}$ for any $t\ge0$, with $X_0=X(0)$.
\item Let $n$ be a solution of~\eqref{Eq:(5.2)} and $a\in\R^d$. Then $n_a(t,x):=n\big(t,x+a\,e^{-\mu t}\big)$ is a strong solution of~\eqref{Eq:(5.2)} on $(0,\infty)\times\R^d$, in the sense of Property~(i) of Definition~\ref{Def:Solution}, which is continuous on $[0,\infty)$ with values in $\mathrm L^1\cap\mathrm L^2(\R^d)$, with initial datum $n_0(\cdot+a)$.
\end{enumerate}
\end{lem}
\begin{proof}
(i) Let $f=n/n_\infty-1$ and $G=c-c_\infty$. Since $n_\infty$ is radial and has Gaussian decay, we have $\ird{x\,n_\infty}=0$, $M\,X(t)=\ird{x\,f(t,x)\,n_\infty}$, and the coordinate functions $x_i$ belong to $\mathcal V=\mathrm H^1(n_\infty\,dx)$. By Step~2 of the proof of Lemma~\ref{Lem:EnergyIdentity}, $t\mapsto\ird{f\,g\,n_\infty}$ is absolutely continuous for any $g\in\mathcal V$, with derivative $-\ird{J\cdot\nabla g}$, where $J=n_\infty\,\nabla f+n\,\nabla G$. With $g=x_i$, we obtain, for a.e. $t>0$,
\[
M\,X_i'=-\ird{n_\infty\,\partial_{x_i}f}-\ird{n\,\partial_{x_i}G}\,.
\]
Since $\nabla n_\infty=-n_\infty\,(\mu\,x+\nabla c_\infty)$ and $\ird{n_\infty\,|\mu\,x+\nabla c_\infty|^2}<\infty$, an integration by parts, which is justified by the density of $C_c^\infty(\R^d)$ in $\mathcal V$, gives
\[
\ird{n_\infty\,\partial_{x_i}f}=\ird{f\,n_\infty\,(\mu\,x_i+\partial_{x_i}c_\infty)}=\mu\,M\,X_i+\ird{(n-n_\infty)\,\partial_{x_i}c_\infty}\,.
\]
On the other hand, $\ird{n\,\partial_{x_i}G}=\ird{n\,\partial_{x_i}c}-\ird{n\,\partial_{x_i}c_\infty}$. For a.e. $t>0$, $n(t,\cdot)\in\mathrm L^1\cap\mathrm L^{d+1}(\R^d)$, so that the integral
\[
\ird{n\,\nabla c}=\iint_{\R^d\times\R^d}n(x)\,n(y)\,\nabla G_d(x-y)\,dx\,dy
\]
is absolutely convergent by Lemma~\ref{Lem:Field}, and it vanishes because $\nabla G_d$ is an odd function: the Coulomb self-interaction exerts no net force. For the same reason, $\ird{n_\infty\,\nabla c_\infty}=0$. Altogether, $M\,X_i'=-\mu\,M\,X_i$.

(ii) Let $X(t):=a\,e^{-\mu t}$, $\tilde n:=n_a$ and $\tilde c:=G_d*\tilde n$, so that $\tilde c(t,x)=c(t,x+X(t))$. With $y=x+X(t)$, we get $\partial_t\tilde n=\partial_tn+X'\cdot\nabla n$ and $\mu\,\nabla\cdot(y\,n)=\mu\,\nabla\cdot(x\,\tilde n)+\mu\,X\cdot\nabla\tilde n$. The last term cancels $X'\cdot\nabla\tilde n=-\mu\,X\cdot\nabla\tilde n$. The diffusion and the Coulomb terms are invariant under translations. The continuity at $t=0$ follows from the continuity of translations in $\mathrm L^1\cap\mathrm L^2(\R^d)$.
\end{proof}
The lemma shows that the translation modes of Lemma~\ref{Lem:Upper} correspond to an exact symmetry of the nonlinear problem. We now prove that, once this symmetry is removed, the decay is governed by the second spectral gap.

\begin{proof}[Proof of Theorem~\ref{Thm:Improved}]
The solution $\tilde n$ with initial datum $\tilde n_0\in\mathrm L^2(n_\infty^{-1}dx)$ is given by Proposition~\ref{Prop:A3}, and $\mathcal F[\tilde n_0]<\infty$. The center of mass of $\tilde n_0$ is $0$. By Lemma~\ref{Lem:CoM}~(i), the center of mass of $\tilde n(t,\cdot)$ is $0$ for any $t\ge0$. Let $n(t,x):=\tilde n\big(t,x-X(t)\big)$ with $X(t):=X_0\,e^{-\mu t}$. By Lemma~\ref{Lem:CoM}~(ii) applied with $a=-X_0$, $n$ solves~\eqref{Eq:(5.2)} with initial datum $n_0$, it satisfies Properties~(i) and~(ii) of Definition~\ref{Def:Solution}, and the center of mass of $n(t,\cdot)$ is $X(t)$. It also satisfies Property~(iii). Indeed, since $\ird{x\,\tilde n}=0$, $\ird{\nabla\tilde n}=0$ and $\ird{\tilde n\,\nabla\tilde c}=0$ with $\tilde c=G_d*\tilde n$, a change of variables gives
\[
\mathcal F[n(t,\cdot)]=\mathcal F[\tilde n(t,\cdot)]+\frac{\mu\,M}2\,|X(t)|^2\,,\qquad\mathcal I[n(t,\cdot)]=\mathcal I[\tilde n(t,\cdot)]+\mu^2\,M\,|X(t)|^2\,,
\]
where the second identity uses $\ird{\nabla\tilde n}=0$: this holds if $\mathcal I[\tilde n]<\infty$, because $\nabla\tilde n+\tilde n\,\nabla\tilde c+\mu\,x\,\tilde n$ is then integrable by the Cauchy--Schwarz inequality, and both sides are infinite otherwise. Moreover, $\frac d{dt}\big(\frac{\mu M}2\,|X|^2\big)=-\mu^2\,M\,|X|^2$, so that the free energy inequality for $\tilde n$ is equivalent to the free energy inequality for $n$. Since $n_\infty$ has Gaussian decay, $\ird{|x|^2\,\tilde n_0}\le\big(\ird{|x|^4\,n_\infty}\big)^{1/2}\big(\ird{\tilde n_0^2\,n_\infty^{-1}}\big)^{1/2}<\infty$, hence $\ird{|x|^2\,n_0}<\infty$, and the uniqueness of $n$ follows from Proposition~\ref{Prop:Uniq}. Let $\tilde f:=\tilde n/n_\infty-1$. Since $n_\infty$ is radial, we have $\ird{x\,\tilde f(t,x)\,n_\infty}=\ird{x\,\tilde n(t,x)}=0$ for any $t\ge0$. By Theorem~\ref{Thm:SecondGap}, $Q_2[\tilde f(t,\cdot)]\ge\Lambda_2\,Q_1[\tilde f(t,\cdot)]$ for any $t\ge0$. Lemma~\ref{Lem:Q1decay} then gives $Q_1[\tilde f(t,\cdot)]\le C\,e^{-2\Lambda_2t}$, so that
\[
\nrm{n(t,\cdot)-n_\infty\big(\cdot-X(t)\big)}{\mathrm L^1(\R^d)}=\nrm{\tilde n(t,\cdot)-n_\infty}{\mathrm L^1(\R^d)}\le\sqrt{M\,Q_1[\tilde f(t,\cdot)]}\le C\,e^{-\Lambda_2t}\,.
\]
This proves Theorem~\ref{Thm:Improved}, since $X(t)=X_0\,e^{-\mu t}$.
\end{proof}
\begin{remark}\label{Rem:Translation}
The weighted space $\mathrm L^2(n_\infty^{-1}dx)$ is not invariant under translations, because of the Gaussian growth of $n_\infty^{-1}$: the condition $n_0(\cdot+X_0)\in\mathrm L^2(n_\infty^{-1}dx)$ neither implies nor is implied by $n_0\in\mathrm L^2(n_\infty^{-1}dx)$ if $X_0\neq0$. This is why the solution of Theorem~\ref{Thm:Improved} is obtained from $\tilde n$ by translation, and why a uniqueness result which does not involve the weight is needed. If $n_0$ also belongs to $\mathrm L^2(n_\infty^{-1}dx)$, then, by Proposition~\ref{Prop:Uniq}, $n$ coincides with the solution of Proposition~\ref{Prop:A3} with initial datum $n_0$, and Theorems~\ref{Thm4.1} and~\ref{Thm:Improved} apply to the same solution.
\end{remark}

\begin{cor}\label{Cor:Expansion}
Under the assumptions of Theorem~\ref{Thm:Improved}, let $\Lambda_2^*:=\min\{\Lambda_2,2\,\mu\}$. Then
\[
n(t,\cdot)=n_\infty-e^{-\mu t}\,X_0\cdot\nabla n_\infty+O\(e^{-\Lambda_2^*t}\)\quad\mbox{in}\quad\mathrm L^1(\R^d)\quad\mbox{as}\quad t\to+\infty\,.
\]
Moreover, all the estimates of Proposition~\ref{cormore} hold for $\tilde n-n_\infty$ with $\gamma=\Lambda_2$. If $d=2$, the radial spectral gap is exactly $2\,\mu$, so that $\Lambda_2\le2\,\mu$ and $\Lambda_2^*=\Lambda_2$.
\end{cor}
\begin{proof}
By Taylor's formula,
\[
\nrm{n_\infty(\cdot-X)-n_\infty+X\cdot\nabla n_\infty}{\mathrm L^1}\le\frac12\,|X|^2\,\nrm{D^2n_\infty}{\mathrm L^1}\,,
\]
and $|X(t)|^2=|X_0|^2\,e^{-2\mu t}$. The second statement follows from Proposition~\ref{cormore} applied to $\tilde n$.
\end{proof}

\subsection{Intermediate asymptotics of the Nernst--Planck equation with Poisson term}\label{Sec:inter}

Let us come back to Equation~\eqref{Eq:(5.1)}. Let $(n_{\infty}, c_{\infty})$ be the stationary solution of~\eqref{Eq:(5.2)} given by~\eqref{st} with $d=2$ and $\phi(x)=\frac{1}{2}\,|x|^2$. The self-similar solution of~\eqref{Eq:(5.1)} is
\begin{equation}
\label{intereq01}
u_{\infty}(t,x)=\frac{1}{1+2\,t}\,n_{\infty}\left(\frac{x}{\sqrt{1+2\,t}}\right)\,,\qquad
v_{\infty}(t,x)=c_{\infty}\left(\frac{x}{\sqrt{1+2\,t}}\right)\,.
\end{equation}
Using Theorems~\ref{Thm4.1},~\ref{Thm:Sharp} and~\ref{Thm:Improved} and Proposition~\ref{cormore}, we obtain a result on the \emph{intermediate asymptotics} for the solutions of the Nernst--Planck equation with Poisson term in the absence of any external potential of confinement.
\begin{thm}\label{Thm:Inter} Assume that $u$ solves~\eqref{Eq:(5.1)}, in the sense that the function $n$ defined by~\eqref{ch2} is the solution of~\eqref{Eq:(5.2)} with $\phi(x)=\frac12\,|x|^2$ given by Proposition~\ref{Prop:A3}, with initial datum $u(0,\cdot)= n_{0}\in \mathrm L^2_{+}(n_{\infty}^{-1}dx)$, $\irtwo{n_{0}}=M$, and $\mathcal F[n_0]<\infty$. Let $(u_\infty,v_\infty)$ be the self-similar solution of mass $M$ defined by~\eqref{intereq01}. Then, as $t\to+\infty$:
\begin{enumerate}[label=(\roman*)]
\item for any $p\in[1,2]$, $\|u(t,\cdot)-u_\infty(t,\cdot)\|_{\mathrm L^p(\R^2)}= O\big((1+2\,t)^{-\frac{1}{2}-\frac{p-1}{p}}\big)$, and for any $p\in(2,\infty]$ and any $\lambda <1$, $\|u(t,\cdot)-u_\infty(t,\cdot)\|_{\mathrm L^p(\R^2)}= O\big((1+2\,t)^{-\frac{\lambda}{2}-\frac{p-1}{p}}\big)$;
\item $\|\nabla v(t,\cdot)-\nabla v_\infty(t,\cdot)\|_{\mathrm L^2(\R^2)}= O\big((1+2\,t)^{-\frac12}\big)$, and for any $q\in(2,\infty]$ and any $\lambda <1$, $\|\nabla v(t,\cdot)-\nabla v_\infty(t,\cdot)\|_{\mathrm L^q(\R^2)}= O\big((1+2\,t)^{-\frac{\lambda +1}{2}+\frac{1}{q}}\big)$;
\item if moreover $X_u:=M^{-1}\irtwo{x\,n_0}$ is such that $n_0(\cdot+X_u)\in\mathrm L^2(n_\infty^{-1}dx)$, then the estimates~(i) and~(ii) hold with $u_\infty(t,\cdot)$ and $v_\infty(t,\cdot)$ replaced by $u_\infty(t,\cdot-X_u)$ and $v_\infty(t,\cdot-X_u)$, and with the exponent $1$ replaced by $\Lambda_2\ge1+\frac12\,\delta_0$, \emph{i.e.}, with $(1+2\,t)^{-\frac12}$ replaced by $(1+2\,t)^{-\frac{\Lambda_2}2}$ and $\lambda<1$ replaced by $\lambda<\Lambda_2$.
\end{enumerate}
\end{thm}
\begin{proof}
Let $(n,c)$ be related to $(u,v)$ by~\eqref{ch2}, so that $n$ solves~\eqref{Eq:(5.2)} with $\phi=\frac12\,|x|^2$ and $n(0,\cdot)=n_0$. With $R=\sqrt{1+2\,t}$ and $\tau=\log R$, we have
\[
u(t,x)-u_\infty(t,x)=R^{-2}\,\big(n(\tau,x/R)-n_\infty(x/R)\big)\,,\qquad\nabla v(t,x)-\nabla v_\infty(t,x)=R^{-1}\,\big(\nabla c(\tau,x/R)-\nabla c_\infty(x/R)\big)\,.
\]
Hence $\nrm{u(t,\cdot)-u_\infty(t,\cdot)}{\mathrm L^p}=R^{-2+\frac2p}\,\nrm{\delta(\tau,\cdot)}{\mathrm L^p}$ and $\nrm{\nabla v(t,\cdot)-\nabla v_\infty(t,\cdot)}{\mathrm L^q}=R^{-1+\frac2q}\,\nrm{E(\tau,\cdot)}{\mathrm L^q}$, with $e^{-\lambda\tau}=(1+2\,t)^{-\lambda/2}$. Then~(i) and~(ii) follow from Proposition~\ref{cormore} with $\gamma=\mu=1$. For~(iii), let $\tilde n$ be the solution of~\eqref{Eq:(5.2)} with initial datum $n_0(\cdot+X_u)$. By Theorem~\ref{Thm:Improved} and the uniqueness result of Proposition~\ref{Prop:Uniq}, we have $n(\tau,\xi)=\tilde n(\tau,\xi-X_u\,e^{-\tau})$. The center of mass of $n(\tau,\cdot)$ is $X_u\,e^{-\tau}=X_u/R$, so that the center of mass of $u(t,\cdot)$ is $X_u$ for any $t\ge0$: it is conserved, in agreement with $\frac d{dt}\irtwo{x\,u}=-\irtwo{u\,\nabla v}=0$. We get
\[
u(t,x)-u_\infty(t,x-X_u)=R^{-2}\,\Big(\tilde n\big(\tau,\tfrac{x-X_u}R\big)-n_\infty\big(\tfrac{x-X_u}R\big)\Big)\,,
\]
and we conclude by Theorem~\ref{Thm:Improved} and Corollary~\ref{Cor:Expansion}. The same argument applies to $\nabla v$.
\end{proof}

\section{Concluding remarks}

\begin{enumerate}[label=(\arabic*)]
\item The algebra of Section~\ref{Sec:Sharp} is valid verbatim for any $d\ge2$, with $\kappa_\ell=(2\,\ell+d-2)/\ell$ and $\inf\beta_\ell\ge(d-2)\,\mu/\ell+2\,\delta_0$. In particular, the second spectral gap satisfies $\Lambda_2\ge\mu+\delta_0/d$ in any dimension $d\ge2$. We restrict our statements to $d\in\{2,3\}$ because the analytic framework used for the nonlinear problem (Section~\ref{NeSec:mi}) relies on these dimensions.
\item The chiral flux method only uses the radial symmetry of $n_\infty$ and Lemma~\ref{Lem:Monotone}. For a general convex potential $\phi$ with $D^2\phi\ge\kappa\,\mathrm{Id}$\,, it is natural to conjecture that $\Cs\ge\kappa$. This would mean that the repulsive Coulomb interaction does not decrease the spectral gap. This question is left open.
\item In the bipolar case (two species of opposite charges), the linearized problem becomes a $2\times2$ system. The extension of the flux method to this case is an interesting open problem.
\end{enumerate}

\appendix
\section{Well-posedness in \texorpdfstring{$\mathrm L^2(n_\infty^{-1}dx)$}{the weighted L2 space}}\label{App:WP}

In this appendix, we prove the existence of solutions in the sense of Definition~\ref{Def:Solution} for initial data $n_0\in\mathrm L^2_+(n_\infty^{-1}dx)$ with $\ird{n_0}=M$, and their uniqueness in the weighted energy space. We also prove a uniqueness result in the class of functions which satisfy Properties~(i) and~(ii) of Definition~\ref{Def:Solution}, without weight (Proposition~\ref{Prop:Uniq}). We work with $f=n/n_\infty-1$, and with
\[
\mathcal H:=\mathrm L^2(n_\infty\,dx)\,,\qquad\mathcal V:=\mathrm H^1(n_\infty\,dx)\,,\qquad X_T:=C\big([0,T];\mathcal H\big)\cap\mathrm L^2\big(0,T;\mathcal V\big)\,,
\]
with $\nrm f{X_T}^2:=\sup_{t\in[0,T]}\nrm{f(t)}{\mathcal H}^2+\int_0^T\nrm{\sqrt{n_\infty}\,\nabla f}{\mathrm L^2}^2\,dt$. By the computation of Section~\ref{Sec:LargeTime}, Equation~\eqref{Eq:(5.2)} is equivalent to
\be{WeakForm}
\frac d{dt}\ird{f\,g\,n_\infty}=-\ird{n_\infty\,\nabla f\cdot\nabla g}-\ird{n_\infty\,(1+f)\,\nabla G_f\cdot\nabla g}\quad\forall\,g\in\mathcal V\,,
\ee
where $G_f:=G_d*(f\,n_\infty)$. This is the weak form of $\partial_tf=\cL_0f+\mathcal B(f)$, with $\cL_0f:=n_\infty^{-1}\nabla\cdot(n_\infty\nabla f)$ and $\mathcal B(f):=n_\infty^{-1}\nabla\cdot\big(n_\infty\,(1+f)\,\nabla G_f\big)$.

\subsection{Preliminary estimates}

\begin{lem}\label{Lem:A1}
Let $d=2$ or $3$ and assume that $\phi$ satisfies~\hyperlink{C}{$(\mathrm C)$}. Let $V_\infty:=\phi+c_\infty$, so that $\nabla n_\infty=-n_\infty\,\nabla V_\infty$. Then:
\begin{enumerate}[label=(\roman*)]
\item $n_\infty\,(1+|\nabla V_\infty|^2)\in\mathrm L^\infty(\R^d)$;
\item for any $g\in\mathcal V$, $g\,n_\infty\in\mathrm H^1(\R^d)$ and $\nrm{g\,n_\infty}{\mathrm H^1(\R^d)}\le C_1\,\nrm g{\mathcal V}$;
\item for any $g\in\mathcal V$ with $\ird{g\,n_\infty}=0$, with $\vartheta=\frac{d(d-1)}{2(d+1)}$ as in Lemma~\ref{Lem:ShortTime},
\[
\nrm{\nabla G_g}{\mathrm L^\infty}\le C_2\(\nrm g{\mathcal H}+\nrm g{\mathcal H}^{1-\vartheta}\,\nrm g{\mathcal V}^\vartheta\)\,.
\]
\end{enumerate}
\end{lem}
\begin{proof}
(i) Since $D^2\phi\in\mathrm L^\infty$ and $\nabla c_\infty\in\mathrm L^\infty$ by Proposition~\ref{Prop:cinfty}, we have $|\nabla V_\infty|^2\le C\,(1+|x|^2)$. By Lemma~\ref{Lem:ninftyBounds}, $n_\infty\le C\,|x|^{-4-\eps}$ if $d=2$. If $d=3$, $n_\infty\le e^\lambda\,e^{-\phi}\le C\,|x|^{-3-\eps}$ by~\eqref{Hyp-C1}. In both cases, $n_\infty\,(1+|x|^2)$ is bounded.
(ii) We have $\nrm{g\,n_\infty}{\mathrm L^2}^2\le\nrm{n_\infty}{\mathrm L^\infty}\nrm g{\mathcal H}^2$ and $\nabla(g\,n_\infty)=n_\infty\,\nabla g-g\,n_\infty\,\nabla V_\infty$, with $\nrm{n_\infty\nabla g}{\mathrm L^2}^2\le\nrm{n_\infty}{\mathrm L^\infty}\nrm{\sqrt{n_\infty}\nabla g}{\mathrm L^2}^2$ and $\nrm{g\,n_\infty\nabla V_\infty}{\mathrm L^2}^2\le\nrm{n_\infty|\nabla V_\infty|^2}{\mathrm L^\infty}\nrm g{\mathcal H}^2$.
(iii) We apply Lemma~\ref{Lem:Field} to $\rho=g\,n_\infty$, with $\nrm{g\,n_\infty}{\mathrm L^1}\le\sqrt M\,\nrm g{\mathcal H}$, and the Gagliardo--Nirenberg inequality of Lemma~\ref{Lem:ShortTime} combined with~(ii).
\end{proof}

\begin{lem}[Chain rule]\label{Lem:ChainRule}
Let $\varpi$ be a positive continuous function on $\R^d$, $\mathsf H:=\mathrm L^2(\R^d,\varpi\,dx)$ and $\mathsf V:=\{g\in\mathsf H\,:\,\nabla g\in\mathrm L^2(\R^d,\varpi\,dx)\}$, with the pairing $\langle\cdot,\cdot\rangle$ between $\mathsf V'$ and $\mathsf V$ which extends the scalar product of $\mathsf H$. Let $w\in\mathrm L^2(0,T;\mathsf V)\cap C([0,T];\mathsf H)$ be such that $\partial_tw\in\mathrm L^2(0,T;\mathsf V')$. Let $\beta:\R\to\R$ be a Lipschitz, piecewise $C^1$ function such that $\beta(0)=0$, and $\mathsf B(z):=\int_0^z\beta(\zeta)\,d\zeta$. Then $\beta(w)\in\mathrm L^2(0,T;\mathsf V)$, and
\be{ChainRule}
\ird{\mathsf B\big(w(t)\big)\,\varpi}-\ird{\mathsf B\big(w(s)\big)\,\varpi}=\int_s^t\big\langle\partial_tw(r),\beta\big(w(r)\big)\big\rangle\,dr\quad\forall\,0\le s\le t\le T\,.
\ee
In particular, $t\mapsto\ird{\mathsf B(w(t))\,\varpi}$ is absolutely continuous on $[0,T]$.
\end{lem}
\begin{proof}
Let $L$ be the Lipschitz constant of $\beta$. Since $|\beta(z)|\le L\,|z|$ and $\nabla\beta(g)=\beta'(g)\,\nabla g$ a.e. for any $g\in\mathsf V$, we have $\nrm{\beta(g)}{\mathsf V}\le L\,\nrm g{\mathsf V}$, hence $\beta(w)\in\mathrm L^2(0,T;\mathsf V)$. Moreover, $|\mathsf B(a)-\mathsf B(b)|\le L\,(|a|+|b|)\,|a-b|$, so that $g\mapsto\ird{\mathsf B(g)\,\varpi}$ is continuous on $\mathsf H$ and both sides of~\eqref{ChainRule} are continuous functions of $(s,t)$. It is therefore enough to prove~\eqref{ChainRule} for $0<s<t<T$.

Let $(\varrho_\eta)_{\eta>0}$ be a family of mollifiers on $\R$, with support in $(-\eta,\eta)$, and let $w_\eta:=\varrho_\eta*w$ be the regularization of $w$ with respect to $t$, which is defined on $[s,t]$ if $\eta<\min\{s,T-t\}$. Then $w_\eta\in C^1([s,t];\mathsf V)$, $\partial_tw_\eta=\varrho_\eta*\partial_tw$, and, as $\eta\to0_+$, $w_\eta\to w$ in $\mathrm L^2(s,t;\mathsf V)$ and in $C([s,t];\mathsf H)$, while $\partial_tw_\eta\to\partial_tw$ in $\mathrm L^2(s,t;\mathsf V')$. Since $\mathsf B\in C^1(\R)$ with $|\mathsf B(a)-\mathsf B(b)-\beta(b)\,(a-b)|\le\frac L2\,|a-b|^2$, the function $r\mapsto\ird{\mathsf B(w_\eta(r))\,\varpi}$ is of class $C^1$ on $[s,t]$, with derivative $\ird{\beta(w_\eta)\,\partial_tw_\eta\,\varpi}=\langle\partial_tw_\eta,\beta(w_\eta)\rangle$. Hence
\[
\ird{\mathsf B\big(w_\eta(t)\big)\,\varpi}-\ird{\mathsf B\big(w_\eta(s)\big)\,\varpi}=\int_s^t\big\langle\partial_tw_\eta(r),\beta\big(w_\eta(r)\big)\big\rangle\,dr\,.
\]
The left-hand side converges to the left-hand side of~\eqref{ChainRule} because $w_\eta\to w$ in $C([s,t];\mathsf H)$. In the right-hand side, $\partial_tw_\eta$ converges strongly in $\mathrm L^2(s,t;\mathsf V')$. The functions $\beta(w_\eta)$ are bounded in $\mathrm L^2(s,t;\mathsf V)$ and converge to $\beta(w)$ in $\mathrm L^2(s,t;\mathsf H)$, because $|\beta(w_\eta)-\beta(w)|\le L\,|w_\eta-w|$. Hence $\beta(w_\eta)\rightharpoonup\beta(w)$ weakly in $\mathrm L^2(s,t;\mathsf V)$, and the right-hand side converges to the right-hand side of~\eqref{ChainRule}.
\end{proof}
With $\beta(z)=z$, Lemma~\ref{Lem:ChainRule} is the Lions--Magenes lemma~\cite[Chapter~III, Lemma~1.2]{Temam}. We shall apply it either with $w=f$ and $\varpi=n_\infty$, so that $\mathsf H=\mathcal H$ and $\mathsf V=\mathcal V$, or with $w=n$ and $\varpi=1$, so that $\mathsf H=\mathrm L^2(\R^d)$ and $\mathsf V=\mathrm H^1(\R^d)$. In the second case, $\beta$ only needs to be defined on $[0,\infty)$ because $n\ge0$, and in the first case on $[-1,\infty)$ once we know that $f\ge-1$.

\subsection{Local existence and uniqueness}

\begin{prop}\label{Prop:A2}
Let $d=2$ or $3$ and assume that $\phi$ satisfies~\hyperlink{C}{$(\mathrm C)$}. For any $f_0\in\mathcal H$ such that $\ird{f_0\,n_\infty}=0$, there exists $T>0$, which depends only on $\nrm{f_0}{\mathcal H}$, such that~\eqref{WeakForm} has a unique solution $f\in X_T$ with $f(0)=f_0$.
\end{prop}
\noindent\emph{Proof.}

\emph{Step 1: linear problem.} For a given vector field $F$ with $\sqrt{n_\infty}\,F\in\mathrm L^2\big((0,T)\times\R^d\big)$, consider
\[
\frac d{dt}\ird{f\,g\,n_\infty}+\ird{n_\infty\,\nabla f\cdot\nabla g}=-\ird{n_\infty\,F\cdot\nabla g}\quad\forall\,g\in\mathcal V\,,\qquad f(0)=f_0\,.
\]
The form $(f,g)\mapsto\ird{n_\infty\nabla f\cdot\nabla g}$ is continuous on $\mathcal V$ and satisfies a G\aa rding inequality, so J.-L.~Lions' theorem (see~\cite[Chapter~III, Theorem~1.1]{Temam}) gives a unique solution $f\in X_T$. Testing with $g=f$, which is justified by the Lions--Magenes lemma, we get $\frac12\frac d{dt}\nrm f{\mathcal H}^2+\nrm{\sqrt{n_\infty}\nabla f}{\mathrm L^2}^2\le\frac12\nrm{\sqrt{n_\infty}\nabla f}{\mathrm L^2}^2+\frac12\nrm{\sqrt{n_\infty}F}{\mathrm L^2}^2$. Hence
\be{A:linear}
\nrm f{X_T}^2\le2\,\nrm{f_0}{\mathcal H}^2+2\int_0^T\nrm{\sqrt{n_\infty}\,F}{\mathrm L^2}^2\,dt\,.
\ee
Testing with $g=1\in\mathcal V$ shows that $\ird{f(t)\,n_\infty}=\ird{f_0\,n_\infty}=0$ for any $t$.

\emph{Step 2: nonlinear estimates.} Let $F(f):=(1+f)\,\nabla G_f$. For $f\in X_T$ with zero average and $T\le1$, Lemma~\ref{Lem:GradGL2}, Lemma~\ref{Lem:A1}~(iii) and H\"older's inequality in time give
\begin{align*}
\int_0^T\nrm{\sqrt{n_\infty}\,F(f)}{\mathrm L^2}^2\,dt&\le2\,\nrm{n_\infty}{\mathrm L^\infty}C_G^2\,T\sup_{[0,T]}\nrm f{\mathcal H}^2+2\sup_{[0,T]}\nrm f{\mathcal H}^2\int_0^T\nrm{\nabla G_f}{\mathrm L^\infty}^2\,dt\\
&\le C\,T^{1-\vartheta}\(\nrm f{X_T}^2+\nrm f{X_T}^4\)\,.
\end{align*}
Here we used $\int_0^T\nrm f{\mathcal V}^2\,dt\le T\sup_{[0,T]}\nrm f{\mathcal H}^2+\nrm f{X_T}^2\le2\,\nrm f{X_T}^2$ and
\[
\int_0^T\nrm f{\mathcal H}^{2(1-\vartheta)}\nrm f{\mathcal V}^{2\vartheta}\,dt\le T^{1-\vartheta}\sup_{[0,T]}\nrm f{\mathcal H}^{2(1-\vartheta)}\(\int_0^T\nrm f{\mathcal V}^2\,dt\)^\vartheta\,.
\] Since $F(f_1)-F(f_2)=(f_1-f_2)\,\nabla G_{f_1}+(1+f_2)\,\nabla G_{f_1-f_2}$, the same estimates give
\[
\(\int_0^T\nrm{\sqrt{n_\infty}\,\big(F(f_1)-F(f_2)\big)}{\mathrm L^2}^2\,dt\)^{1/2}\le C\,T^{\frac{1-\vartheta}2}\big(1+\nrm{f_1}{X_T}+\nrm{f_2}{X_T}\big)\,\nrm{f_1-f_2}{X_T}\,.
\]

\emph{Step 3: fixed point.} Let $\Phi(f)$ be the solution of the linear problem of Step~1 with $F=F(f)$. By~\eqref{A:linear} and Step~2, $\Phi$ maps the ball $\{\nrm f{X_T}\le2\,\nrm{f_0}{\mathcal H}+1\}$ of zero-average functions into itself and is a contraction there, provided $T$ is small enough, depending only on $\nrm{f_0}{\mathcal H}$. Banach's fixed point theorem gives existence. Uniqueness in $X_T$ follows from the same Lipschitz estimate. \hfill$\square$

\subsection{Global existence and properties of the solution}

\begin{prop}\label{Prop:A3}
Let $d=2$ or $3$ and assume that $\phi$ satisfies~\hyperlink{C}{$(\mathrm C)$}. For any $n_0\in\mathrm L^2_+(n_\infty^{-1}dx)$ such that $\ird{n_0}=M$, there exists a solution $n$ of~\eqref{Eq:(5.2)} in the sense of Definition~\ref{Def:Solution}, and $\mathcal F[n_0]<\infty$. Moreover, $f=n/n_\infty-1$ is the unique solution of~\eqref{WeakForm} in $X_T$ with $f(0)=n_0/n_\infty-1$, for any $T>0$.
\end{prop}
\begin{proof}
Let $f$ be the solution of Proposition~\ref{Prop:A2} on $[0,T]$ with $f_0=n_0/n_\infty-1$, and let $n:=n_\infty\,(1+f)$ and $c:=c_\infty+G_f$. We repeatedly use the chain rule of Lemma~\ref{Lem:ChainRule}.

\emph{Step 4: nonnegativity.} Let $w:=(1+f)_-=n_-/n_\infty\in\mathrm L^2(0,T;\mathcal V)$. By the chain rule applied with $\beta(z)=-(1+z)_-$, we may test~\eqref{WeakForm} with $g=-w$. Since $\nabla f=-\nabla w$ and $1+f=-w$ on the support of $w$, we obtain
\[
\frac12\,\frac d{dt}\nrm w{\mathcal H}^2=-\nrm{\sqrt{n_\infty}\nabla w}{\mathrm L^2}^2-\ird{n_\infty\,w\,\nabla G_f\cdot\nabla w}\le-\frac12\nrm{\sqrt{n_\infty}\nabla w}{\mathrm L^2}^2+\frac12\,\nrm{\nabla G_f}{\mathrm L^\infty}^2\nrm w{\mathcal H}^2\,.
\]
Since $w(0)=0$ and $\nrm{\nabla G_f}{\mathrm L^\infty}^2\in\mathrm L^1(0,T)$ by Step~2, Gr\"onwall's lemma gives $w\equiv0$, that is, $n\ge0$.

\emph{Step 5: a priori estimate and global existence.} By Lemma~\ref{Lem:A1}~(ii) applied to $1+f$, we have $n\in C\big([0,T];\mathrm L^2(\R^d)\big)\cap\mathrm L^2\big(0,T;\mathrm H^1(\R^d)\big)$. Moreover, $n^2\,|\nabla\phi|^2\le\nrm{n_\infty|\nabla\phi|^2}{\mathrm L^\infty}\,n_\infty\,(1+f)^2$, where $n_\infty\,|\nabla\phi|^2$ is bounded as in the proof of Lemma~\ref{Lem:A1}~(i), and $\nabla c=\nabla c_\infty+\nabla G_f\in\mathrm L^2(0,T;\mathrm L^\infty(\R^d))$ by Step~2. Hence
\[
J:=\nabla n+n\,\nabla c+n\,\nabla\phi\in\mathrm L^2\big((0,T)\times\R^d\big)\,,\qquad\partial_tn=\nabla\cdot J\in\mathrm L^2\big(0,T;\mathrm H^{-1}(\R^d)\big)\,.
\]
We test this equation with $n$. For a.e. $t\in(0,T)$, $n^2\in\mathrm W^{1,1}(\R^d)$, $\nabla c\in\mathrm L^\infty(\R^d)$ and $-\Delta c=n$. With the cut-off function $\chi_R$ of the proof of Lemma~\ref{Lem:ShortTime},
\begin{align*}
\ird{\chi_R\,n\,\nabla n\cdot\nabla c}&=\frac12\ird{\chi_R\,n^3}-\frac12\ird{n^2\,\nabla c\cdot\nabla\chi_R}\,,\\
\ird{\chi_R\,n\,\nabla n\cdot\nabla\phi}&=-\frac12\ird{\chi_R\,n^2\,\Delta\phi}-\frac12\ird{n^2\,\nabla\phi\cdot\nabla\chi_R}\,.
\end{align*}
As $R\to+\infty$, the last term of each identity converges to $0$, because $|\nabla c|\,|\nabla\chi_R|\le C/R$, $|\nabla\phi|\,|\nabla\chi_R|\le A\,\ind_{|x|\ge R}$ and $n^2\in\mathrm L^1(\R^d)$. By monotone convergence, $n^3\in\mathrm L^1(\R^d)$ and
\[
\frac d{dt}\nrm n{\mathrm L^2}^2=-2\,\nrm{\nabla n}{\mathrm L^2}^2-\ird{n^3}+\ird{n^2\,\Delta\phi}\le-2\,\nrm{\nabla n}{\mathrm L^2}^2+C_\phi\,\nrm n{\mathrm L^2}^2\,.
\]
Hence~\eqref{ShortTime1} holds on $[0,T]$. By the Gagliardo--Nirenberg inequality and Lemma~\ref{Lem:Field}, exactly as in the proof of Lemma~\ref{Lem:ShortTime}, $\int_0^T\nrm{\nabla G_f}{\mathrm L^\infty}^2\,dt\le K(T)$, where $K(T)$ depends only on $T$, $M$, $\nrm{n_0}{\mathrm L^2}$ and $n_\infty$, and is a nondecreasing function of $T$. Testing~\eqref{WeakForm} with $g=f$ and using Lemma~\ref{Lem:GradGL2}, we get
\[
\frac d{dt}\nrm f{\mathcal H}^2+\nrm{\sqrt{n_\infty}\nabla f}{\mathrm L^2}^2\le\(2\,\nrm{n_\infty}{\mathrm L^\infty}C_G^2+2\,\nrm{\nabla G_f}{\mathrm L^\infty}^2\)\nrm f{\mathcal H}^2\,,
\]
so that $\nrm{f(t)}{\mathcal H}^2\le\nrm{f_0}{\mathcal H}^2\,\exp\big(2\,\nrm{n_\infty}{\mathrm L^\infty}C_G^2\,t+2\,K(t)\big)$. Since the existence time of Proposition~\ref{Prop:A2} depends only on the $\mathcal H$-norm of the initial datum, the solution can be continued, and it is global. Uniqueness in $X_T$ for any $T>0$ follows from Proposition~\ref{Prop:A2}. This proves Property~(iv). The continuity $n\in C\big([0,\infty);\mathrm L^1\cap\mathrm L^2(\R^d)\big)$ follows from $\nrm{g\,n_\infty}{\mathrm L^1}\le\sqrt M\nrm g{\mathcal H}$ and $\nrm{g\,n_\infty}{\mathrm L^2}\le\nrm{n_\infty}{\mathrm L^\infty}^{1/2}\nrm g{\mathcal H}$, and we have also shown that $\nabla n\in\mathrm L^2\big((0,T)\times\R^d\big)$. The mass is conserved, by testing~\eqref{WeakForm} with $g=1$.

\emph{Step 6: $\mathrm L^p$ and $\mathrm L^\infty$ bounds for $t>0$.} Let $p\ge2$ be an integer, $k>0$, and
\[
b_k(z):=\min\{z,k\}^{p-1}\,,\qquad B_k(z):=\int_0^zb_k(\zeta)\,d\zeta\,,\qquad\Gamma_k(z):=\int_0^zb_k'(\zeta)\,\zeta\,d\zeta=\frac{p-1}p\,\min\{z,k\}^p
\]
for $z\ge0$. Then $b_k$ is Lipschitz, $b_k(0)=0$, $\Gamma_k\le(p-1)\,B_k$, $B_k(z)\le z^p/p$ and $B_k(z)\uparrow z^p/p$ as $k\to+\infty$. By the chain rule, for a.e. $t\in(0,T)$,
\[
\frac d{dt}\ird{B_k(n)}=-\ird{b_k'(n)\,\nabla n\cdot J}=-\ird{b_k'(n)\,|\nabla n|^2}-\ird{\nabla\Gamma_k(n)\cdot(\nabla c+\nabla\phi)}\,.
\]
Since $\Gamma_k(n)\le C_k\,n^2$ and $|\nabla\Gamma_k(n)|\le C_k\,n\,|\nabla n|$, we have $\Gamma_k(n)\in\mathrm W^{1,1}\cap\mathrm L^\infty(\R^d)$ for a.e. $t$, and the same cut-off argument as in Step~5 gives $\ird{\nabla\Gamma_k(n)\cdot\nabla c}=\ird{\Gamma_k(n)\,n}$ and $\ird{\nabla\Gamma_k(n)\cdot\nabla\phi}=-\ird{\Gamma_k(n)\,\Delta\phi}\ge-C_\phi\ird{\Gamma_k(n)}$. Hence
\be{A:trunc}
\frac d{dt}\ird{B_k(n)}+\ird{\Gamma_k(n)\,n}\le C_\phi\ird{\Gamma_k(n)}\le(p-1)\,C_\phi\ird{B_k(n)}\,.
\ee
By Gr\"onwall's lemma, for any $0\le s\le t\le T$,
\be{A:Lp}
\ird{B_k(n(t))}+\int_s^t\!\!\ird{\Gamma_k(n)\,n}\,dr\le e^{(p-1)C_\phi(t-s)}\ird{B_k(n(s))}\le\frac1p\,e^{(p-1)C_\phi(t-s)}\,\nrm{n(s)}{\mathrm L^p}^p\,.
\ee
We claim that, for any integer $p\ge2$ and any $\tau\in(0,T)$,
\[
n\in\mathrm L^\infty\big(\tau,T;\mathrm L^p(\R^d)\big)\cap\mathrm L^{p+1}\big((\tau,T)\times\R^d\big)\,.
\]
If $p=2$, we take $s=0$ in~\eqref{A:Lp} and let $k\to+\infty$ by monotone convergence. If the claim holds for $p-1\ge2$, then $n(s,\cdot)\in\mathrm L^p(\R^d)$ for a.e. $s\in(0,T)$. We choose such an $s<\tau$ and let $k\to+\infty$ in~\eqref{A:Lp}, which proves the claim for $p$. As a consequence, by interpolation with $n\in C([0,T];\mathrm L^2(\R^d))$, we have $n\in C\big([\tau,T];\mathrm L^p(\R^d)\big)$ for any $p\in[2,\infty)$ and any $\tau\in(0,T)$. Integrating the first inequality in~\eqref{A:trunc} on $(s,t)$ and letting $k\to+\infty$, we obtain that $z(t):=\ird{n(t,\cdot)^p}$ satisfies~\eqref{IntIneq} with $C=C_\phi$ for any $\tau\le s\le t\le T$. The arguments which follow~\eqref{IntIneq} in the proof of Lemma~\ref{lem2.1} (end of Step~3, Steps~4 and~5) apply without change and give
\[
\nrm{n(t,\cdot)}{\mathrm L^p(\R^d)}\le M^{\frac1p}\,\max\Big\{2\,C_\phi\,,\;\frac2{t-\tau}\Big\}^{\frac{p-1}p}\quad\forall\,t\in(\tau,T]\,.
\]
Letting $p\to+\infty$ and $\tau\to0_+$, we conclude that $\nrm{n(t,\cdot)}{\mathrm L^\infty(\R^d)}\le\max\{2\,C_\phi,2/t\}$ for any $t\in(0,T]$. With Step~5, this proves Property~(ii).

\emph{Step 7: regularity for $t>0$.} On $(\tau,T)\times\R^d$, $n$ is a bounded weak solution of $\partial_tn=\nabla\cdot\big(\nabla n+n\,b\big)$ with $b=\nabla c+\nabla\phi$, where $\nabla c\in\mathrm L^\infty\big((\tau,T)\times\R^d\big)$ by Lemma~\ref{Lem:Field} and $\nabla\phi$ is locally Lipschitz. By the De Giorgi--Nash--Moser theory, $n$ is locally H\"older continuous, hence $\nabla c$ is locally H\"older continuous in $x$, locally uniformly in $t$. Interior parabolic $\mathrm L^p$ estimates~\cite{LSU} then show that $\partial_tn$ and $D^2n$ belong to $\mathrm L^p_{\rm loc}$ for any $p<\infty$, and $\nabla n$ is continuous by the parabolic Sobolev embeddings. This is Property~(i), and it is enough to justify all the integrations by parts of this paper. Positivity for $t>0$ follows from the strong maximum principle, as noted after Definition~\ref{Def:Solution}.

\emph{Step 8: free energy inequality.} By~\eqref{Fdiff} with $\bar n=n_\infty$ and Lemma~\ref{Lem:GradGL2},
\be{A:FreeEnergy}
\mathcal F[n]-\mathcal F[n_\infty]=\ird{n_\infty\,\Psi(f)}+\frac12\ird{|\nabla G_f|^2}\,,\qquad\Psi(z):=(1+z)\log(1+z)-z\,,
\ee
where we used $\ird{f\,n_\infty}=0$. Since $0\le\Psi(z)\le z^2$ for $z\ge-1$, the right-hand side is finite for any $f\in\mathcal H$ such that $f\ge-1$. In particular, $\mathcal F[n_0]<\infty$. For $\eps>0$, let
\[
\Psi_\eps(z):=(1+\eps)\,\Psi\(\frac z{1+\eps}\)\,,\qquad\Psi_\eps'(z)=\log\(\frac{1+z+\eps}{1+\eps}\)\,,
\]
so that $0\le\Psi_\eps(z)\le z^2$, and $\Psi_\eps'$ is Lipschitz on $[-1,\infty)$ with $\Psi_\eps'(0)=0$. Let $\rho:=1+f\ge0$. By the chain rule and~\eqref{WeakForm} with $g=\Psi_\eps'(f)$, $\nabla g=\nabla f/(\rho+\eps)$, we have, for a.e. $t$,
\[
\frac d{dt}\ird{n_\infty\,\Psi_\eps(f)}=-\ird{n_\infty\,\frac{|\nabla f|^2}{\rho+\eps}}-\ird{n_\infty\,\frac\rho{\rho+\eps}\,\nabla G_f\cdot\nabla f}\,.
\]
As in Step~3 of the proof of Lemma~\ref{Lem:EnergyIdentity}, which only uses $J\in\mathrm L^2((0,T)\times\R^d)$ and $J=n_\infty\,\nabla f+n_\infty\,\rho\,\nabla G_f$,
\[
\frac12\,\frac d{dt}\ird{|\nabla G_f|^2}=-\ird{n_\infty\,\nabla f\cdot\nabla G_f}-\ird{n_\infty\,\rho\,|\nabla G_f|^2}\,.
\]
Let
\[
\mathcal E_\eps(t):=\ird{n_\infty\,\Psi_\eps(f(t,\cdot))}+\frac12\ird{|\nabla G_f(t,\cdot)|^2}\,.
\]
Adding these two identities and completing the square, we obtain that, for any $0\le s\le t\le T$,
\begin{multline}\label{A:eps}
\mathcal E_\eps(t)+\int_s^t\!\!\ird{n_\infty\,(\rho+\eps)\,\left|\frac{\nabla f}{\rho+\eps}+\nabla G_f\right|^2}\,dr\\
=\mathcal E_\eps(s)+\int_s^t\!\!\ird{n_\infty\(\frac\eps{\rho+\eps}\,\nabla f\cdot\nabla G_f+\eps\,|\nabla G_f|^2\)}\,dr\,.
\end{multline}
We let $\eps\to0_+$. Since $\Psi_\eps\to\Psi$ pointwise and $\Psi_\eps(z)\le z^2$, $\mathcal E_\eps(t)$ converges to $\mathcal F[n(t,\cdot)]-\mathcal F[n_\infty]$ for any $t\in[0,T]$ by dominated convergence and~\eqref{A:FreeEnergy}. The last integral in~\eqref{A:eps} converges to $0$: indeed $n_\infty\,|\nabla f|\,|\nabla G_f|$ and $n_\infty\,|\nabla G_f|^2$ are integrable on $(0,T)\times\R^d$, $\eps/(\rho+\eps)\le1$ converges to $0$ a.e. on $\{\rho>0\}$, and $\nabla f=0$ a.e. on $\{\rho=0\}$. Finally, since $n=n_\infty\,\rho>0$ for $t>0$ and $\log n+c+\phi=\lambda+\log\rho+G_f$, we have $\mathcal I[n]=\ird{n_\infty\,\rho\,|\nabla f/\rho+\nabla G_f|^2}$, and Fatou's lemma gives
\[
\int_s^t\mathcal I[n(r,\cdot)]\,dr\le\liminf_{\eps\to0_+}\int_s^t\!\!\ird{n_\infty\,(\rho+\eps)\,\left|\frac{\nabla f}{\rho+\eps}+\nabla G_f\right|^2}\,dr\,.
\]
Altogether, $\mathcal F[n(t,\cdot)]+\int_s^t\mathcal I[n(r,\cdot)]\,dr\le\mathcal F[n(s,\cdot)]$ for any $0\le s\le t$, which is Property~(iii).
\end{proof}

\begin{remark}
The proof of Proposition~\ref{Prop:A3} is self-contained up to two classical ingredients: J.-L.~Lions' theorem for linear parabolic problems~\cite{Temam}, which is used in Step~1, and the interior regularity theory for linear parabolic equations with bounded coefficients~\cite{LSU}, which is used only in Step~7.
\end{remark}

\subsection{Uniqueness without weight}\label{App:Uniq}

In Proposition~\ref{Prop:A3}, uniqueness is proved in the weighted energy space $X_T$. Since $\mathrm L^2(n_\infty^{-1}dx)$ is not invariant under translations, we also need a uniqueness result in a class which is invariant under translations (see Remark~\ref{Rem:Translation}). We say that $n$ belongs to the class $\mathcal S(n_0)$ if $n$ is a nonnegative function in $C\big([0,\infty);\mathrm L^1\cap\mathrm L^2(\R^d)\big)$ such that $n(0,\cdot)=n_0$, which satisfies Properties~(i) and~(ii) of Definition~\ref{Def:Solution}. Lemmas~\ref{Lem:Field} and~\ref{Lem:ShortTime} apply to any $n\in\mathcal S(n_0)$, because their proofs use neither Property~(iii) nor Property~(iv), and the conservation of mass is established in Lemma~\ref{Lem:Moment} below using only Lemma~\ref{Lem:Field}. In particular, with $c=G_d*n$,
\be{A:gradc}
\nabla n\in\mathrm L^2\big((0,T)\times\R^d\big)\quad\mbox{and}\quad\int_0^T\nrm{\nabla c(t,\cdot)}{\mathrm L^\infty(\R^d)}^2\,dt<\infty\quad\forall\,T>0\,,
\ee
where we also used $\nabla c_\infty\in\mathrm L^\infty(\R^d)$ (Proposition~\ref{Prop:cinfty}).
\begin{lem}\label{Lem:Moment}
Let $d=2$ or $3$ and assume that $\phi$ satisfies~\hyperlink{C}{$(\mathrm C)$}. Let $n\in\mathcal S(n_0)$. Then $\ird{n(t,x)}=\ird{n_0}$ for any $t\ge0$. If moreover $\ird{|x|^2\,n_0}<\infty$, then
\[
\sup_{t\in[0,T]}\ird{|x|^2\,n(t,x)}<\infty\quad\forall\,T>0\,.
\]
\end{lem}
\begin{proof}
Let $M:=\ird{n_0}$ and $0<\tau<T$. Since $n$ is a strong solution of $\partial_tn=\nabla\cdot J$ with $J=\nabla n+n\,\nabla c+n\,\nabla\phi$ on $(0,\infty)\times\R^d$, we have, for any $\zeta\in C_c^\infty(\R^d)$ and any $t\in[\tau,T]$,
\be{A:weak}
\frac d{dt}\ird{n\,\zeta}=\ird{n\,\Delta\zeta}-\ird{n\,(\nabla c+\nabla\phi)\cdot\nabla\zeta}\,.
\ee
\emph{Conservation of mass.} We take $\zeta=\chi_R$, with $\chi_R$ as in the proof of Lemma~\ref{Lem:ShortTime}. The right-hand side of~\eqref{A:weak} is bounded by
\[
\frac{C_\chi}{R^2}\,\nrm{n(t)}{\mathrm L^1}+\frac{C_\chi}R\,\nrm{\nabla c(t)}{\mathrm L^\infty}\,\nrm{n(t)}{\mathrm L^1}+A\int_{|x|\ge R}n(t,x)\,dx\,.
\]
Since $n\in C([\tau,T];\mathrm L^1(\R^d))$ and $\nabla c$ is bounded on $[\tau,T]\times\R^d$ by Lemma~\ref{Lem:Field}, the integral on $[\tau,t]$ of this quantity converges to $0$ as $R\to+\infty$. Hence $\ird{n(t)}=\ird{n(\tau)}$, and we let $\tau\to0_+$.

\emph{Second moment.} Let $\Theta\in C^\infty([0,\infty))$ be a nondecreasing function such that $\Theta(s)=s$ if $s\le1$, $\Theta(s)\le s$ for any $s\ge0$, and $\Theta$ is constant on $[4,\infty)$. There is a constant $C_\Theta$ such that
\be{A:Theta}
0\le\Theta'(s)\le C_\Theta\,,\qquad s\,|\Theta''(s)|\le C_\Theta\,,\qquad s\,\Theta'(s)\le C_\Theta\,\Theta(s)\quad\forall\,s\ge0\,,
\ee
because $s\,\Theta'(s)=\Theta(s)$ if $s\le1$, and $\Theta(s)\ge1$ if $s\ge1$. Let $\theta_R(x):=R^2\,\Theta(|x|^2/R^2)$. Then $\theta_R-R^2\,\Theta(4)\in C_c^\infty(\R^d)$, $\theta_R(x)\le|x|^2$, and $\theta_R(x)\to|x|^2$ as $R\to+\infty$. With $s=|x|^2/R^2$, we have $\nabla\theta_R=2\,x\,\Theta'(s)$ and $\Delta\theta_R=2\,d\,\Theta'(s)+4\,s\,\Theta''(s)$. By~\eqref{A:Theta}, $|\Delta\theta_R|\le(2\,d+4)\,C_\Theta$ and
\[
|\nabla\theta_R|^2=4\,R^2\,s\,\Theta'(s)^2\le4\,C_\Theta^2\,\theta_R\,.
\]
Since $D^2\phi$ is bounded, there is a constant $C_0$ such that $|\nabla\phi(x)|\le C_0\,(1+|x|)$, hence
\[
|\nabla\phi\cdot\nabla\theta_R|\le2\,C_0\,(|x|+|x|^2)\,\Theta'(s)\le C_0\,(1+3\,|x|^2)\,\Theta'(s)\le C_0\,C_\Theta\,(1+3\,\theta_R)\,.
\]
Let $m_R(t):=\ird{n(t,x)\,\theta_R(x)}$. By~\eqref{A:weak} applied to $\zeta=\theta_R-R^2\,\Theta(4)$, the conservation of mass and the Cauchy--Schwarz inequality, there is a constant $C$, which does not depend on $R$, such that
\[
\frac{dm_R}{dt}\le C\,M+2\,C_\Theta\,\nrm{\nabla c}{\mathrm L^\infty}\,\sqrt{M\,m_R}+C\,(M+m_R)\le C\(1+\nrm{\nabla c}{\mathrm L^\infty}^2\)\(1+M+m_R\)\quad\mbox{on}\quad[\tau,T]\,.
\]
By Gr\"onwall's lemma,
\[
1+M+m_R(t)\le\big(1+M+m_R(\tau)\big)\,\exp\(C\int_\tau^t\big(1+\nrm{\nabla c(r)}{\mathrm L^\infty}^2\big)\,dr\)\quad\forall\,t\in[\tau,T]\,.
\]
Since $\theta_R$ is bounded and $n\in C([0,T];\mathrm L^1(\R^d))$, we have $m_R(\tau)\to\ird{n_0\,\theta_R}\le\ird{|x|^2\,n_0}$ as $\tau\to0_+$. We conclude by~\eqref{A:gradc}, and by Fatou's lemma as $R\to+\infty$.
\end{proof}

\begin{prop}\label{Prop:Uniq}
Let $d=2$ or $3$ and assume that $\phi$ satisfies~\hyperlink{C}{$(\mathrm C)$}. Let $n_0\in\mathrm L^1_+\cap\mathrm L^2(\R^d)$ be such that $\ird{|x|^2\,n_0}<\infty$. If $n_1$ and $n_2$ belong to $\mathcal S(n_0)$, then $n_1=n_2$.
\end{prop}
\begin{proof}
Let $c_i:=G_d*n_i$, $w:=n_1-n_2$ and $D:=c_1-c_2=G_d*w$, so that $-\Delta D=w$. We estimate $w$ in $\dot{\mathrm H}^{-1}(\R^d)$, that is, we estimate
\[
y(t):=\ird{|\nabla D(t,x)|^2}\,.
\]
Let $T>0$. By Lemma~\ref{Lem:Moment}, $\ird{w(t,x)}=0$ and $\mathsf m:=\sup_{t\in[0,T]}\ird{|x|^2\,(n_1+n_2)(t,x)}<\infty$.

\emph{Step 1: $y$ is finite and $y(t)\to0$ as $t\to0_+$.} The Fourier computation of the proof of Lemma~\ref{Lem:GradGL2} applies to any function $h\in\mathrm L^1\cap\mathrm L^2(\R^d)$, with zero integral and finite first moment if $d=2$. It gives
\begin{align*}
&\ird{|\nabla G_d*h|^2}\le\frac1{2\,\pi^2}\,\nrm h{\mathrm L^1(\R^3)}^2+\nrm h{\mathrm L^2(\R^3)}^2\quad\mbox{if}\quad d=3\,,\\
&\irtwo{|\nabla G_2*h|^2}\le\frac1{4\,\pi}\(\irtwo{|x|\,|h|}\)^2+\nrm h{\mathrm L^2(\R^2)}^2\quad\mbox{if}\quad d=2\,.
\end{align*}
With $h=w(t,\cdot)$ and $\big(\ird{|x|\,|w|}\big)^2\le\nrm w{\mathrm L^1}\ird{|x|^2\,|w|}\le\mathsf m\,\nrm w{\mathrm L^1}$, we obtain
\[
y(t)\le\frac1{2\,\pi^2}\,\nrm{w(t)}{\mathrm L^1}^2+\frac{\mathsf m}{4\,\pi}\,\nrm{w(t)}{\mathrm L^1}+\nrm{w(t)}{\mathrm L^2}^2\quad\forall\,t\in[0,T]\,.
\]
Since $n_1$ and $n_2$ are continuous with values in $\mathrm L^1\cap\mathrm L^2(\R^d)$ and $w(0,\cdot)=0$, $y$ is bounded on $[0,T]$ and $\lim_{t\to0_+}y(t)=0$.

\emph{Step 2: evolution of $y$.} Let $\tau\in(0,T)$. The function $w$ solves $\partial_tw=\nabla\cdot\mathsf K$ on $(0,\infty)\times\R^d$, with
\[
\mathsf K:=\nabla w+n_1\,\nabla c_1-n_2\,\nabla c_2+w\,\nabla\phi=\nabla w+w\,\nabla c_1+n_2\,\nabla D+w\,\nabla\phi\,.
\]
We claim that $\mathsf K\in\mathrm L^2\big((\tau,T)\times\R^d\big)$. By Property~(ii), $\nabla w\in\mathrm L^2\big((\tau,T)\times\R^d\big)$, and $n_1$, $n_2$, $\nabla c_1$ are bounded on $(\tau,T)\times\R^d$, so that $w\,\nabla c_1$ and $n_2\,\nabla D$ are bounded in $\mathrm L^2(\R^d)$, uniformly in $t\in[\tau,T]$, by Step~1. Finally, $|\nabla\phi(x)|^2\le2\,C_0^2\,(1+|x|^2)$ and
\[
\ird{w^2\,(1+|x|^2)}\le\nrm w{\mathrm L^\infty((\tau,T)\times\R^d)}\,\big(2\,M+\mathsf m\big)\,,\qquad M:=\ird{n_0}\,.
\]
As in Step~3 of the proof of Lemma~\ref{Lem:EnergyIdentity}, we deduce that $\partial_t\nabla D=-\mathcal P\,\mathsf K$, where $\mathcal P$ is the $\mathrm L^2$-orthogonal projection onto gradient fields. Hence $y$ is absolutely continuous on $[\tau,T]$ and, for a.e. $t\in(\tau,T)$,
\be{A:yprime}
\frac12\,\frac{dy}{dt}=-\ird{\nabla D\cdot\mathsf K}=-\ird{\nabla D\cdot\nabla w}-\ird{w\,\nabla D\cdot\nabla c_1}-\ird{n_2\,|\nabla D|^2}-\ird{w\,\nabla D\cdot\nabla\phi}\,.
\ee

\emph{Step 3: estimate of the right-hand side of~\eqref{A:yprime}.} By Plancherel's formula, $\ird{\nabla D\cdot\nabla w}=\ird{w^2}$. The third term is nonpositive because $n_2\ge0$. By the Cauchy--Schwarz inequality,
\[
\left|\ird{w\,\nabla D\cdot\nabla c_1}\right|\le\frac12\ird{w^2}+\frac12\,\nrm{\nabla c_1}{\mathrm L^\infty}^2\,y\,.
\]
For the last term, we use $w=-\Delta D$, $D\in\mathrm H^2_{\rm loc}(\R^d)$, and the pointwise identity $\Delta D\,\partial_jD=\sum_i\partial_i(\partial_iD\,\partial_jD)-\frac12\,\partial_j|\nabla D|^2$. With the cut-off function $\chi_R$, two integrations by parts give
\begin{multline*}
-\ird{\chi_R\,w\,\nabla D\cdot\nabla\phi}=\ird{\chi_R\,\Delta D\,\nabla D\cdot\nabla\phi}=-\ird{\chi_R\,D^2\phi\,\nabla D\cdot\nabla D}+\frac12\ird{\chi_R\,|\nabla D|^2\,\Delta\phi}\\
-\ird{(\nabla\chi_R\cdot\nabla D)\,(\nabla\phi\cdot\nabla D)}+\frac12\ird{|\nabla D|^2\,\nabla\chi_R\cdot\nabla\phi}\,.
\end{multline*}
Since $|\nabla\phi|\,|\nabla\chi_R|\le A\,\ind_{|x|\ge R}$ and $\nabla D\in\mathrm L^2(\R^d)$, the last two terms converge to $0$ as $R\to+\infty$. The left-hand side converges by dominated convergence, because $w\,\nabla\phi$ and $\nabla D$ belong to $\mathrm L^2(\R^d)$. With $\Delta\phi\le C_\phi$, we obtain
\[
-\ird{w\,\nabla D\cdot\nabla\phi}\le\(\nrm{D^2\phi}{\mathrm L^\infty}+\frac{C_\phi}2\)y\,.
\]
Altogether, for a.e. $t\in(\tau,T)$,
\[
\frac{dy}{dt}+\ird{w^2}\le\(\nrm{\nabla c_1(t,\cdot)}{\mathrm L^\infty}^2+2\,\nrm{D^2\phi}{\mathrm L^\infty}+C_\phi\)y\,.
\]

\emph{Step 4: conclusion.} By Gr\"onwall's lemma,
\[
y(t)\le y(\tau)\,\exp\(\int_0^T\(\nrm{\nabla c_1(r,\cdot)}{\mathrm L^\infty}^2+2\,\nrm{D^2\phi}{\mathrm L^\infty}+C_\phi\)dr\)\quad\forall\,t\in[\tau,T]\,,
\]
where the integral is finite by~\eqref{A:gradc}. By Step~1, we may let $\tau\to0_+$ and obtain $y\equiv0$ on $[0,T]$. Hence $\nabla D=0$ and $w=-\Delta D=0$.
\end{proof}
Notice that, for the harmonic potential, the last term of~\eqref{A:yprime} is explicit: $-\ird{w\,\nabla D\cdot\nabla\phi}=\frac{d-2}2\,\mu\,y$.

\bigskip{\bf Acknowledgments.} This work has been supported by the Project EFI ANR-17-CE40-0030 of the French National Research Agency.

\end{document}